\documentclass[11pt,oneside]{amsart}
\usepackage[T1]{fontenc}
\usepackage[latin9]{inputenc}
\usepackage{geometry}
\usepackage{verbatim}
\usepackage{textcomp}
\usepackage{amsthm}
\usepackage{amstext}
\usepackage{amssymb}
\usepackage{esint}
\usepackage{mathrsfs}
\usepackage{wasysym}
\usepackage{nomencl}
\usepackage[all]{xy}
\usepackage{axodraw}
\def\sameau{\rule[0.017in]{0.2in}{0.012in}}
\usepackage{graphicx} 
\usepackage{cancel} 

\usepackage{tikz}

\makeatletter
\numberwithin{equation}{section}
\numberwithin{figure}{section}
\theoremstyle{plain}
\newtheorem{thm}{\protect\theoremname}
  \theoremstyle{definition}
  \newtheorem{defn}[thm]{\protect\definitionname}
  \theoremstyle{plain}
  \newtheorem{prop}[thm]{\protect\propositionname}
  \theoremstyle{plain}
  \newtheorem{lem}[thm]{\protect\lemmaname}
  \theoremstyle{remark}
  \newtheorem{rem}[thm]{\protect\remarkname}
  \theoremstyle{definition}
  \newtheorem{example}[thm]{\protect\examplename}

\def\id{{\bf id}}

\makeatother

\usepackage[english]{babel}
  \providecommand{\definitionname}{Definition}
  \providecommand{\examplename}{Example}
  \providecommand{\lemmaname}{Lemma}
  \providecommand{\propositionname}{Proposition}
  \providecommand{\remarkname}{Remark}
\providecommand{\theoremname}{Theorem}

\def\et{
 \begin{tikzpicture}
 \draw[thick] (0cm,0cm) circle(0.4cm);
                \filldraw[black] (90:0.4cm) (0.08,.48)rectangle(-.08,.34);
                		\draw (75:0.60cm) node[above=-9pt]{$\hspace{-0.2cm}\phantom{l}^{i}$};
 \end{tikzpicture}}

\def\etn{
 \begin{tikzpicture}
 \draw[thick] (0cm,0cm) circle(0.4cm);
                \filldraw[black] (90:0.4cm) (0.08,.48)rectangle(-.08,.34);
                		\draw (75:0.60cm) node[above=-9pt]{$\hspace{-0.0cm}\phantom{l}^{n_1}$};
 \end{tikzpicture}}

\def\etk{
 \begin{tikzpicture}
 \draw[thick] (0cm,0cm) circle(0.4cm);
                \filldraw[black] (90:0.4cm) (0.08,.48)rectangle(-.08,.34);
                		\draw (75:0.60cm) node[above=-9pt]{$\hspace{-0.0cm}\phantom{l}^{n_2}$};
 \end{tikzpicture}}

\def\bt{
 \begin{tikzpicture}
 \draw[thick] (0cm,0cm) circle(0.4cm);
                \filldraw[black] (90:0.4cm) (0.08,.48)rectangle(-.08,.34);
                                \draw (75:0.60cm) node[above=-9pt]{$\hspace{-0.2cm}\phantom{l}^{i}$};
	 \filldraw[black] (-90:0.4cm) circle(1.7pt);
 \end{tikzpicture}}

\def\btn{
 \begin{tikzpicture}
 \draw[thick] (0cm,0cm) circle(0.4cm);
                \filldraw[black] (90:0.4cm) (0.08,.48)rectangle(-.08,.34);
                                \draw (75:0.60cm) node[above=-9pt]{$\hspace{-0.2cm}\phantom{l}^{i}$};
	 \filldraw[black] (-90:0.4cm) circle(1.7pt);
                               \draw (72:0.6cm) node[above=-42pt]{$\hspace{-0.2cm}\phantom{l}^{n_1}$};
 \end{tikzpicture}}

\def\bktn{
 \begin{tikzpicture}
 \draw[thick] (0cm,0cm) circle(0.4cm);
                \filldraw[black] (90:0.4cm) (0.08,.48)rectangle(-.08,.34);
                                \draw (75:0.60cm) node[above=-9pt]{$\hspace{-0.0cm}\phantom{l}^{n_1}$};
	 \filldraw[black] (-90:0.4cm) circle(1.7pt);
                                \draw (75:0.6cm) node[above=-42pt]{$\hspace{-0.2cm}\phantom{l}^{n_2}$};
 \end{tikzpicture}}

\def\nbt{
 \begin{tikzpicture}
 \draw[thick] (0cm,0cm) circle(0.4cm);
                \filldraw[black] (90:0.4cm) (0.08,.48)rectangle(-.08,.34);
                                \draw (75:0.60cm) node[above=-9pt]{$\hspace{-0.0cm}\phantom{l}^{n_1}$};
	 \filldraw[black] (-90:0.4cm) circle(1.7pt);
\end{tikzpicture}}

\def\bbt{
 \begin{tikzpicture}
 \draw[thick] (0cm,0cm) circle(0.4cm);
                \filldraw[black] (90:0.4cm) (0.08,.48)rectangle(-.08,.34);
	                                \draw (75:0.60cm) node[above=-9pt]{$\hspace{-0.2cm}\phantom{l}^{i}$};
	  \filldraw[black] (-45:0.4cm) circle(1.7pt);
	  \filldraw[black] (-135:0.4cm) circle(1.7pt);
 \end{tikzpicture}}

\def\bnbt{
 \begin{tikzpicture}
 \draw[thick] (0cm,0cm) circle(0.4cm);
                \filldraw[black] (90:0.4cm) (0.08,.48)rectangle(-.08,.34);
	                                \draw (75:0.60cm) node[above=-9pt]{$\hspace{-0.2cm}\phantom{l}^{i}$};
	  \filldraw[black] (-45:0.4cm) circle(1.7pt);
	   		\draw (48:0.67cm) node[above=-36pt]{$\hspace{-0.1cm}\phantom{l}^{n_1}$};
	  \filldraw[black] (-135:0.4cm) circle(1.7pt);
 \end{tikzpicture}}

\def\bbnt{
 \begin{tikzpicture}
 \draw[thick] (0cm,0cm) circle(0.4cm);
                \filldraw[black] (90:0.4cm) (0.08,.48)rectangle(-.08,.34);
	                                \draw (75:0.60cm) node[above=-9pt]{$\hspace{-0.2cm}\phantom{l}^{i}$};
	  \filldraw[black] (-45:0.4cm) circle(1.7pt);
	  \filldraw[black] (-135:0.4cm) circle(1.7pt);
	  	   		\draw (-155:0.43cm) node[above=-16pt]{$\hspace{-0.2cm}\phantom{l}^{n_1}$};
 \end{tikzpicture}}

\def\bkbnt{
 \begin{tikzpicture}
 \draw[thick] (0cm,0cm) circle(0.4cm);
                \filldraw[black] (90:0.4cm) (0.08,.48)rectangle(-.08,.34);
	                                \draw (75:0.60cm) node[above=-9pt]{$\hspace{-0.2cm}\phantom{l}^{i}$};
	  \filldraw[black] (-45:0.4cm) circle(1.7pt);
	  			\draw (50:0.65cm) node[above=-36pt]{$\hspace{-0.1cm}\phantom{l}^{n_1}$};
	  \filldraw[black] (-135:0.4cm) circle(1.7pt);
	  	   		\draw (-152:0.4cm) node[above=-16pt]{$\hspace{-0.2cm}\phantom{l}^{n_2}$};
 \end{tikzpicture}}

\def\bbbt{
 \begin{tikzpicture}
 \draw[thick] (0cm,0cm) circle(0.4cm);
                \filldraw[black] (90:0.4cm) (0.08,.48)rectangle(-.08,.34);
                                \draw (75:0.60cm) node[above=-9pt]{$\hspace{-0.2cm}\phantom{l}^{i}$};
	 \filldraw[black] (0:0.4cm) circle(1.7pt);
	 \filldraw[black] (-90:0.4cm) circle(1.7pt);
	 \filldraw[black] (-180:0.4cm) circle(1.7pt);
 \end{tikzpicture}}

\def\bbbbt{
 \begin{tikzpicture}
 \draw[thick] (0cm,0cm) circle(0.4cm);
                \filldraw[black] (90:0.4cm) (0.08,.48)rectangle(-.08,.34);
                                \draw (75:0.60cm) node[above=-9pt]{$\hspace{-0.2cm}\phantom{l}^{i}$};
	 \filldraw[black] (20:0.4cm) circle(1.7pt);
	 \filldraw[black] (-50:0.4cm) circle(1.7pt);
	 \filldraw[black] (-130:0.4cm) circle(1.7pt);
	 \filldraw[black] (-200:0.4cm) circle(1.7pt);
 \end{tikzpicture}}

\def\bbnbt{
 \begin{tikzpicture}
 \draw[thick] (0cm,0cm) circle(0.4cm);
                \filldraw[black] (90:0.4cm) (0.08,.48)rectangle(-.08,.34);
                                \draw (75:0.60cm) node[above=-9pt]{$\hspace{-0.2cm}\phantom{l}^{i}$};
	 \filldraw[black] (0:0.4cm) circle(1.7pt);
	 \filldraw[black] (-90:0.4cm) circle(1.7pt);
	 	\draw (-90:0.35cm) node[above=-16pt]{$\hspace{0.2cm}\phantom{l}^{n_1}$};
	 \filldraw[black] (-180:0.4cm) circle(1.7pt);
 \end{tikzpicture}}

\def\bbnbkt{
 \begin{tikzpicture}
 \draw[thick] (0cm,0cm) circle(0.4cm);
                \filldraw[black] (90:0.4cm) (0.08,.48)rectangle(-.08,.34);
                                \draw (75:0.60cm) node[above=-9pt]{$\hspace{-0.2cm}\phantom{l}^{i}$};
	 \filldraw[black] (0:0.4cm) circle(1.7pt);
	 \filldraw[black] (-90:0.4cm) circle(1.7pt);
	 	\draw (-90:0.35cm) node[above=-16pt]{$\hspace{0.2cm}\phantom{l}^{n_1}$};
	 \filldraw[black] (-180:0.4cm) circle(1.7pt);
	 	\draw (-210:0.8cm) node[above=-16pt]{$\hspace{0.0cm}\phantom{l}^{n_2}$};
 \end{tikzpicture}}

\def\bbkbbnt{
 \begin{tikzpicture}
 \draw[thick] (0cm,0cm) circle(0.4cm);
                	\filldraw[black] (90:0.4cm) (0.08,.48)rectangle(-.08,.34);
                			\draw (75:0.60cm) node[above=-9pt]{$\hspace{-0.2cm}\phantom{l}^{i}$};
	\filldraw[black] (25:0.4cm) circle(1.7pt);
	\filldraw[black] (-50:0.4cm) circle(1.7pt);
		\draw (-20:0.5cm) node[above=-16pt]{$\hspace{0.0cm}\phantom{l}^{n_1}$};
	\filldraw[black] (-130:0.4cm) circle(1.7pt);
	\filldraw[black] (-200:0.4cm) circle(1.7pt);
		\draw (-210:0.9cm) node[above=-16pt]{$\hspace{0.1cm}\phantom{l}^{n_2}$};
 \end{tikzpicture}}

\def\rt{
 \begin{tikzpicture}
 \draw[thick] (0cm,0cm) circle(0.4cm);
                	\filldraw[black] (90:0.4cm) (0.08,.48)rectangle(-.08,.34);
                			\draw (75:0.60cm) node[above=-9pt]{$\hspace{-0.2cm}\phantom{l}^{i}$};
	\filldraw[fill=white, draw=black,thick] (-90:0.4cm) circle(1.7pt);
 \end{tikzpicture}}

\def\rtn{
 \begin{tikzpicture}
 \draw[thick] (0cm,0cm) circle(0.4cm);
                \filldraw[black] (90:0.4cm) (0.08,.48)rectangle(-.08,.34);
                                \draw (75:0.60cm) node[above=-9pt]{$\hspace{-0.0cm}\phantom{l}^{n_1}$};
	 \filldraw[fill=white, draw=black,thick] (-90:0.4cm) circle(1.7pt);
 \end{tikzpicture}}

\def\rrt{
 \begin{tikzpicture}
 \draw[thick] (0cm,0cm) circle(0.4cm);
                \filldraw[black] (90:0.4cm) (0.08,.48)rectangle(-.08,.34);
	                                \draw (75:0.60cm) node[above=-9pt]{$\hspace{-0.2cm}\phantom{l}^{i}$};
	  \filldraw[fill=white, draw=black,thick] (-45:0.4cm) circle(1.7pt);
	  \filldraw[fill=white, draw=black,thick] (-135:0.4cm) circle(1.7pt);
 \end{tikzpicture}}

\def\rrrt{
 \begin{tikzpicture}
 \draw[thick] (0cm,0cm) circle(0.4cm);
                	\filldraw[black] (90:0.4cm) (0.08,.48)rectangle(-.08,.34);
                			\draw (75:0.60cm) node[above=-9pt]{$\hspace{-0.2cm}\phantom{l}^{i}$};
	\filldraw[fill=white, draw=black,thick] (0:0.4cm) circle(1.7pt);
	\filldraw[fill=white, draw=black,thick] (-90:0.4cm) circle(1.7pt);
	\filldraw[fill=white, draw=black,thick] (-180:0.4cm) circle(1.7pt);
 \end{tikzpicture}}

\def\fourrt{
 \begin{tikzpicture}
 \draw[thick] (0cm,0cm) circle(0.4cm);
                \filldraw[black] (90:0.4cm) (0.08,.48)rectangle(-.08,.34);
                                \draw (75:0.60cm) node[above=-9pt]{$\hspace{-0.2cm}\phantom{l}^{i}$};
	 \filldraw[fill=white, draw=black,thick] (18:0.4cm) circle(1.7pt);
	 \filldraw[fill=white, draw=black,thick] (-54:0.4cm) circle(1.7pt);
	 \filldraw[fill=white, draw=black,thick] (-126:0.4cm) circle(1.7pt);
	 \filldraw[fill=white, draw=black,thick] (-198:0.4cm) circle(1.7pt);
 \end{tikzpicture}}
 
 \def\fivert{
 \begin{tikzpicture}
 \draw[thick] (0cm,0cm) circle(0.4cm);
                \filldraw[black] (90:0.4cm) (0.08,.48)rectangle(-.08,.34);
                                \draw (75:0.60cm) node[above=-9pt]{$\hspace{-0.2cm}\phantom{l}^{i}$};
	 \filldraw[fill=white, draw=black,thick] (30:0.4cm) circle(1.7pt);
	 \filldraw[fill=white, draw=black,thick] (-30:0.4cm) circle(1.7pt);
	 \filldraw[fill=white, draw=black,thick] (-90:0.4cm) circle(1.7pt);
	 \filldraw[fill=white, draw=black,thick] (-150:0.4cm) circle(1.7pt);
	 \filldraw[fill=white, draw=black,thick] (-210:0.4cm) circle(1.7pt);
 \end{tikzpicture}}
 
 \def\sixrt{
 \begin{tikzpicture}
 \draw[thick] (0cm,0cm) circle(0.4cm);
                \filldraw[black] (90:0.4cm) (0.08,.48)rectangle(-.08,.34);
                                \draw (75:0.60cm) node[above=-9pt]{$\hspace{-0.2cm}\phantom{l}^{i}$};
	 \filldraw[fill=white, draw=black,thick] (38.6:0.4cm) circle(1.7pt);
	 \filldraw[fill=white, draw=black,thick] (-12.8:0.4cm) circle(1.7pt);
	 \filldraw[fill=white, draw=black,thick] (-64.3:0.4cm) circle(1.7pt);
	 \filldraw[fill=white, draw=black,thick] (-115.7:0.4cm) circle(1.7pt);
	 \filldraw[fill=white, draw=black,thick] (-167.1:0.4cm) circle(1.7pt);
	 \filldraw[fill=white, draw=black,thick] (-218.6:0.4cm) circle(1.7pt);
 \end{tikzpicture}}
 
 \def\sevenrt{
 \begin{tikzpicture}
 \draw[thick] (0cm,0cm) circle(0.4cm);
                \filldraw[black] (90:0.4cm) (0.08,.48)rectangle(-.08,.34);
                                \draw (75:0.60cm) node[above=-9pt]{$\hspace{-0.2cm}\phantom{l}^{i}$};
	 \filldraw[fill=white, draw=black,thick] (45:0.4cm) circle(1.7pt);
	 \filldraw[fill=white, draw=black,thick] (0:0.4cm) circle(1.7pt);
	 \filldraw[fill=white, draw=black,thick] (-45:0.4cm) circle(1.7pt);
	 \filldraw[fill=white, draw=black,thick] (-90:0.4cm) circle(1.7pt);
	 \filldraw[fill=white, draw=black,thick] (-135:0.4cm) circle(1.7pt);
	 \filldraw[fill=white, draw=black,thick] (-180:0.4cm) circle(1.7pt);
	 \filldraw[fill=white, draw=black,thick] (-225:0.4cm) circle(1.7pt);
 \end{tikzpicture}}

\def\brt{
 \begin{tikzpicture}
 \draw[thick] (0cm,0cm) circle(0.4cm);
                \filldraw[black] (90:0.4cm) (0.08,.48)rectangle(-.08,.34);
	                                \draw (75:0.60cm) node[above=-9pt]{$\hspace{-0.2cm}\phantom{l}^{i}$};
	  \filldraw[black] (-45:0.4cm) circle(1.7pt);
	  \filldraw[fill=white, draw=black,thick] (-135:0.4cm) circle(1.7pt);
 \end{tikzpicture}}

\def\brtn{
 \begin{tikzpicture}
 \draw[thick] (0cm,0cm) circle(0.4cm);
                \filldraw[black] (90:0.4cm) (0.08,.48)rectangle(-.08,.34);
	                               \draw (75:0.60cm) node[above=-9pt]{$\hspace{0.cm}\phantom{l}^{n_1}$};
	  \filldraw[black] (-45:0.4cm) circle(1.7pt);
	  \filldraw[fill=white, draw=black,thick] (-135:0.4cm) circle(1.7pt);
 \end{tikzpicture}}

\def\rbnt{
 \begin{tikzpicture}
 \draw[thick] (0cm,0cm) circle(0.4cm);
                \filldraw[black] (90:0.4cm) (0.08,.48)rectangle(-.08,.34);
	                                \draw (75:0.60cm) node[above=-9pt]{$\hspace{-0.2cm}\phantom{l}^{i}$};
	  \filldraw[fill=white, draw=black,thick] (-45:0.4cm) circle(1.7pt);
	  \filldraw[black] (-135:0.4cm) circle(1.7pt);
	  	   		\draw (-155:0.42cm) node[above=-16pt]{$\hspace{-0.2cm}\phantom{l}^{n_2}$};
 \end{tikzpicture}}

\def\bnrt{
 \begin{tikzpicture}
 \draw[thick] (0cm,0cm) circle(0.4cm);
                \filldraw[black] (90:0.4cm) (0.08,.48)rectangle(-.08,.34);
	                                \draw (75:0.60cm) node[above=-9pt]{$\hspace{-0.2cm}\phantom{l}^{i}$};
	  \filldraw[black] (-45:0.4cm) circle(1.7pt);
	  	  	   	\draw (-20:0.55cm) node[above=-16pt]{$\hspace{-0.2cm}\phantom{l}^{n_1}$};
	  \filldraw[fill=white, draw=black,thick] (-135:0.4cm) circle(1.7pt);
 \end{tikzpicture}}

\def\rbt{
 \begin{tikzpicture}
 \draw[thick] (0cm,0cm) circle(0.4cm);
                \filldraw[black] (90:0.4cm) (0.08,.48)rectangle(-.08,.34);
	                                \draw (75:0.60cm) node[above=-9pt]{$\hspace{-0.2cm}\phantom{l}^{i}$};
	  \filldraw[fill=white, draw=black,thick] (-45:0.4cm) circle(1.7pt);
	  \filldraw[black] (-135:0.4cm) circle(1.7pt);
 \end{tikzpicture}}

\def\brbt{
 \begin{tikzpicture}
 \draw[thick] (0cm,0cm) circle(0.4cm);
                	\filldraw[black] (90:0.4cm) (0.08,.48)rectangle(-.08,.34);
                			\draw (75:0.60cm) node[above=-9pt]{$\hspace{-0.2cm}\phantom{l}^{i}$};
	\filldraw[black] (0:0.4cm) circle(1.7pt);
	\filldraw[fill=white, draw=black,thick] (-90:0.4cm) circle(1.7pt);
	\filldraw[black] (-180:0.4cm) circle(1.7pt);
 \end{tikzpicture}}

\def\rbbt{
 \begin{tikzpicture}
 \draw[thick] (0cm,0cm) circle(0.4cm);
                	\filldraw[black] (90:0.4cm) (0.08,.48)rectangle(-.08,.34);
                			\draw (75:0.60cm) node[above=-9pt]{$\hspace{-0.2cm}\phantom{l}^{i}$};
	\filldraw[fill=white, draw=black,thick] (0:0.4cm) circle(1.7pt);
	\filldraw[black] (-90:0.4cm) circle(1.7pt);
	\filldraw[black] (-180:0.4cm) circle(1.7pt);
 \end{tikzpicture}}

\def\bbrt{
 \begin{tikzpicture}
 \draw[thick] (0cm,0cm) circle(0.4cm);
                	\filldraw[black] (90:0.4cm) (0.08,.48)rectangle(-.08,.34);
                			\draw (75:0.60cm) node[above=-9pt]{$\hspace{-0.2cm}\phantom{l}^{i}$};
	\filldraw[black] (0:0.4cm) circle(1.7pt);
	\filldraw[black] (-90:0.4cm) circle(1.7pt);
	\filldraw[fill=white, draw=black,thick] (-180:0.4cm) circle(1.7pt);
 \end{tikzpicture}}

\def\bbnrt{
 \begin{tikzpicture}
 \draw[thick] (0cm,0cm) circle(0.4cm);
                	\filldraw[black] (90:0.4cm) (0.08,.48)rectangle(-.08,.34);
                			\draw (75:0.60cm) node[above=-9pt]{$\hspace{-0.2cm}\phantom{l}^{i}$};
	\filldraw[black] (0:0.4cm) circle(1.7pt);
	\filldraw[black] (-90:0.4cm) circle(1.7pt);
		\draw (-90:0.35cm) node[above=-16pt]{$\hspace{0.2cm}\phantom{l}^{n_1}$};
	\filldraw[fill=white, draw=black,thick] (-180:0.4cm) circle(1.7pt);
 \end{tikzpicture}}

\def\rrbt{
 \begin{tikzpicture}
 \draw[thick] (0cm,0cm) circle(0.4cm);
                	\filldraw[black] (90:0.4cm) (0.08,.48)rectangle(-.08,.34);
                			\draw (75:0.60cm) node[above=-9pt]{$\hspace{-0.2cm}\phantom{l}^{i}$};
	\filldraw[fill=white, draw=black,thick] (0:0.4cm) circle(1.7pt);
	\filldraw[fill=white, draw=black,thick] (-90:0.4cm) circle(1.7pt);
	\filldraw[black] (-180:0.4cm) circle(1.7pt);
 \end{tikzpicture}}

\def\rbbt{
 \begin{tikzpicture}
 \draw[thick] (0cm,0cm) circle(0.4cm);
                	\filldraw[black] (90:0.4cm) (0.08,.48)rectangle(-.08,.34);
                			\draw (75:0.60cm) node[above=-9pt]{$\hspace{-0.2cm}\phantom{l}^{i}$};
	\filldraw[fill=white, draw=black,thick] (0:0.4cm) circle(1.7pt);
	\filldraw[black] (-90:0.4cm) circle(1.7pt);
	\filldraw[black] (-180:0.4cm) circle(1.7pt);
 \end{tikzpicture}}

\def\bbrt{
 \begin{tikzpicture}
 \draw[thick] (0cm,0cm) circle(0.4cm);
                	\filldraw[black] (90:0.4cm) (0.08,.48)rectangle(-.08,.34);
                			\draw (75:0.60cm) node[above=-9pt]{$\hspace{-0.2cm}\phantom{l}^{i}$};
	\filldraw[black] (0:0.4cm) circle(1.7pt);
	\filldraw[black] (-90:0.4cm) circle(1.7pt);
	\filldraw[fill=white, draw=black,thick] (-180:0.4cm) circle(1.7pt);
 \end{tikzpicture}}

\def\brbrt{
 \begin{tikzpicture}
 \draw[thick] (0cm,0cm) circle(0.4cm);
                	\filldraw[black] (90:0.4cm) (0.08,.48)rectangle(-.08,.34);
                			\draw (75:0.60cm) node[above=-9pt]{$\hspace{-0.2cm}\phantom{l}^{i}$};
	\filldraw[black] (25:0.4cm) circle(1.7pt);
	\filldraw[fill=white, draw=black,thick] (-50:0.4cm) circle(1.7pt);
	\filldraw[black] (-130:0.4cm) circle(1.7pt);
	\filldraw[fill=white, draw=black,thick] (-200:0.4cm) circle(1.7pt);
 \end{tikzpicture}}

\def\bbnbrt{
 \begin{tikzpicture}
 \draw[thick] (0cm,0cm) circle(0.4cm);
                	\filldraw[black] (90:0.4cm) (0.08,.48)rectangle(-.08,.34);
                			\draw (75:0.60cm) node[above=-9pt]{$\hspace{-0.2cm}\phantom{l}^{i}$};
	\filldraw[black] (25:0.4cm) circle(1.7pt);
	\filldraw[black] (-50:0.4cm) circle(1.7pt);
		\draw (-20:0.55cm) node[above=-16pt]{$\hspace{-0.1cm}\phantom{l}^{n_1}$};
	\filldraw[black] (-130:0.4cm) circle(1.7pt);
	\filldraw[fill=white, draw=black,thick] (-200:0.4cm) circle(1.7pt);
 \end{tikzpicture}}

\def\brbbnt{
 \begin{tikzpicture}
 \draw[thick] (0cm,0cm) circle(0.4cm);
                	\filldraw[black] (90:0.4cm) (0.08,.48)rectangle(-.08,.34);
                			\draw (75:0.60cm) node[above=-9pt]{$\hspace{-0.2cm}\phantom{l}^{i}$};
	\filldraw[black] (25:0.4cm) circle(1.7pt);
	\filldraw[fill=white, draw=black,thick] (-50:0.4cm) circle(1.7pt);
	\filldraw[black] (-130:0.4cm) circle(1.7pt);
	\filldraw[black] (-200:0.4cm) circle(1.7pt);
		\draw (-215:0.85cm) node[above=-16pt]{$\hspace{0.05cm}\phantom{l}^{n_2}$};
 \end{tikzpicture}}

\def\rbrbt{
 \begin{tikzpicture}
 \draw[thick] (0cm,0cm) circle(0.4cm);
                	\filldraw[black] (90:0.4cm) (0.08,.48)rectangle(-.08,.34);
                			\draw (75:0.60cm) node[above=-9pt]{$\hspace{-0.2cm}\phantom{l}^{i}$};
	\filldraw[fill=white, draw=black,thick] (25:0.4cm) circle(1.7pt);
	\filldraw[black] (-50:0.4cm) circle(1.7pt);
	\filldraw[fill=white, draw=black,thick] (-130:0.4cm) circle(1.7pt);
	\filldraw[black] (-200:0.4cm) circle(1.7pt);
 \end{tikzpicture}}

 \def\arbrea{
 \begin{tikzpicture}
\draw[thick] (0,0) -- (0,-0.5);
                	\filldraw[green] (0,0) circle(2.5pt);
                	\filldraw[black] (0,-0.5) circle(2pt);
\end{tikzpicture}
}

 \def\arbreb{
 \begin{tikzpicture}
\draw[thick] (0,0) -- (0.5,-0.5);
\draw[thick] (0,0) -- (-0.5,-0.5);
                	\filldraw[black] (0.5,-0.5) circle(2pt);
                	\filldraw[black] (-0.5,-0.5) circle(2pt);
                	\filldraw[green] (0,0) circle(2.5pt);
\end{tikzpicture}
}

 \def\arbrec{
 \begin{tikzpicture}
\draw[thick] (0,0) -- (0,-0.5);
\draw[thick] (0,-0.5) -- (0,-1);
                	\filldraw[black] (0,-0.5) circle(2pt);
                	\filldraw[black] (0,-1) circle(2pt);
                	\filldraw[green] (0,0) circle(2.5pt);
\end{tikzpicture}
}

 \def\arbred{
 \begin{tikzpicture}
\draw[thick] (0,0) -- (0,-0.5);
\draw[thick] (0,-0.5) -- (0,-1);
\draw[thick] (0,-1) -- (0,-1.5);
                	\filldraw[black] (0,-0.5) circle(2pt);
                	\filldraw[black] (0,-1) circle(2pt);
                	\filldraw[black] (0,-1.5) circle(2pt);
                	\filldraw[green] (0,0) circle(2.5pt);
\end{tikzpicture}
}

 \def\arbree{
 \begin{tikzpicture}
\draw[thick] (0,0) -- (0.5,-0.5);
\draw[thick] (0,0) -- (-0.5,-0.5);
\draw[thick] (0,0) -- (0,-0.5);
                	\filldraw[black] (0,-0.5) circle(2pt);
                	\filldraw[black] (0.5,-0.5) circle(2pt);
                	\filldraw[black] (-0.5,-0.5) circle(2pt);
                	\filldraw[green] (0,0) circle(2.5pt);
\end{tikzpicture}
}

 \def\arbref{
 \begin{tikzpicture}
\draw[thick] (0,0) -- (0.5,-0.5);
\draw[thick] (0,0) -- (-0.5,-0.5);
\draw[thick] (-0.5,-0.5) -- (-0.5,-1);
                	\filldraw[black] (0.5,-0.5) circle(2pt);
                	\filldraw[black] (-0.5,-0.5) circle(2pt);
                	\filldraw[black] (-0.5,-1) circle(2pt);
                	\filldraw[green] (0,0) circle(2.5pt);
\end{tikzpicture}
}

 \def\arbreg{
 \begin{tikzpicture}
\draw[thick] (0,0) -- (0,-0.5);
\draw[thick] (0,-0.5) -- (0.5,-1);
\draw[thick] (0,-0.5) -- (-0.5,-1);
                	\filldraw[green] (0,0) circle(2.5pt);
                	\filldraw[black] (0,-0.5) circle(2pt);
                	\filldraw[black] (0.5,-1) circle(2pt);
                	\filldraw[black] (-0.5,-1) circle(2pt);
\end{tikzpicture}
}

 \def\arbreci{
 \begin{tikzpicture}
\draw[thick] (0,0) -- (0,-0.5);
\draw[thick] (0,-0.5) -- (0,-1);
\draw[thick] (0,-1) -- (0,-1.5);
                	\filldraw[black] (0,-0.5) circle(2pt);
                	\filldraw[black] (0,-1) circle(2pt);
                	\filldraw[black] (0,-1.5) circle(2pt);
                	\filldraw[green] (0,0) circle(2.5pt);
                			\draw (0.2,0.2) node[above=-9pt]{$i$};
                			\draw (0.2,-0.5) node[above=-9pt]{$1$};
                			\draw (0.2,-1) node[above=-9pt]{$2$};
                			\draw (0.2,-1.5) node[above=-9pt]{$3$};
\end{tikzpicture}
}

 \def\arbredi{
 \begin{tikzpicture}
\draw[thick] (0,0) -- (0.5,-0.5);
\draw[thick] (0,0) -- (-0.5,-0.5);
\draw[thick] (0,0) -- (0,-0.5);
                	\filldraw[black] (0,-0.5) circle(2pt);
                	\filldraw[black] (0.5,-0.5) circle(2pt);
                	\filldraw[black] (-0.5,-0.5) circle(2pt);
                	\filldraw[green] (0,0) circle(2.5pt);
                			\draw (0.2,0.2) node[above=-9pt]{$i$};
                			\draw (0,-0.7) node[above=-9pt]{$2$};
                			\draw (0.5,-0.7) node[above=-9pt]{$3$};
                			\draw (-0.5,-0.7) node[above=-9pt]{$1$};
\end{tikzpicture}
}

 \def\arbreei{
 \begin{tikzpicture}
\draw[thick] (0,0) -- (0.5,-0.5);
\draw[thick] (0,0) -- (-0.5,-0.5);
\draw[thick] (-0.5,-0.5) -- (-0.5,-1);
                	\filldraw[black] (0.5,-0.5) circle(2pt);
                	\filldraw[black] (-0.5,-0.5) circle(2pt);
                	\filldraw[black] (-0.5,-1) circle(2pt);
                	\filldraw[green] (0,0) circle(2.5pt);
                			\draw (0.2,0.2) node[above=-9pt]{$i$};
                			\draw (-0.7,-0.5) node[above=-9pt]{$1$};
                			\draw (0.5,-0.7) node[above=-9pt]{$2$};
                			\draw (-0.5,-1.2) node[above=-9pt]{$3$};
\end{tikzpicture}
}

 \def\arbregi{
 \begin{tikzpicture}
\draw[thick] (0,0) -- (0,-0.5);
\draw[thick] (0,-0.5) -- (0.5,-1);
\draw[thick] (0,-0.5) -- (-0.5,-1);
                	\filldraw[green] (0,0) circle(2.5pt);
                	\filldraw[black] (0,-0.5) circle(2pt);
                	\filldraw[black] (0.5,-1) circle(2pt);
                	\filldraw[black] (-0.5,-1) circle(2pt);
					\draw (0.2,0.2) node[above=-9pt]{$i$};
                			\draw (0.2,-0.4) node[above=-9pt]{$1$};
                			\draw (0.5,-1.2) node[above=-9pt]{$3$};
                			\draw (-0.5,-1.2) node[above=-9pt]{$2$};
\end{tikzpicture}
}

\def\abs#1{\left\vert #1 \right\vert}
\def\allpoly{\mbox{$\re\langle X \rangle$}}

\def\allpolyx0degn{\mbox{$P_n$}}

\def\allseries{\mbox{$\re\langle\langle X \rangle\rangle$}}

\def\allseriesdeltam{\mbox{$\re^m\langle\langle X_\delta \rangle\rangle$}}

\def\allseriesell{\mbox{$\re^{\ell} \langle\langle X \rangle\rangle$}}

\def\allseriesm{\mbox{$\re^m\langle\langle X \rangle\rangle$}}

\def\allseriesmdual{\mbox{$\re^m\langle\langle X \rangle\rangle^\ast$}}

\def\allseriesX1{\mbox{$\re [[ X_1 ]]$}}
\def\compAH{\psi} 

\def\deg{\mathsf{deg}}

\def\Endallseries{{\rm End}(\allseries)}
\def\eqref#1{(\ref{#1})} 

\def\Fliessdelta{\mathscr{F}_{\delta}}

\def\gsc{\mathfrak{c}} 
\def\gsd{\mathfrak{d}} 
\def\gse{\mathfrak{e}} 

\def\id{{\rm id}}

\def\mbf#1{\hbox{\mathversion{bold}$#1$}} 
\def\modcomp{\:\tilde{\circ}\,} 
\def\modcompAH{\phi} 

\def\norm#1{\Vert#1\Vert}

\def\re{{\mathbb R}} 

\def\sameau{\rule[0.017in]{0.2in}{0.012in}}

\def\shuffle{{\scriptscriptstyle \;\sqcup \hspace*{-0.05cm}\sqcup\;}}

\def\spanset{{\rm span}}

\def\begce{\begin{center}}
\def\endce{\end{center}}
\def\begar{\begin{array}}
\def\endar{\end{array}}
\def\begeq{\begin{equation}}
\def\endeq{\end{equation}}
\def\begdi{\begin{displaymath}}
\def\enddi{\end{displaymath}}
\def\begdis{\begin{eqnarray*}}
\def\enddis{\end{eqnarray*}}
\def\begeqa{\begin{eqnarray}}
\def\endeqa{\end{eqnarray}}
\def\begdes{\begin{description}}
\def\enddes{\end{description}}
\def\begit{\begin{itemize}}
\def\endit{\end{itemize}}
\def\begen{\begin{enumerate}}
\def\enden{\end{enumerate}}
\def\beglar{\left[\begin{array}}
\def\endrar{\end{array}\right]}
\def\begle{\begin{lem}}
\def\endle{\end{lem}}
\def\begde{\begin{defn}}
\def\endde{\end{defn}}
\def\begth{\begin{thm}}
\def\endth{\end{thm}}
\def\begco{\begin{coro}}
\def\endco{\end{coro}}
\def\begprop{\begin{proposition}}
\def\endprop{\end{proposition}}
\def\begex{\begin{example}}
\def\endex{\end{example}}
\def\begexer{\begin{exercise}}
\def\endexer{\end{exercise}}

\def\begres{\noindent{\bf Remarks}:\begin{enumerate}}
\def\endres{\end{enumerate} \par}
\def\begpr{\begin{pf}}
\def\endpr{\end{pf}}
\def\begtab{\begin{tabular}}
\def\endtab{\end{tabular}}
\def\rref#1{(\ref{#1})}

\begin{document}

\selectlanguage{english}

\title[Hopf algebra of rooted circle trees]{A combinatorial Hopf algebra for nonlinear\\
output feedback control systems}

\begin{abstract}
In this work a combinatorial description is provided of a Fa\`a di Bruno type Hopf algebra which naturally appears in the context of Fliess operators in nonlinear feedback control theory. It is a connected graded commutative and non-cocommutative Hopf algebra defined on rooted circle trees. A cancellation free forest formula for its antipode is given.
\end{abstract}

\author{Luis A.~Duffaut Espinosa}
\address{Department of Electrical and Computer Engineering, George Mason University, Fairfax, Virginia 22030 USA}
\email{lduffaut@gmu.edu}

\author{Kurusch Ebrahimi-Fard}
\address{Instituto de Ciencias Matem\'aticas, Consejo Superior de Investigaciones Cient\'{i}ficas,
		C/ Nicol\'as Cabrera, no.~13-15, 28049 Madrid, Spain.
		On leave from Univ.~de Haute Alsace, Mulhouse, France}
         \email{kurusch@icmat.es, kurusch.ebrahimi-fard@uha.fr}
         \urladdr{www.icmat.es/kurusch}

\author{W.~Steven Gray}
\address{Department of Electrical and Computer Engineering,
Old Dominion University,
		Norfolk, Virginia 23529 USA}
\email{sgray@odu.edu}

\date{\today}

\maketitle
\tableofcontents


\section{Introduction}
\label{sect:intro}

Fliess operators \cite{Fliess_81,Fliess_83}, also known as Chen--Fliess functional expansions, play a crucial role in the theory of nonlinear control systems \cite{Isidori_95}. Most systems found in science and engineering can be viewed as a set of simpler interconnected subsystems. Ferfera was the first to describe some of the mathematical structures that result when these interconnected subsystems are modeled as Fliess operators \cite{Ferfera_79,Ferfera_80}. Four basic system interconnections are found in most control systems. They are known as the parallel sum, parallel product, cascade and feedback connections. (See  \cite{Gray-Duffauc_Espinosa_FdB14} for a review.) The first two are rather simple to describe mathematically. The third operation is equivalent to composing Fliess operators. Combining Fliess operators through a feedback connection as shown in Figure~\ref{fig:feedback-with-v} is also based on composition. In short, the output of one system is fed into the input of another system, the output of which is then fed back into the input of the first. This feedback loop is, mathematically speaking, the most interesting interconnection. When one or both of the subsystems is nonlinear, the nature of the composite system can be remarkably complex. Despite this, it turns out that given two nonlinear input-output systems written in terms of such Chen--Fliess functional expansions, it is known that the feedback interconnected system is always well defined and in the same class \cite{Gray-Li_05,Gray-et-al_CLCA13}. An explicit formula for the generating series of a single-input, single-output (SISO) closed-loop system was provided by two of the authors, L.~Duffaut Espinosa and W.S.~Gray, in earlier work \cite{Gray-Duffauc_Espinosa_SCL11,Gray-Duffauc_Espinosa_FdB14} in the context of a Fa\`a di Bruno type Hopf algebra. L.~Foissy \cite{Foissy_13} made a crucial contribution by showing that this Hopf algebra is connected. As a result, a simple standard recursive formula could be used to calculate the antipode \cite{Gray-et-al_MTNS14}, which plays an important role in the analysis and design of feedback systems \cite{Gray-et-al_AUTOxx,Gray-et-al_CLCA13}. In a more recent work, the full multivariable extension of the theory as it applies to control theory is presented \cite{Gray-et-al_SCL14}. An obvious but largely non-physical multivariable extension is provided in \cite{Foissy_EJM15}. The focus here will be on the extension in \cite{Gray-et-al_SCL14} which, as will be described later, has a richer structure and more realistic properties from the point of view of proper applications in control theory.

\begin{figure}[t]
\begin{center}
\includegraphics*[scale=0.45]{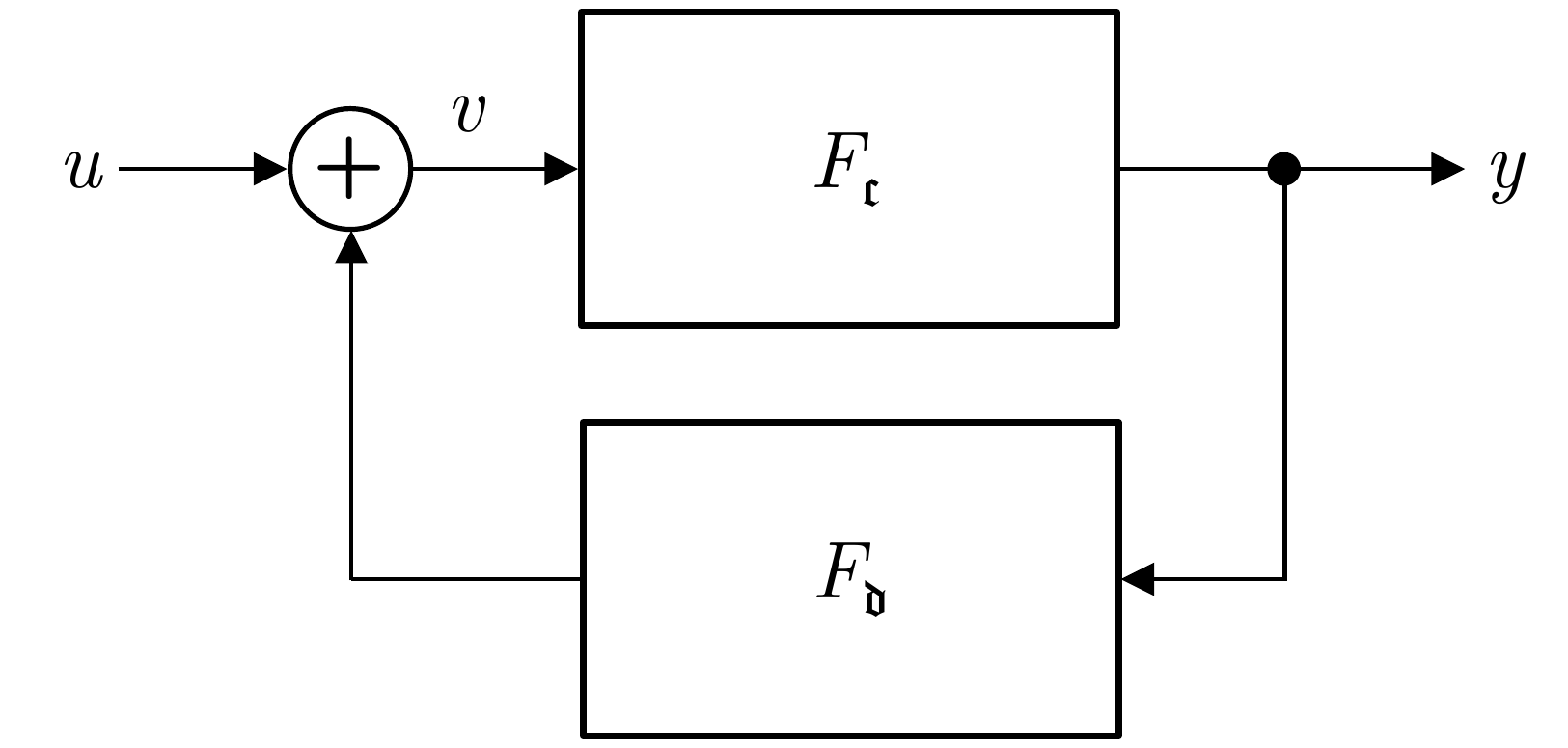}
\end{center}
\caption{Feedback connection of Fliess operators $F_\gsc$ and $F_\gsd$}
\label{fig:feedback-with-v}
\end{figure}

Foissy's insight regarding the grading in \cite{Foissy_13} brought into focus the combinatorial nature of the Fa\`a di Bruno type Hopf algebra found by L.~Duffaut Espinosa and W.S.~Gray. However, the central coproduct was presented in terms of two recursions, which made it difficult to understand its representation as a combinatorial algebra. M.~Aguiar characterized the latter in terms of a connected graded vector space with homogeneous components spanned by finite sets of combinatorial objects, and the algebraic structures are given by particular constructions on those objects. In the present work, a set of combinatorial objects is given that allows one to realize the algebraic structures of the aforementioned multivariable extension of the Hopf algebra of L.~Duffaut Espinosa and W.S.~Gray in terms of combinatorial operations on those objects. The elements in this set are based on the notion of rooted circle trees on which a connected graded bialgebra is defined in terms of extractions of rooted circle subtrees. The latter can appear nested, disjoint or overlapping, which is remotely similar to the situation found in Hopf algebras of Feynman graphs \cite{Connes-Kreimer_00,GBFV,Manchon_08}. Motivated by this analogy, a Zimmermann type forest formula on rooted circle trees is given, which provides a solution of the standard recursions for the antipode. As a result, the antipode can be calculated in a way which is free of cancellations. This is of interest with respect to applications in control theory such as system inversion and output tracking, where the antipode needs to be computed explicitly to high order. Thus, combinatorial efficient methods are needed for its computation.

\smallskip

The paper is organized as follows. In Section \ref{sect:prelim} a few basics on Hopf algebras are collected mainly to fix the notation. Section \ref{sect:rctHopfAlg} describes the Hopf algebra of decorated rooted circle trees. A cancellation free Zimmermann type forest formula is given for the antipode of this Hopf algebra. The pre-Lie algebra structure on rooted circle trees is presented in Subsection \ref{ssect:prct}. In the last section, the link to the multivariable extension of the Fa\`a di Bruno type Hopf algebra of L.~Duffaut Espinosa and W.S.~Gray is described.


\vspace{0.2cm}
\subsection*{Acknowledgments} The second author was supported by Ram\'on y Cajal research grant RYC-2010-06995 from the Spanish government. The third author was supported by grant SEV-2011-0087 from the Severo Ochoa Excellence Program at the Instituto de Ciencias Matem\'{a}ticas in Madrid, Spain.  This research was also supported by a grant from the BBVA Foundation.


\section{Preliminaries}
\label{sect:prelim}

In this section, a few basics on Hopf algebras are collected to fix the notation. For details, in particular on combinatorial Hopf algebras, the reader is referred to \cite{Figueroa-Gracia-Bondia_05,GBFV,Manchon_08}, as well as the standard reference \cite{Sweedler_69}.

In general, $k$ denotes the ground field of characteristic zero over which all algebraic structures are defined. A {\it{coalgebra}} over $k$ consist of a triple $(C,\Delta,\varepsilon)$. The coproduct $\Delta: C \to C \otimes C$ is coassociative, that is, $(\Delta \otimes \id)\circ \Delta=(\id \otimes \Delta)\circ \Delta$, and $\varepsilon: C \to k$ denotes the counit map. For a cocommutative coalgebra $\Delta = \pi \circ \Delta$, where the flip map $\pi$ is defined on $C \otimes C$ by $\pi(x \otimes y) = y \otimes x$. A {\it{bialgebra}} $B$ is both a unital algebra and a counital coalgebra together with compatibility relations, such as both the algebra product, $m$, and unit map, $e: k \to B$, are coalgebra morphisms \cite{Sweedler_69}. The unit of $B$ is denoted by $\mathbf{1} = e(1)$. A bialgebra is called {\it{graded}} if there are $k$-vector subspaces $B^{(n)}$, $n \geq 0$ such that $B= \bigoplus_{n \geq 0} B^{(n)}$, $m(B^{(p)} \otimes B^{(q)}) \subseteq B^{(p+q)}$ and $\Delta(B^{(n)}) \subseteq \bigoplus_{p+q=n} B^{(p)}\otimes B^{(q)}.$ Elements $x \in B^{(n)}$ are given a degree $\mathsf{deg}(x)=n$. Moreover, $B$ is called {\it{connected}} if $B^{(0)} = k\mathbf{1}$. Define $B^+=\bigoplus_{n > 0} B^{(n)}$. For any $x \in B^{(n)}$ the coproduct is of the form
\begin{equation*}
	\Delta(x) = x \otimes \mathbf{1} + \mathbf{1} \otimes x + \Delta'(x) \in \bigoplus_{k+l=n} B^{(k)} \otimes B^{(l)},
\end{equation*}
where $\Delta'(x) := \Delta(x) -  x \otimes \mathbf{1} - \mathbf{1} \otimes x \in B^+ \otimes B^+$ is the {\em reduced} coproduct. By definition an element $x \in B$ is {\it{primitive}} if $\Delta'(x) = 0.$  Suppose $A$ is a $k$-algebra with product $m_A$ and unit $e_A$, e.g., $A=k$ or $A=B$. The vector space  $L(B, A)$ of linear maps from the bialgebra $B$ to $A$ together with the convolution product $\Phi \star \Psi := m_{A} \circ (\Phi \otimes \Psi) \circ \Delta : B \to A$, where $\Phi,\Psi \in L(B,A)$, is an associative algebra with unit $\iota := e_{A} \circ \varepsilon$.

A {\it{Hopf algebra}} $H$ is a bialgebra together with a particular $k$-linear map called an {\it{antipode}} $S: H \to H$ which satisfies the Hopf algebra axioms~\cite{Sweedler_69}. When $A=H$, the antipode $S \in L(H,H)$ is characterized as the inverse of the identity map with respect to the convolution product, that is,
\begin{equation*}
    S  \star\id  = m \circ (S \otimes \id) \circ \Delta = \id \star S= e \circ \varepsilon.
\end{equation*}
It is well-known that any connected graded bialgebra is a {\sl{connected graded Hopf algebra}\/}.

Let $H=\bigoplus_{n \ge 0} H^{(n)}$ be a connected graded Hopf algebra. Suppose $A$ is a commutative unital algebra. The subset $g_0 \subset  L( H, A)$ of linear maps $\alpha$ that send the unit to zero, $\alpha(\mathbf{1})=0$, forms a Lie algebra in $ L( H, A)$. The exponential $ \exp^\star(\alpha) = \sum_{j\ge 0} \frac{1}{j!}\alpha^{\star j}$ is well defined and gives a bijection from $g_0$ onto the group $G_0 = \iota + g_0$ of linear maps, $\gamma$, that send the unit of $H$ to the algebra unit in $A$, $\gamma(\mathbf{1})=1_{A}$. An {\it{infinitesimal character}} with values in $A$ is a linear map $\xi \in L( H, A)$ such that for $x, y \in  H$, $\xi(xy) = \xi(x)\iota(y) + \iota(x)\xi(y)$. The linear space of infinitesimal characters is denoted $g_{A} \subset g_0$. An $A$-valued map $\Phi$ in $ L( H, A)$ is called a {\it{character}} if $\Phi(\mathbf{1})=1_{ A}$ and for $x,y \in  H$, $\Phi(xy) = \Phi(x)\Phi(y)$. The set of characters is denoted by $G_{ A} \subset G_0$. It forms a pro-unipotent group for the convolution product with (pro-nilpotent) Lie algebra $ g_{ A}$ \cite{GBFV,Manchon_08}. The exponential map $\exp^{\star}$ restricts to a bijection between $ g_{ A}$ and $G_{ A}$. The neutral element $\iota:=e_{ A}\circ \epsilon$  in $G_{ A}$ is given by $\iota(\mathbf{1})=1_{ A}$ and $\iota(x) = 0$ for $x \in \mathsf{Ker}(\varepsilon)=H^+$. The inverse of $\Phi \in G_{ A}$ is given by composition with the Hopf algebra antipode $S$, i.e., $\Phi^{\star -1} = \Phi \circ S$.


\section{Hopf algebra of rooted circle trees}
\label{sect:rctHopfAlg}

The presentation begins with a description of a class of rooted circle trees. The motivation for this setup will become clear once the theory of Fliess operators in nonlinear feedback control theory and the related  Fa\`a di Bruno type Hopf algebra are presented in the next section.

\begin{defn}\label{def:rct}
A {\it{rooted circle tree}} (rct) $c^i$ consists of a set of vertices, $V(c^i)$, one of which is distinguished as the root indexed by $i$, where $i \in \mathbb{N}$, $1 \le i \le m$. The vertices are strictly ordered and define a clockwise oriented closed edge $e$. The {\em weight} of $c^i$ is defined by its number of vertices, i.e., $|c^i|:=|V(c^i)|$. The empty rct is denoted by $\mathbf{1}$, and $|\mathbf{1}|=|V(\mathbf{\emptyset})|=0$. The set of internal vertices of $c^i$ is denoted $V(c)$ and is defined by excluding the root vertex from $V(c^i)$.
\end{defn}

A rooted circle tree with $n$ vertices may be seen as an $n$-polygone, with edges oriented clockwise starting from the root, such that each vertex has exactly one incoming and one outgoing edge. The last outgoing edge is the incoming edge of the root vertex. Here are the first few rooted circle trees of weights one up to five:
$$
	\et \quad \bt \quad \bbt \quad \bbbt \quad \bbbbt
$$
The root is denoted by a black square to distinguish it clearly from the internal vertices. A {\it{decorated rooted circle tree}} is a rct $c^i$ together with a mapping $\mathsf{Dec}: V(c) \to X$ from the set of its internal vertices $V(c)$ to the alphabet $X:=\{x_0,x_1,\ldots,x_m\}$ represented by
$$
\raisebox{-20pt}{
 \begin{tikzpicture}
 \draw[thick] (0cm,0cm) circle(0.4cm);
                \filldraw[black] (90:0.4cm) (0.08,.48)rectangle(-.08,.34);
                                \draw (75:0.60cm) node[above=-9pt]{$\hspace{-0.2cm}\phantom{l}^{i}$};
	 \filldraw[black] (45:0.4cm) circle(1.7pt);
	 		\draw (30:0.73cm) node[above=-9pt]{${\hspace{-0.2cm}\phantom{l}^{l_1}}$};
	 \filldraw[black] (0:0.4cm) circle(1.7pt);
	 		\draw (-8:0.73cm) node[above=-9pt]{$\hspace{-0.2cm}\phantom{l}^{l_2}$};
	 \filldraw[black] (-45:0.4cm) circle(1.7pt);
	 		\draw (-46:0.74cm) node[above=-9pt]{$\hspace{-0.2cm}\phantom{l}^{l_3}$};
	 \filldraw[black] (-105:0.33cm) circle(0.3pt);
	 \filldraw[black] (-130:0.33cm) circle(0.3pt);
	 \filldraw[black] (-154:0.33cm) circle(0.3pt);
	 \filldraw[black] (-210:0.4cm) circle(1.7pt);
	 		\draw (-215:0.76cm) node[above=-9pt]{$\hspace{0.15cm}\phantom{l}^{l_k}$};
 \end{tikzpicture}}
$$
The set (resp.~the linear span) of rooted circle trees will be denoted by $C^{\circ}$ (resp.~$\mathcal{C}^{\circ}$). The set (resp.~the linear span) of rooted circle trees decorated by an alphabet $X$ will be denoted by $C_X^{\circ}$ (resp.~$\mathcal{C}_X^{\circ}$). Consider a particular grading on the set $X$: $|x_0|=2$ and $|x_i|=1$ for all $i \neq 0$. The {\it degree} of a decorated rooted circle tree $c^i \in C_X^{\circ}$ of weight $|c^i|=k+1$ is defined by
\begin{equation}
\label{def:degree}
	\deg(c^i) := |x_{l_1}| + \cdots + |x_{l_k}| + 1.
\end{equation}
 Clearly this grading on $X$ distinguishes internal vertices of a rct decorated by $x_0$ from the rest. They will play an important role in the Hopf algebraic structure to be defined further below. To this end, it is natural to suppress the decoration by letters $x_k$ in favor of using colored vertices $f:=\{ \begin{tikzpicture}  \filldraw[black] (0,0) circle(2pt); \end{tikzpicture}, \begin{tikzpicture}  \filldraw[fill=white, draw=black,thick] (0,0) circle(2pt); \end{tikzpicture}\}$ and to consider bicolored rcts, i.e., rcts with inner vertices being colored either black or white. The white vertices correspond to a decoration by the letter $x_0$, whereas the black vertices represent decorations by any of the letters $x_k$, $k \neq 0$. Note that the decoration of the root is fixed to be $i$, $1 \le i \le m$. Here are the first few bicolored rcts of degree one up to five:
$$
	\et 		\qquad 
	\bt 		\qquad 
	\bbt  		\qquad 
	\rt  		\qquad
	\bbbt 	\qquad 
	\rbt 		\qquad 
	\brt 		\qquad
$$
$$
	\bbbbt 	\qquad
	\bbrt 		\qquad 
	\brbt 		\qquad 
	\rbbt 		\qquad 
	\rrt
$$
It will also be useful to have the white vertices of a rct $c^i$ enumerated consecutively in increasing order. For instance, the $s$ white vertices of the rct below are clockwise labeled from $1$ to $s$ 
$$
\raisebox{-20pt}{
 \begin{tikzpicture}
 \draw[thick] (0cm,0cm) circle(0.4cm);
                \filldraw[black] (90:0.4cm) (0.08,.48)rectangle(-.08,.34);
                                \draw (75:0.60cm) node[above=-9pt]{$\hspace{-0.2cm}\phantom{l}^{i}$};
	 \filldraw[black] (45:0.33cm) circle(0.2pt);
	 \filldraw[black] (35:0.33cm) circle(0.2pt);
	 \filldraw[black] (25:0.33cm) circle(0.2pt);
	 \filldraw[fill=white, draw=black,thick] (0:0.4cm) circle(1.7pt);
	 		\draw (-8:0.6cm) node[above=-9pt]{$\hspace{-0.2cm}\phantom{l}^{1}$};
 	\filldraw[black] (-25:0.33cm) circle(0.2pt);
	 \filldraw[black] (-35:0.33cm) circle(0.2pt);
	 \filldraw[black] (-45:0.33cm) circle(0.2pt);			
	 \filldraw[fill=white, draw=black,thick] (-75:0.4cm) circle(1.7pt);
	 		\draw (-70:0.6cm) node[above=-9pt]{$\hspace{-0.cm}\phantom{l}^{2}$};
	 \filldraw[black] (-105:0.33cm) circle(0.3pt);
	 \filldraw[black] (-115:0.33cm) circle(0.3pt);
	 \filldraw[black] (-125:0.33cm) circle(0.3pt);
	 \filldraw[fill=white, draw=black,thick] (-150:0.4cm) circle(1.7pt);
	 		\draw (-150:0.6cm) node[above=-9pt]{${\hspace{-0.1cm}\phantom{l}^s}$};
	 \filldraw[black] (-180:0.33cm) circle(0.3pt);
	 \filldraw[black] (-190:0.33cm) circle(0.3pt);
	 \filldraw[black] (-200:0.33cm) circle(0.3pt);
 \end{tikzpicture}}
$$
Denote by $|c^i|_{\begin{tikzpicture}  \filldraw[black] (0,0) circle(1.3pt); \end{tikzpicture}}$ and $|c^i|_{\begin{tikzpicture}  \filldraw[fill=white, draw=black,thick] (0,0) circle(1.3pt); \end{tikzpicture}}$ the number of black and white vertices of $c^i$, respectively. The weight of a bicolored rct $c^i$ is defined as before by the number of its internal vertices plus one. The degree of a bicolored rct $c^i$ is defined by:
$$
	\deg(c^i):= 2 |c^i|_{\begin{tikzpicture}  \filldraw[fill=white, draw=black,thick] (0,0) circle(1.3pt); \end{tikzpicture}}
		+ |c^i|_{\begin{tikzpicture}  \filldraw[black] (0,0) circle(1.3pt); \end{tikzpicture}}
		+ 1.
$$
The set (resp.~the linear span) of bicolored rooted circle trees will be denoted by $C_f^{\circ}$ (resp.~$\mathcal{C}_f^{\circ}$).
The use of bicolored rcts with white vertices enumerated consecutively significantly simplifies the description of the combinatorial structure to be defined next
on decorated rooted circle trees.

\smallskip

The set of internal vertices $V(c) \subset V(c^i)$ of a rct $c^i$ is naturally ordered with the minimal element being the first vertex to the right of the root, and the maximal one being the vertex to the left of the root. Therefore, the elements in any subset $J \subseteq V(c)$ can be written in strictly increasing order. 

Consider next the notion of {\it{admissible subsets}} of $V(c) \subset V(c^i)$.

\begin{defn}\label{def:admSubset}
Let $c^i$ be a rct in $C_f^{\circ}$. A subset  $J \subseteq V(c)$ is called {\em admissible} if its minimal element is a white vertex.
\end{defn}

For instance, all subsets of the rct
$$
	\rrrt
$$
are obviously admissible. The only admissible subset of the rct
$$
	\brt
$$
is the singleton $J_{1}=\{\begin{tikzpicture}  \filldraw[fill=white, draw=black,thick] (0,0) circle(2pt); \end{tikzpicture}\}$. For the rct
$$
	\rbt
$$
there are two admissible subsets, $J_{11}=\{\begin{tikzpicture}  \filldraw[fill=white, draw=black,thick] (0,0) circle(2pt); \end{tikzpicture}\}$ and $J_{12}=\{\begin{tikzpicture}  \filldraw[fill=white, draw=black,thick] (0,0) circle(2pt); \end{tikzpicture},\begin{tikzpicture}  \filldraw[black] (0,0) circle(2pt); \end{tikzpicture}\}$. The rct
$$
	\rrbt
$$
has six admissible subsets:
$$
	J_{11}=\{\begin{tikzpicture}  \filldraw[fill=white, draw=black,thick] (0,0) circle(2pt); \end{tikzpicture}\},
	J_{21}=\{\begin{tikzpicture}  \filldraw[fill=white, draw=black,thick] (0,0) circle(2pt); \end{tikzpicture}\},
	J_{12}=\{\begin{tikzpicture}  \filldraw[fill=white, draw=black,thick] (0,0) circle(2pt); \end{tikzpicture},
		\begin{tikzpicture}  \filldraw[fill=white, draw=black,thick] (0,0) circle(2pt); \end{tikzpicture}\},
	J_{13}=\{\begin{tikzpicture}  \filldraw[fill=white, draw=black,thick] (0,0) circle(2pt); \end{tikzpicture},
		\begin{tikzpicture}  \filldraw[black] (0,0) circle(2pt); \end{tikzpicture}\},
	J_{22}=\{\begin{tikzpicture}  \filldraw[fill=white, draw=black,thick] (0,0) circle(2pt); \end{tikzpicture},
		\begin{tikzpicture}  \filldraw[black] (0,0) circle(2pt); \end{tikzpicture}\},
	J_{14}=\{\begin{tikzpicture}  \filldraw[fill=white, draw=black,thick] (0,0) circle(2pt); \end{tikzpicture},
		\begin{tikzpicture}  \filldraw[fill=white, draw=black,thick] (0,0) circle(2pt); \end{tikzpicture},
		\begin{tikzpicture}  \filldraw[black] (0,0) circle(2pt); \end{tikzpicture}\}.
$$
A remark is in order regarding notation of the admissible subsets $J_{kl}$. The first index corresponds to the label of the white vertex being the minimal element in $J_{kl}$. The second index serves to distinguish among the possible admissible subsets that have the $k$-th white vertex as minimal element. 
The rct
$$
	\rbbt
$$
has four admissible subsets all of which share the same white vertex as minimal element:
$$
	J_{11}=\{\begin{tikzpicture}  \filldraw[fill=white, draw=black,thick] (0,0) circle(2pt); \end{tikzpicture}\},
	J_{12}=\{\begin{tikzpicture}  \filldraw[fill=white, draw=black,thick] (0,0) circle(2pt); \end{tikzpicture},
		\begin{tikzpicture}  \filldraw[black] (0,0) circle(2pt); \end{tikzpicture}\},
	J_{13}=\{\begin{tikzpicture}  \filldraw[fill=white, draw=black,thick] (0,0) circle(2pt); \end{tikzpicture},
		\begin{tikzpicture}  \filldraw[black] (0,0) circle(2pt); \end{tikzpicture}\},
	J_{14}=\{\begin{tikzpicture}  \filldraw[fill=white, draw=black,thick] (0,0) circle(2pt); \end{tikzpicture},
		\begin{tikzpicture}  \filldraw[black] (0,0) circle(2pt); \end{tikzpicture},
		\begin{tikzpicture}  \filldraw[black] (0,0) circle(2pt); \end{tikzpicture}\}.
$$
The rct
$$
	\brbrt
$$
has five admissible subsets:
$$
	J_{11}=\{\begin{tikzpicture}  \filldraw[fill=white, draw=black,thick] (0,0) circle(2pt); \end{tikzpicture}\},
	J_{21}=\{\begin{tikzpicture}  \filldraw[fill=white, draw=black,thick] (0,0) circle(2pt); \end{tikzpicture}\},
	J_{12}=\{\begin{tikzpicture}  \filldraw[fill=white, draw=black,thick] (0,0) circle(2pt); \end{tikzpicture},
		\begin{tikzpicture}  \filldraw[black] (0,0) circle(2pt); \end{tikzpicture}\},
	J_{13}=\{\begin{tikzpicture}  \filldraw[fill=white, draw=black,thick] (0,0) circle(2pt); \end{tikzpicture},
		\begin{tikzpicture}  \filldraw[fill=white, draw=black,thick] (0,0) circle(2pt); \end{tikzpicture}\},
	J_{14}=\{\begin{tikzpicture}  \filldraw[fill=white, draw=black,thick] (0,0) circle(2pt); \end{tikzpicture},
		\begin{tikzpicture}  \filldraw[black] (0,0) circle(2pt); \end{tikzpicture},
		\begin{tikzpicture}  \filldraw[fill=white, draw=black,thick] (0,0) circle(2pt); \end{tikzpicture}\}.
$$

\begin{rem}\label{BCKhopf}
The notion of admissible subsets of rcts is remotely analogous to the notion of admissible cuts on non-planar rooted trees in the context of the Butcher--Connes--Kreimer Hopf algebra. See \cite{Figueroa-Gracia-Bondia_05,Manchon_08} for reviews.
\end{rem}

\begin{defn}\label{def:admExtract}
For any rct $c^i \in C_f^{\circ}$ with $n \ge 0$ white vertices, define the set of all admissible subsets $W(c^i):=\{J_{11},\ldots,J_{1i_1},\ldots , J_{n1},\ldots,J_{ni_n}\}$. A subset $\mathcal{J} \subset W(c^i)$ is an {\it{admissible extraction}} if it consists of pairwise disjoint admissible subsets, i.e., $J_{ab},J_{uv} \in \mathcal{J}$ implies $J_{ab} \cap J_{uv} = \emptyset$ if $a \neq u$. The set of all admissible extractions of $c^i$, including the empty set as well as all vertices, $V(c^i)$, is denoted $\mathcal{W}(c^i)$.  The set of {\it{proper admissible extractions}} of $c^i$, $\mathcal{W}'(c^i)$, excludes both the empty set and $V(c^i)$.
\end{defn}

For each admissible subset, $J_{lj}$, in an admissible extraction $\mathcal{J}$, one can associate an oriented polygon inside $c^i$ by connecting the vertices of $J_{lj}$ by oriented edges according to the total order implied by $c^i$ inside each admissible subset $J_{lj}$. For instance, returning to the previous example
$$
	\brbrt
$$
the following are all possible admissible extractions:
$$
	\mathcal{J}_1=\{J_{11}\},\
	\mathcal{J}_2=\{J_{21}\},\
	\mathcal{J}_3=\{J_{12}\},\
	\mathcal{J}_4=\{J_{13}\},\
$$
$$
	\mathcal{J}_5=\{J_{14}\},\
	\mathcal{J}_6=\{J_{11},J_{21}\},\
	\mathcal{J}_7=\{J_{12},J_{21}\}.
$$
Pictorially these admissible extractions can be represented, respectively, by oriented polygons inscribed into the rct $c^i$:
\begin{equation}
\label{ex:polygons}
\begin{tikzpicture}
 \draw[thick] (0cm,0cm) circle(0.4cm);
                	\filldraw[black] (90:0.4cm) (0.08,.48)rectangle(-.08,.34);
                			\draw (75:0.60cm) node[above=-9pt]{$\hspace{-0.2cm}\phantom{l}^{i}$};
	\filldraw[black] (25:0.4cm) circle(1.7pt);
	 \draw[thick=black] (-50:0.4cm) arc (-45:360:0.1cm);
	\filldraw[fill=white, draw=black,thick] (-50:0.4cm) circle(1.7pt);
	\filldraw[black] (-130:0.4cm) circle(1.7pt);
	\filldraw[fill=white, draw=black,thick] (-200:0.4cm) circle(1.7pt);
 \end{tikzpicture}
\qquad\
\begin{tikzpicture}
 \draw[thick] (0cm,0cm) circle(0.4cm);
                	\filldraw[black] (90:0.4cm) (0.08,.48)rectangle(-.08,.34);
                			\draw (75:0.60cm) node[above=-9pt]{$\hspace{-0.2cm}\phantom{l}^{i}$};
	\filldraw[black] (25:0.4cm) circle(1.7pt);
	\filldraw[fill=white, draw=black,thick] (-50:0.4cm) circle(1.7pt);
	\filldraw[black] (-130:0.4cm) circle(1.7pt);
	 	\draw[thick=black] (-200:0.4cm) arc (-200:180:0.1cm);
	\filldraw[fill=white, draw=black,thick] (-200:0.4cm) circle(1.7pt);	
 \end{tikzpicture}
\qquad\
\begin{tikzpicture}
 \draw[thick] (0cm,0cm) circle(0.4cm);
                	\filldraw[black] (90:0.4cm) (0.08,.48)rectangle(-.08,.34);
                			\draw (75:0.60cm) node[above=-9pt]{$\hspace{-0.2cm}\phantom{l}^{i}$};
	\filldraw[black] (25:0.4cm) circle(1.7pt);	
	 	\draw[thick=black] (-50:0.4cm) to [bend left=-40]  (-130:0.4cm);
	 	\draw[thick=black] (-50:0.4cm) to [bend left=-5]  (-130:0.4cm);
	\filldraw[fill=white, draw=black,thick] (-50:0.4cm) circle(1.7pt);
	\filldraw[black] (-130:0.4cm) circle(1.7pt);
	\filldraw[fill=white, draw=black,thick] (-200:0.4cm) circle(1.7pt);
 \end{tikzpicture}
\qquad\
\begin{tikzpicture}
 \draw[thick] (0cm,0cm) circle(0.4cm);
                	\filldraw[black] (90:0.4cm) (0.08,.48)rectangle(-.08,.34);
                			\draw (75:0.60cm) node[above=-9pt]{$\hspace{-0.2cm}\phantom{l}^{i}$};
	\filldraw[black] (25:0.4cm) circle(1.7pt);
	 	\draw[thick=black] (-200:0.4cm) to [bend left=-15] (-50:0.4cm);   	
	 	\draw[thick=black] (-200:0.4cm) to [bend left=15] (-50:0.4cm);
	\filldraw[fill=white, draw=black,thick] (-50:0.4cm) circle(1.7pt);
	\filldraw[black] (-130:0.4cm) circle(1.7pt);
	\filldraw[fill=white, draw=black,thick] (-200:0.4cm) circle(1.7pt);
 \end{tikzpicture}
\qquad\
\begin{tikzpicture}
 \draw[thick] (0cm,0cm) circle(0.4cm);
                	\filldraw[black] (90:0.4cm) (0.08,.48)rectangle(-.08,.34);
                			\draw (75:0.60cm) node[above=-9pt]{$\hspace{-0.2cm}\phantom{l}^{i}$};
	\filldraw[black] (25:0.4cm) circle(1.7pt);
	 	\draw[thick=black] (-50:0.4cm) to [bend left=-15]  (-130:0.4cm);
		\draw[thick=black] (-200:0.4cm) to [bend left=15] (-50:0.4cm);
	\filldraw[fill=white, draw=black,thick] (-50:0.4cm) circle(1.7pt);
	 	\draw[thick=black] (-130:0.4cm) to [bend left=-15] (-200:0.4cm);
	\filldraw[black] (-130:0.4cm) circle(1.7pt);
	\filldraw[fill=white, draw=black,thick] (-200:0.4cm) circle(1.7pt);
 \end{tikzpicture}
\qquad\
\begin{tikzpicture}
 \draw[thick] (0cm,0cm) circle(0.4cm);
                	\filldraw[black] (90:0.4cm) (0.08,.48)rectangle(-.08,.34);
                			\draw (75:0.60cm) node[above=-9pt]{$\hspace{-0.2cm}\phantom{l}^{i}$};
	\filldraw[black] (25:0.4cm) circle(1.7pt);
		\draw[thick=black] (-50:0.4cm) arc (-45:360:0.1cm);
	\filldraw[fill=white, draw=black,thick] (-50:0.4cm) circle(1.7pt);
	\filldraw[black] (-130:0.4cm) circle(1.7pt);
	 	\draw[thick=black] (-200:0.4cm) arc (-200:180:0.1cm);
	\filldraw[fill=white, draw=black,thick] (-200:0.4cm) circle(1.7pt);
 \end{tikzpicture}
\qquad\
\begin{tikzpicture}
 \draw[thick] (0cm,0cm) circle(0.4cm);
                	\filldraw[black] (90:0.4cm) (0.08,.48)rectangle(-.08,.34);
                			\draw (75:0.60cm) node[above=-9pt]{$\hspace{-0.2cm}\phantom{l}^{i}$};
	\filldraw[black] (25:0.4cm) circle(1.7pt);
		\draw[thick=black] (-50:0.4cm) to [bend left=-40] (-130:0.4cm);
		 \draw[thick=black] (-50:0.4cm) to [bend left=-5] (-130:0.4cm);
	\filldraw[fill=white, draw=black,thick] (-50:0.4cm) circle(1.7pt);
	\filldraw[black] (-130:0.4cm) circle(1.7pt);
		\draw[thick=black] (-200:0.4cm) arc (-200:180:0.1cm);
	\filldraw[fill=white, draw=black,thick] (-200:0.4cm) circle(1.7pt);
 \end{tikzpicture}
\end{equation}
Note that each polygon starts and ends at a particular white vertex, i.e., the minimal element in the corresponding admissible subset.

Let $\mathcal{J} \in \mathcal{W}(c)$ be an admissible extraction. Observe that the order of the white vertices implies an order in $\mathcal{J}=(J_{l_1j_1},\ldots, J_{l_sj_s})$ according to the order of the minimal elements ${l_1} < \cdots < {l_s}$. The indices indicate the corresponding minimal element of the admissible subsets, that is, $0 \leq s \leq |c^i|_{\begin{tikzpicture}  \filldraw[fill=white, draw=black,thick] (0,0) circle(1pt); \end{tikzpicture}}$. Let $J_{l_1j_1}$ be an admissible subset of the rct $c^i$ with the $l_1$-th white vertex of $c^i$ as minimal element. Define $(c^i /  J_{l_1j_1})$ to be the rct with $V(c^i /  J_{l_1j_1}):=V(c^i) - J_{l_1j_1}$, that is, the rct formed by eliminating in $V(c^i)$ the vertices corresponding to the admissible subset $J_{l_1j_1}$. The rct ${c^i_{n_{l_1}}}:=(c^i /  J_{l_1j_1})_{n_{l_1}}$ is defined by replacing in $c^i$ the $l_1$-th white vertex corresponding to the minimal element of $J_{l_1j_1}$ by a black vertex (i.e.~a vertex decorated by $x_{n_{l_1}}$) and eliminating in $V(c)$ the vertices corresponding to the remaining elements in $J_{l_1j_1}$. The rct $c^{n_{l_1}}_{|J_{l_1j_1}}$ is defined in terms of the vertices of $J_{l_1j_1}$ by making the minimal element in $J_{l_1j_1}$ the new root labeled ${n_{l_1}}$, and the remaining elements of the admissible subset $J_{l_1j_1}$ are the internal vertices ordered according to the order in $J_{l_1j_1}$. Note that $|V({c^i_{n_{l_1}}})|=|V(c^i)| - |J_{l_1j_1}|+1$ and $|V(c^{n_{l_1}}_{|J_{l_1j_1}})|=|J_{l_1j_1}|$. The rct $c^{n_{l_1}}_{|J_{l_1j_1}}$ may be considered a rooted circle subtree (sub-rct) of $c^i$. Hence, to any admissible set $J_{l_1j_1}$ there corresponds a unique pair of rcts $(c^i_{n_{l_1}}, c^{n_{l_1}}_{|J_{l_1j_1}})$. This somewhat unwieldy notation becomes much clearer when working directly with the diagrammatical notation of rcts. Hence, continuing the example above, the admissible subset $J_{12}=\{\begin{tikzpicture}  \filldraw[fill=white, draw=black,thick] (0,0) circle(2pt); \end{tikzpicture},\begin{tikzpicture}  \filldraw[black] (0,0) circle(2pt); \end{tikzpicture}\}$

$$
\begin{tikzpicture}
 \draw[thick] (0cm,0cm) circle(0.4cm);
                	\filldraw[black] (90:0.4cm) (0.08,.48)rectangle(-.08,.34);
                			\draw (75:0.60cm) node[above=-9pt]{$\hspace{-0.2cm}\phantom{l}^{i}$};
	 \draw[thick=black] (-50:0.4cm) to [bend left=-40]  (-130:0.4cm);
	 \draw[thick=black] (-50:0.4cm) to [bend left=-5]  (-130:0.4cm);
	\filldraw[black] (25:0.4cm) circle(1.7pt);
	\filldraw[fill=white, draw=black,thick] (-50:0.4cm) circle(1.7pt);
	\filldraw[black] (-130:0.4cm) circle(1.7pt);
	\filldraw[fill=white, draw=black,thick] (-200:0.4cm) circle(1.7pt);
 \end{tikzpicture}
$$
gives the rcts
$$
c_{n_1}^i=\raisebox{-20pt}{
\begin{tikzpicture}{
 \draw[thick] (0cm,0cm) circle(0.4cm);
                	\filldraw[black] (90:0.4cm) (0.08,.48)rectangle(-.08,.34);
                			\draw (75:0.60cm) node[above=-9pt]{$\hspace{-0.2cm}\phantom{l}^{i}$};
	\filldraw[black] (0:0.4cm) circle(1.7pt);
	\filldraw[black] (-90:0.4cm) circle(1.7pt);
			\draw(-60:0.45cm) node[above=-14pt]{$\hspace{-0.2cm}\phantom{l}^{n_1}$};
	\filldraw[fill=white, draw=black,thick] (-180:0.4cm) circle(1.7pt);
 }\end{tikzpicture}}
 \qquad\
 c_{|J_{12}}^{n_1}= \raisebox{-10pt}{
\begin{tikzpicture}
 \draw[thick] (0cm,0cm) circle(0.4cm);
                	\filldraw[black] (90:0.4cm) (0.08,.48)rectangle(-.08,.34);
                			\draw (65:0.63cm) node[above=-9pt]{$\hspace{-0.2cm}\phantom{l}^{n_1}$};
	\filldraw[black] (-90:0.4cm) circle(1.7pt);
 \end{tikzpicture}}
$$

In the following, with the aim to simplify notation, an admissible subset of a rct is indexed only by its minimal element. If several admissible subsets with the same minimal element are considered then again pairs of indices are used, where, as before, the first index indicates the white vertex and the second index labels the particular admissible subset.

\begin{rem}\label{rmk:Physics} \hfill
\begin{enumerate}
\item
Any admissible subset $J_l$ of a rct $c^i$ will be associated to a sub-rct of $c^i$ by identifying its minimal element as the root. See \eqref{ex:polygons} for an example. The basic assumption is that two admissible subsets shall not have the same minimal element, i.e., roots. Apart from this, the following cases can occur when comparing two admissible subsets $J_k$ and $J_l$, $k \neq l$. If $k > l$  then the two sub-rcts are either strictly nested, $J_k \subset J_l$, disjoint $J_k \cap J_l = \emptyset$, or they overlap, i.e., $J_k \cap J_l \neq \emptyset$. The first case is illustrated in the rct
$$
\begin{tikzpicture}
 \draw[thick] (0cm,0cm) circle(0.4cm);
                	\filldraw[black] (90:0.4cm) (0.08,.48)rectangle(-.08,.34);
                			\draw (75:0.60cm) node[above=-9pt]{$\hspace{-0.2cm}\phantom{l}^{i}$};
	 \draw[thick=black] (0:0.4cm) to [bend left=-40]  (-40:0.4cm);
	  \draw[thick=black] (-40:0.4cm) to [bend left=-20]  (-90:0.4cm);
	    \draw[thick=black] (-90:0.4cm) to [bend left=-3]  (-180:0.4cm);
	    \draw[thick=black] (-180:0.4cm) to [bend left=20]  (-0:0.4cm);
	 \draw[thick=black] (-40:0.4cm) to [bend left=-20]  (-180:0.4cm);
	 \draw[thick=black] (-40:0.4cm) to [bend left=15]  (-180:0.4cm);
	\filldraw[black] (45:0.4cm) circle(1.7pt);
	\filldraw[fill=white, draw=black,thick] (0:0.4cm) circle(1.7pt);
		\draw (0:0.6cm) node[above=-9pt]{$\hspace{-0.2cm}\phantom{l}^{1}$};
	\filldraw[fill=white, draw=black,thick] (-40:0.4cm) circle(1.7pt);
		\draw (-40:0.6cm) node[above=-9pt]{$\hspace{-0.1cm}\phantom{l}^{2}$};
	\filldraw[fill=white, draw=black,thick] (-90:0.4cm) circle(1.7pt);
	\filldraw[black] (-140:0.4cm) circle(1.7pt);
	\filldraw[black] (-180:0.4cm) circle(1.7pt);
	\filldraw[fill=white, draw=black,thick] (-225:0.4cm) circle(1.7pt);
 \end{tikzpicture}
$$
where the admissible subset $J_2=\{\begin{tikzpicture}  \filldraw[fill=white, draw=black,thick] (0,0) circle(2pt); \end{tikzpicture}_2, \begin{tikzpicture}  \filldraw[black] (0,0) circle(2pt); \end{tikzpicture}\}$ with the second white vertex as minimal element is nested inside the admissible subset $J_1=\{\begin{tikzpicture}  \filldraw[fill=white, draw=black,thick] (0,0) circle(2pt); \end{tikzpicture}_1, \begin{tikzpicture}  \filldraw[fill=white, draw=black,thick] (0,0) circle(2pt); \end{tikzpicture}_2, \begin{tikzpicture}  \filldraw[fill=white, draw=black,thick] (0,0) circle(2pt); \end{tikzpicture}, \begin{tikzpicture}  \filldraw[black] (0,0) circle(2pt); \end{tikzpicture}\}$.
The following rct
$$
\begin{tikzpicture}
 \draw[thick] (0cm,0cm) circle(0.4cm);
                	\filldraw[black] (90:0.4cm) (0.08,.48)rectangle(-.08,.34);
                			\draw (75:0.60cm) node[above=-9pt]{$\hspace{-0.2cm}\phantom{l}^{i}$};
	 \draw[thick=black] (0:0.4cm) to [bend left=-40]  (-90:0.4cm);
	  \draw[thick=black] (-90:0.4cm) to [bend left=10]  (-225:0.4cm);
	    \draw[thick=black] (-225:0.4cm) to [bend left=-3]  (0:0.4cm);
	    \draw[thick=black] (-0:0.4cm) to [bend left=20]  (-0:0.4cm);
	 \draw[thick=black] (-40:0.4cm) to [bend left=-20]  (-180:0.4cm);
	 \draw[thick=black] (-40:0.4cm) to [bend left=15]  (-180:0.4cm);
	\filldraw[black] (45:0.4cm) circle(1.7pt);
	\filldraw[fill=white, draw=black,thick] (0:0.4cm) circle(1.7pt);
		\draw (0:0.6cm) node[above=-9pt]{$\hspace{-0.2cm}\phantom{l}^{1}$};
	\filldraw[fill=white, draw=black,thick] (-40:0.4cm) circle(1.7pt);
		\draw (-40:0.6cm) node[above=-9pt]{$\hspace{-0.1cm}\phantom{l}^{2}$};
	\filldraw[fill=white, draw=black,thick] (-90:0.4cm) circle(1.7pt);
	\filldraw[black] (-140:0.4cm) circle(1.7pt);
	\filldraw[black] (-180:0.4cm) circle(1.7pt);
	\filldraw[fill=white, draw=black,thick] (-225:0.4cm) circle(1.7pt);
 \end{tikzpicture}
 $$
illustrates the case of disjoint admissible subsets $J_1=\{\begin{tikzpicture}  \filldraw[fill=white, draw=black,thick] (0,0) circle(2pt); \end{tikzpicture}_1, \begin{tikzpicture}  \filldraw[fill=white, draw=black,thick] (0,0) circle(2pt); \end{tikzpicture}, \begin{tikzpicture}  \filldraw[fill=white, draw=black,thick] (0,0) circle(2pt); \end{tikzpicture}\}$
and
$J_2=\{\begin{tikzpicture}  \filldraw[fill=white, draw=black,thick] (0,0) circle(2pt); \end{tikzpicture}_2, \begin{tikzpicture}  \filldraw[black] (0,0) circle(2pt); \end{tikzpicture}\}$.
The {\it{overlapping}} case, i.e., $J_k \cap J_l \neq \emptyset$, is represented by the example
$$
\begin{tikzpicture}
 \draw[thick] (0cm,0cm) circle(0.4cm);
                	\filldraw[black] (90:0.4cm) (0.08,.48)rectangle(-.08,.34);
                			\draw (75:0.60cm) node[above=-9pt]{$\hspace{-0.2cm}\phantom{l}^{i}$};
	 \draw[thick=black] (0:0.4cm) to [bend left=-40]  (-90:0.4cm);
	  \draw[thick=black] (-90:0.4cm) to [bend left=1]  (-180:0.4cm);
	    \draw[thick=black] (-180:0.4cm) to [bend left=-100]  (-225:0.4cm);
	    \draw[thick=black] (-225:0.4cm) to [bend left=-3]  (0:0.4cm);
	    \draw[thick=black] (-0:0.4cm) to [bend left=20]  (-0:0.4cm);
	 \draw[thick=black] (-40:0.4cm) to [bend left=-20]  (-180:0.4cm);
	 \draw[thick=black] (-40:0.4cm) to [bend left=15]  (-180:0.4cm);
	\filldraw[black] (45:0.4cm) circle(1.7pt);
	\filldraw[fill=white, draw=black,thick] (0:0.4cm) circle(1.7pt);
		\draw (0:0.6cm) node[above=-9pt]{$\hspace{-0.2cm}\phantom{l}^{1}$};
	\filldraw[fill=white, draw=black,thick] (-40:0.4cm) circle(1.7pt);
		\draw (-40:0.6cm) node[above=-9pt]{$\hspace{-0.1cm}\phantom{l}^{2}$};
	\filldraw[fill=white, draw=black,thick] (-90:0.4cm) circle(1.7pt);
	\filldraw[black] (-140:0.4cm) circle(1.7pt);
	\filldraw[black] (-180:0.4cm) circle(1.7pt);
	\filldraw[fill=white, draw=black,thick] (-225:0.4cm) circle(1.7pt);
 \end{tikzpicture}
$$
where $J_1=\{\begin{tikzpicture}  \filldraw[fill=white, draw=black,thick] (0,0) circle(2pt); \end{tikzpicture}_1, \begin{tikzpicture}  \filldraw[fill=white, draw=black,thick] (0,0) circle(2pt); \end{tikzpicture}, \begin{tikzpicture}  \filldraw[black] (0,0) circle(2pt); \end{tikzpicture}, \begin{tikzpicture}  \filldraw[fill=white, draw=black,thick] (0,0) circle(2pt); \end{tikzpicture}\}$
and
 $J_2=\{\begin{tikzpicture}  \filldraw[fill=white, draw=black,thick] (0,0) circle(2pt); \end{tikzpicture}_2, \begin{tikzpicture}  \filldraw[black] (0,0) circle(2pt); \end{tikzpicture}\}$
share one black vertex.

\item \label{rmk:2}
The combinatorics described here are similar in flavor to those of Feynman graphs in physics. Those graphs can have subgraphs, and they can be disjoint, nested or overlapping. Feynman graphs appear in perturbative quantum field theory, and their combinatorics are crucial in the context of the solution of the so-called renormalization problem \cite{Collins_84}. The analog of an admissible extraction is called a {\it{spinney}} in the context of Feynman graphs. The analog of the set of all admissible extractions is called a {\it{wood}}, which contains all spinneys \cite{Caswell-Kennedy_82,Smirnov_91}.
\end{enumerate}
\end{rem}

Observe that for any admissible subset $J_{l}$ the weight $|c^i|= |c^i_{n_l}| + |c_{|J_{l}}^{n_l}| -1$. Regarding the degrees of $c^i_{n_l}$ and $c_{|J_{l}}^{n_l}$, it follows that $\deg(c^i_{n_l}) = |c^i| - 2 |J_l|_{\begin{tikzpicture}  \filldraw[fill=white, draw=black,thick] (0,0) circle(1.3pt); \end{tikzpicture}} - |J_l|_{\begin{tikzpicture}  \filldraw[black] (0,0) circle(1.3pt); \end{tikzpicture}} + 1$, where $|J_{l}|_{\begin{tikzpicture}  \filldraw[fill=white, draw=black,thick] (0,0) circle(1.3pt); \end{tikzpicture}}$ and $ |J_{l}|_{\begin{tikzpicture}  \filldraw[black] (0,0) circle(1.3pt); \end{tikzpicture}}$ denote the number of white and black vertices in $J_{l}$, respectively. For the extracted rct
$$
	\deg(c_{|J_{l}}^{n_l}) = 2 |J_{l}|_{\begin{tikzpicture} \filldraw[fill=white, draw=black,thick] (0,0) circle(1.3pt); \end{tikzpicture}}
			+ |J_{l}|_{\begin{tikzpicture} \filldraw[black] (0,0) circle(1.3pt);\end{tikzpicture}}
			-1
$$
such  that
\begin{equation}
\label{degreepreserving}
	\deg(c^i) =  \deg(c^i_{n_l}) + \deg(c_{|J_{l}}^{n_l}).
\end{equation}

Consider next the rct $c_{|J_{l}}^{n_l}$ corresponding to the admissible subset $J_{l}$ as a sub-rct in $c^i$. Since an admissible extraction $\mathcal{J}$ consists of strictly disjoint admissible subsets, one may extract several sub-rcts from a rct $c^i$. Hence, to any admissible extraction $\mathcal{J}:=\{J_{l_1},\ldots,J_{l_k}\}$, $l_1 < \cdots < l_k$, there corresponds a unique rct
$$
	c^i_{n_{l_1} \cdots n_{l_k}}:=(c^i /  J_{l_1} \cup \cdots \cup J_{l_k})_{n_{l_1}\cdots n_{l_k}}
$$
with $k$ white vertices replaced by black vertices decorated by $x_{n_{l_1}}, \ldots, x_{n_{l_k}}$ and a sequence of sub-rcts $c_{|J_{l_1}}^{n_{l_1}} , \ldots, c_{|J_{l_k}}^{n_{l_k}}$ with roots indexed by $n_{l_1} ,\ldots, n_{l_k} $. Returning to the earlier example, the admissible extraction $\mathcal{J}_7=\{J_{12},J_{21}\}$ corresponding to the rct with two sub-rcts
$$
\begin{tikzpicture}
 \draw[thick] (0cm,0cm) circle(0.4cm);
                	\filldraw[black] (90:0.4cm) (0.08,.48)rectangle(-.08,.34);
                			\draw (75:0.60cm) node[above=-9pt]{$\hspace{-0.2cm}\phantom{l}^{i}$};
	\draw[thick=black] (-200:0.4cm) arc (-200:180:0.1cm);
	 \draw[thick=black] (-50:0.4cm) to [bend left=-40]  (-130:0.4cm);
	 \draw[thick=black] (-50:0.4cm) to [bend left=-5]  (-130:0.4cm);
	\filldraw[black] (25:0.4cm) circle(1.7pt);
	\filldraw[fill=white, draw=black,thick] (-50:0.4cm) circle(1.7pt);
	\filldraw[black] (-130:0.4cm) circle(1.7pt);
	\filldraw[fill=white, draw=black,thick] (-200:0.4cm) circle(1.7pt);
 \end{tikzpicture}
$$
gives
$$
c^i_{n_1n_2} =\hspace{-0.4cm} \raisebox{-20pt}{
{\begin{tikzpicture}{
 \draw[thick] (0cm,0cm) circle(0.4cm);
                	\filldraw[black] (90:0.4cm) (0.08,.48)rectangle(-.08,.34);
                			\draw (75:0.60cm) node[above=-9pt]{$\hspace{-0.2cm}\phantom{l}^{i}$};
	\filldraw[black] (0:0.4cm) circle(1.7pt);
	\filldraw[black] (-90:0.4cm) circle(1.7pt);
			\draw(-60:0.5cm) node[above=-14pt]{$\hspace{-0.2cm}\phantom{l}^{n_1}$};
	\filldraw[black] (-180:0.4cm) circle(1.7pt);
			\draw(-198:0.8cm) node[above=-14pt]{$\hspace{0.1cm}\phantom{l}^{n_2}$};
 }\end{tikzpicture}} }
$$
and
{$$
c_{|J_{12}}^{n_1}= \raisebox{-10pt}{
\begin{tikzpicture}
 \draw[thick] (0cm,0cm) circle(0.4cm);
                	\filldraw[black] (90:0.4cm) (0.08,.48)rectangle(-.08,.34);
                			\draw (75:0.60cm) node[above=-9pt]{$\hspace{-0.0cm}\phantom{l}^{n_1}$};
	\filldraw[black] (-90:0.4cm) circle(1.7pt);
 \end{tikzpicture}}
\qquad
c_{|J_{21}}^{n_2}=\raisebox{-10pt}{
\begin{tikzpicture}{
 \draw[thick] (0cm,0cm) circle(0.4cm);
                	\filldraw[black] (90:0.4cm) (0.08,.48)rectangle(-.08,.34);
                			\draw (75:0.60cm) node[above=-9pt]{$\hspace{-0.0cm}\phantom{l}^{n_2}$};
 }\end{tikzpicture}}
$$}

The linear span of bicolored rooted circle trees $\mathcal{C}_f^{\circ}$ is extended to a commutative $k$-algebra $\mathcal{H}^\circ_f$ by using the disjoint union as a product, i.e., $m_{\mathcal{H}^\circ_\mathcal{C}}({c'}^i,{c''}^j)=:{c'}^i{c''}^j$. This makes  $\mathcal{H}^\circ_f$ a polynomial algebra with $C_f^{\circ}$ as a generating set. A coproduct is defined on $\mathcal{H}_f^{\circ}$ in terms of admissible extractions. Let $c^i$ be a bicolored rct, the map $\Delta: \mathcal{C}_f^{\circ} \to \mathcal{C}_f^{\circ} \otimes \mathcal{H}_f^{\circ}$ is defined by summing over all admissible extractions on $c^i$, namely,
\begin{equation}
\label{coprod}
	\Delta(c^i) := \sum_{\mathcal{J} \in {\mathcal{W}}(c^i)
					\atop \mathcal{J} :=\{J_{l_1},\ldots, J_{l_k}\}}
				c^i_{n_{l_1}\cdots n_{l_k}}
				\otimes c_{|J_{l_1}}^{n_{l_1}} \cdots c_{|J_{l_k}}^{n_{l_k}}.
\end{equation}
The coproduct is extended multiplicatively to $\mathcal{H}_f^{\circ}$. The sum on the right-hand side is over all admissible extractions of the rct $c^i$, including the empty set and the full rct. Note that it includes $k$ {\it{internal}} summations, i.e., $1 \leq n_{l_1},\cdots, n_{l_k} \leq m$. In other words, {\it Einstein's summation convention} applies so that pairs of indices are implying a summation from $1$ to $m$. On the left-hand side of the tensor product there is always a single bicolored rct with root indexed by $i$, where several internal vertices have been eliminated, and $k$ internal black vertices are indexed by $n_{l_1},\ldots, n_{l_k}$. The latter correspond to the minimal elements of the admissible subsets $J_{l_1},\ldots, J_{l_k}$, i.e., the white vertices enumerated $l_1 < \cdots <l_k$ that have been turned into black vertices. On the right-hand side of the tensor product there is a monomial of $k$ sub-rcts with roots indexed by $n_{l_1},\ldots, n_{l_k}$. Examples are in order to make this definition more transparent. First, note that the rct with no internal vertices is primitive, since there are no proper admissible subsets. That is,
$$
	\Delta \Big(\!\!\raisebox{-10pt}{\et}\Big) \; = \; \raisebox{-10pt}{\et} \; \otimes \; {\mathbf{1}}  \;
	+ \; {\mathbf{1}}  \; \otimes \raisebox{-10pt}{\et}
$$
Similarly, the rcts with only black internal vertices also has no proper admissible subsets, therefore,
$$
	\Delta \Big(\!\!\raisebox{-10pt}{
 \begin{tikzpicture}
 \draw[thick] (0cm,0cm) circle(0.4cm);
                \filldraw[black] (90:0.4cm) (0.08,.48)rectangle(-.08,.34);
                                \draw (75:0.60cm) node[above=-9pt]{$\hspace{-0.2cm}\phantom{l}^{i}$};
	 \filldraw[black] (45:0.33cm) circle(0.2pt);
	 \filldraw[black] (35:0.33cm) circle(0.2pt);
	 \filldraw[black] (25:0.33cm) circle(0.2pt);
	 \filldraw[black] (0:0.4cm) circle(1.7pt);
	
 	 \filldraw[black] (-25:0.33cm) circle(0.2pt);
	 \filldraw[black] (-35:0.33cm) circle(0.2pt);
	 \filldraw[black] (-45:0.33cm) circle(0.2pt);			
	 \filldraw[black] (-75:0.4cm) circle(1.7pt);
	 		
	 \filldraw[black] (-105:0.33cm) circle(0.3pt);
	 \filldraw[black] (-115:0.33cm) circle(0.3pt);
	 \filldraw[black] (-125:0.33cm) circle(0.3pt);
	 \filldraw[black] (-150:0.4cm) circle(1.7pt);
	 \filldraw[black] (-180:0.33cm) circle(0.3pt);
	 \filldraw[black] (-190:0.33cm) circle(0.3pt);
	 \filldraw[black] (-200:0.33cm) circle(0.3pt);
	 		
 \end{tikzpicture}}\Big) \; = \; \raisebox{-10pt}{ \begin{tikzpicture}
 \draw[thick] (0cm,0cm) circle(0.4cm);
                \filldraw[black] (90:0.4cm) (0.08,.48)rectangle(-.08,.34);
                                \draw (75:0.60cm) node[above=-9pt]{$\hspace{-0.2cm}\phantom{l}^{i}$};
	 \filldraw[black] (45:0.33cm) circle(0.2pt);
	 \filldraw[black] (35:0.33cm) circle(0.2pt);
	 \filldraw[black] (25:0.33cm) circle(0.2pt);
	 \filldraw[black] (0:0.4cm) circle(1.7pt);
	
 	 \filldraw[black] (-25:0.33cm) circle(0.2pt);
	 \filldraw[black] (-35:0.33cm) circle(0.2pt);
	 \filldraw[black] (-45:0.33cm) circle(0.2pt);			
	 \filldraw[black] (-75:0.4cm) circle(1.7pt);
	 		
	 \filldraw[black] (-105:0.33cm) circle(0.3pt);
	 \filldraw[black] (-115:0.33cm) circle(0.3pt);
	 \filldraw[black] (-125:0.33cm) circle(0.3pt);
	 \filldraw[black] (-150:0.4cm) circle(1.7pt);
	 \filldraw[black] (-180:0.33cm) circle(0.3pt);
	 \filldraw[black] (-190:0.33cm) circle(0.3pt);
	 \filldraw[black] (-200:0.33cm) circle(0.3pt);
	 		
 \end{tikzpicture}} \; \otimes \; {\mathbf{1}}  \;
	+ \; {\mathbf{1}}  \; \otimes \raisebox{-10pt}{ \begin{tikzpicture}
 \draw[thick] (0cm,0cm) circle(0.4cm);
                \filldraw[black] (90:0.4cm) (0.08,.48)rectangle(-.08,.34);
                                \draw (75:0.60cm) node[above=-9pt]{$\hspace{-0.2cm}\phantom{l}^{i}$};
	 \filldraw[black] (45:0.33cm) circle(0.2pt);
	 \filldraw[black] (35:0.33cm) circle(0.2pt);
	 \filldraw[black] (25:0.33cm) circle(0.2pt);
	 \filldraw[black] (0:0.4cm) circle(1.7pt);
	
 	 \filldraw[black] (-25:0.33cm) circle(0.2pt);
	 \filldraw[black] (-35:0.33cm) circle(0.2pt);
	 \filldraw[black] (-45:0.33cm) circle(0.2pt);			
	 \filldraw[black] (-75:0.4cm) circle(1.7pt);
	 		
	 \filldraw[black] (-105:0.33cm) circle(0.3pt);
	 \filldraw[black] (-115:0.33cm) circle(0.3pt);
	 \filldraw[black] (-125:0.33cm) circle(0.3pt);
	 \filldraw[black] (-150:0.4cm) circle(1.7pt);
	 \filldraw[black] (-180:0.33cm) circle(0.3pt);
	 \filldraw[black] (-190:0.33cm) circle(0.3pt);
	 \filldraw[black] (-200:0.33cm) circle(0.3pt);
 \end{tikzpicture}}
$$

\vspace{0.3cm}

\noindent White vertices imply admissible subsets. Consider the following examples:
$$
	\Delta \Big(\!\!\raisebox{-10pt}{\rt}\Big) \; = \; \raisebox{-10pt}{\rt} \; \otimes \; {\mathbf{1}}  \;
	+ \; {\mathbf{1}}  \; \otimes \raisebox{-10pt}{\rt} \;
	+ \; \raisebox{-20pt}{\btn} \; \otimes \raisebox{-7pt}{\etn}
$$
$$
	\Delta \Big(\!\!\raisebox{-10pt}{\brt}\Big) \; = \; \raisebox{-10pt}{\brt} \; \otimes \; {\mathbf{1}}  \;
	+ \; {\mathbf{1}}  \; \otimes \raisebox{-10pt}{\brt} \;
	+ \; \raisebox{-19pt}{\bbnt}\; \otimes\!  \raisebox{-9pt}{\etn}
$$
$$
	\Delta \Big(\!\!\raisebox{-10pt}{\rbt}\Big) \; = \; \raisebox{-10pt}{\rbt} \; \otimes \; {\mathbf{1}}  \;
	+ \; {\mathbf{1}}  \; \otimes \raisebox{-10pt}{\rbt} \;
	+ \; \raisebox{-21pt}{\btn} \; \otimes  \raisebox{-10pt}{\nbt}\;
	+ \; \raisebox{-19pt}{\bnbt}\!\! \otimes  \raisebox{-9pt}{\etn}
$$
\allowdisplaybreaks{
\begin{eqnarray*}
	\Delta \Big(\!\!\raisebox{-10pt}{\rrt}\Big) &=& \raisebox{-10pt}{\rrt} \; \otimes \; {\mathbf{1}}  \;
	+ \; {\mathbf{1}}  \; \otimes \raisebox{-10pt}{\rrt} \;
	+ \; \raisebox{-21pt}{\btn} \; \otimes  \raisebox{-10pt}{\rtn}\;
	+ \; \raisebox{-19pt}{\bnrt} \!\!\otimes  \raisebox{-9pt}{\etn}\\
	& &   + \; \raisebox{-19pt}{\rbnt} \otimes  \raisebox{-9pt}{\etk}
		+  \; \raisebox{-19pt}{\bkbnt} \otimes  \raisebox{-9pt}{\etn}\raisebox{-9pt}{\etk}
\end{eqnarray*}}
\allowdisplaybreaks{
\begin{eqnarray*}
	\Delta \Big(\!\!\raisebox{-10pt}{\brbrt}\Big) &=& \raisebox{-10pt}{\brbrt} \; \otimes \; {\mathbf{1}}  \;
	+ \; {\mathbf{1}}  \; \otimes \raisebox{-10pt}{\brbrt} \;
	+ \; \raisebox{-20pt}{\bbnbrt}\! \otimes  \raisebox{-10pt}{\etn}
		+ \!\!\! \raisebox{-10pt}{\brbbnt} \ \otimes  \raisebox{-9pt}{\etk}\\
	& & + \;\raisebox{-19pt}{\bbnrt} \; \otimes  \raisebox{-9pt}{\nbt}
	+ \; \raisebox{-21pt}{\bbnbt} \; \otimes  \raisebox{-10pt}{\rtn}
		 + \!\! \raisebox{-21pt}{\bbnt} \; \otimes  \raisebox{-10pt}{\brtn}\\
	& & +  \!\! \raisebox{-18pt}{\bbkbbnt}\!\! \otimes  \raisebox{-9pt}{\etn}\raisebox{-9pt}{\etk}	
		+ \!\! \raisebox{-21pt}{\bbnbkt} \; \otimes  \raisebox{-10pt}{\nbt}\raisebox{-10pt}{\etk}
\end{eqnarray*}}
Regarding the last example, the reader is referred to (\ref{ex:polygons}).

Note that from \eqref{degreepreserving} and the construction of $\Delta(c^i)$ each term in the sum on the right-hand side of \eqref{coprod} preserves the grading, that is, for any admissible extraction $\mathcal{J} :=\{J_{l_1},\ldots, J_{l_k}\} \in \mathcal{W}(c^i)$ one has
\begin{equation}
\label{degreepreserving2}
	 \deg(c^i) = \deg(c^i_{n_{l_1}\cdots n_{l_k}}) + \deg(c_{|J_{l_1}}^{n_{l_1}}) 
	 					+ \cdots + \deg(c_{|J_{l_k}}^{n_{l_k}}).
\end{equation}

\begin{thm}\label{thm:HA}
$\mathcal{H}_f^{\circ}$ is a connected graded commutative unital Hopf algebra.
\end{thm}

\begin{proof}
The grading of $\mathcal{H}_f^{\circ} = \bigoplus_{n \ge 0} \mathcal{H}_n^{\circ} $ is given in terms of the degree of rcts, namely, $\deg({c'}^i{c''}^j)= \deg({c'}^i) + \deg({c''}^j)$. From \eqref{degreepreserving2} it follows that the coproduct is compatible with the grading, i.e., $\Delta(c^i) \subset \oplus_{p+q=\deg(c^i)}  \mathcal{H}_p^{\circ} \otimes  \mathcal{H}_q^{\circ}$. Next coassociativity is shown, i.e., $(\Delta \otimes \id)\circ\Delta = (\id \otimes \Delta)\circ\Delta$. Some notation is introduced to make the argument more transparent. Suppressing the internal summations, the coproduct is written as
$$
	\Delta(c^i) := \sum_{\tau \subset c^i} (c^i / \tau) \otimes \tau.
$$
Here $\tau:= c_{|J_{l_1}}^{n_{l_1}} \cdots c_{|J_{l_k}}^{n_{l_k}}$ denotes the set of completely disjoint sub-rcts corresponding to an admissible extraction $\mathcal{J} =\{J_{l_1},\ldots, J_{l_k}\} \in \mathcal{W}(c^i)$, and $c_{|J_{l_p}}^{n_{l_p}} \subset c^i$ is the sub-rct corresponding to the admissible subset $J_{l_p}$ of $c^i$. The sum is taken over the set of sub-rct corresponding to $\mathcal{J} =\{J_{l_1},\ldots, J_{l_k}\} $. Extracting the set of sub-rcts $\tau$ is denoted by $ c^i / \tau :=c^i_{n_{l_1}\cdots n_{l_k}}$. This notation follows the one used in the context of Hopf algebras of rooted trees or Feynman graphs \cite{GBFV}. Any sub-rct of $c^i$ may have itself sub-rcts. Indeed, for $s>t$ and admissible subsets $J_s \subset J_t$, the two sub-rcts $c_{|J_{s}}^{n_s}$ and $c_{|J_{t}}^{n_t}$ of a rct $c^i$ for which $c_{|J_{s}}^{n_s}$ is a sub-rct of $c_{|J_{l}}^{n_t}$ yield
\begin{equation}
\label{nesting}
	(c^i / c_{|J_{s}}^{n_s})/(c_{|J_{t}}^{n_t} / c_{|J_{s}}^{n_s}) = c^i / c_{|J_{t}}^{n_t}.  		
\end{equation}
In this notation, the identity to be verified is the equality between
$$
	 (\id \otimes \Delta)\circ\Delta(c^i) = \sum_{\tau \subset \tau' \subset c^i} (c^i / \tau') \otimes (\tau' / \tau) \otimes \tau
$$
and
$$
	 (\Delta \otimes \id)\circ\Delta(c^i) = \sum_{\tau \subset c^i \atop \tau'' \subset (c^i/\tau) } ((c^i / \tau)/\tau'') \otimes \tau'' \otimes \tau.
$$
Hence, what needs to be shown is that for each $\tau \subset c^i$
$$
	\sum_{\tau \subset \tau' \subset c^i} (c^i / \tau') \otimes \tau'/ \tau
	= \sum_{ \tau'' \subset (c^i/\tau) } ((c^i / \tau)/\tau'') \otimes \tau''.
$$
The argument given here follows \cite{GBFV}. Let $\tau \subset c^i$ such that $\tau \subset \tau' \subset c^i$. Then $\tau' / \tau \subset c^i / \tau$. On the other hand, for $\tau'' \subset c^i/\tau$ there corresponds a $\tau'$ such that $\tau \subset \tau ' \subset c^i$ and $\tau' /\tau = \tau''$. Equality follows from (\ref{nesting}).
\end{proof}

Below another (indirect) proof will be presented by showing that the Hopf algebra $\mathcal{H}_X^{\circ}$ of decorated rcts is isomorphic to the output feedback Hopf algebra.

The reduced coproduct $\Delta'$ is defined as
$$
	\Delta'(c^i):= \Delta(c^i) - c^i \otimes {\mathbf{1}} - {\mathbf{1}} \otimes c^i
			=: \sum_{\mathcal{J} \in {\mathcal{W}'}(c^i)
					\atop \mathcal{J} :=\{J_{l_1},\ldots, J_{l_k}\}}
				c^i_{n_{l_1}\cdots n_{l_k}} \otimes c_{|J_{l_1}}^{n_{l_1}}\cdots c_{|J_{l_k}}^{n_{l_k}}.
$$
Recall that $\mathcal{W}'(c^i)$ is the set of proper admissible extractions, i.e., the empty admissible set $J_0:=\emptyset$ as well as the total admissible subset $J_i:=\{V(c^i)\}$ are excluded.
$\mathcal{H}_f^{\circ}$ is a {\sl{connected graded bialgebra}\/}, and therefore a Hopf algebra. The antipode $S: \mathcal{H}_f^{\circ} \to \mathcal{H}_f^{\circ}$ is given in terms of the following recursive formulas starting with $S(\mathbf{1})=\mathbf{1}$ and
\allowdisplaybreaks{
\begin{eqnarray}
	S(c^i) &=& - c^i - \sum_{\mathcal{J} \in {\mathcal{W}'}(c^i)
					\atop \mathcal{J} :=\{J_{l_1},\ldots, J_{l_k}\}}
				S(c^i_{n_{l_1}\cdots n_{l_k}}) c_{|J_{l_1}}^{n_{l_1}}\cdots c_{|J_{l_k}}^{n_{l_k}}		\label{antipod1} \\
		&=&- c^i - \sum_{\mathcal{J} \in {\mathcal{W}'}(c^i)
					\atop \mathcal{J} :=\{J_{l_1},\ldots, J_{l_k}\}}
				c^i_{n_{l_1}\cdots n_{l_k}} S(c_{|J_{l_1}}^{n_{l_1}})\cdots S(c_{|J_{l_k}}^{n_{l_k}}). \label{antipod2}		
\end{eqnarray}}
A few examples are in order:
$$
	S \Big(\!\!\raisebox{-10pt}{\et}\Big) \; = - \raisebox{-10pt}{\et}
	\qquad
	S\Big(\!\!\raisebox{-10pt}{\bbbt}\Big) \; = - \raisebox{-10pt}{\bbbt}
$$
$$
	S \Big(\!\!\raisebox{-10pt}{\rt}\Big) \; = - \raisebox{-10pt}{\rt} 	
	+\!\! \raisebox{-20pt}{\btn} \raisebox{-7pt}{\etn}
\qquad
	S \Big(\!\!\raisebox{-10pt}{\brt}\Big) \; = - \raisebox{-10pt}{\brt} \;
	+\!\!  \raisebox{-21pt}{\bbnt}\;  \raisebox{-9pt}{\etn}
$$
$$
	S \Big(\!\!\raisebox{-10pt}{\rbt}\Big) \; = - \raisebox{-10pt}{\rbt}
	+ \raisebox{-21pt}{\btn}   \raisebox{-10pt}{\nbt}\;
	+ \raisebox{-19pt}{\bnbt}\!\! \raisebox{-9pt}{\etn}
$$
\allowdisplaybreaks{
\begin{align} \label{eq:left-S-rct-example}
	S \Big(\!\!\raisebox{-10pt}{\rrt}\Big) &= - \raisebox{-10pt}{\rrt} 	
	+ \raisebox{-21pt}{\btn}  \raisebox{-10pt}{\rtn}\;
	+ \raisebox{-18pt}{\bnrt}\!\!\!\!  \raisebox{-9pt}{\etn}
	+\!\! \raisebox{-19pt}{\rbnt}  \raisebox{-9pt}{\etk}  \\
	&
		- \raisebox{-19pt}{\btn}  \raisebox{-19pt}{\bktn}  \raisebox{-7pt}{\etk}
		-\!\!  \raisebox{-19pt}{\bkbnt} \!\! \raisebox{-9pt}{\etn}  \raisebox{-9pt}{\etk} \nonumber
\end{align}}%
In light of the last calculation and the cancellation appearing in it when using \eqref{antipod1}, one may wonder whether a cancellation free formula is available. The answer is yes, and it is known as Zimmermann's forest formula as discussed next \cite{GBFV}.


\subsection{Zimmermann forest type formula for the antipode of $\mathcal{H}_f^{\circ}$}
\label{ssect:zimmermann}

Zimmermann's celebrated forest formula first appeared in the context of the renormalization problem in perturbative quantum field theory \cite{Zimmermann_69}. It provides a closed form solution to Bogoliubov's so-called counterterm recursion for Feynman graphs \cite{Caswell-Kennedy_82, Collins_84}. The work of Connes--Kreimer unveiled the Hopf algebraic structure underlying the theory of renormalization  \cite{Connes-Kreimer_98,Connes-Kreimer_00}. As a result, a forest type formula, i.e., a closed form solution for the antipode, can be given in the context of the Butcher--Connes--Kreimer Hopf algebra of rooted trees. It is defined by extending the notion of admissible cuts on rooted trees to all cuts. See \cite{Figueroa-Gracia-Bondia_05,Manchon_08} for details. From a combinatorial point of view forest type formulas are interesting since they provide a way to calculate antipodes which is free of cancellations \cite{Einziger_10,GBFV}. In control theory applications they have the potential to significantly reduce computation times \cite{Gray-et-al_MTNS14,Gray-et-al_AUTOxx,Gray-et-al_CLCA13}.

In this section, a forest type formula for the antipode of the Hopf algebra of rooted circle trees is presented. It also provides a closed form solution for the recursions (\ref{antipod2}), which is moreover free of cancellations.  It is defined by replacing admissible extractions by all extractions. However, the situation is somewhat involved due to the internal summations over the decorations $n_{l_1}, \ldots, n_{l_k}$ associated with extracting admissible subsets from a rct in the coproduct
\begin{equation}
\label{cop}
	\Delta(c^i) = \sum_{\mathcal{J} \in {\mathcal{W}}(c^i)
					\atop \mathcal{J} :=\{J_{l_1},\ldots, J_{l_k}\}}
				c^i_{n_{l_1}\cdots n_{l_k}}
				\otimes c_{|J_{l_1}}^{n_{l_1}}\cdots c_{|J_{l_k}}^{n_{l_k}}.
\end{equation}
Recall that in the definition of an admissible extraction $\mathcal{J} \in \mathcal{W}(c^i)$ of a rct $c^i$ all admissible subsets of $\mathcal{J}$ have to be strictly disjoint. This constraint is now relaxed by allowing for extractions $\mathcal{I}=\{J_{l_1},\ldots,J_{l_k}\}$, ${l_1} < \cdots < {l_k}$, which may contain admissible subsets that are either nested, $J_{l_a} \subset J_{l_b}$ for $a>b$, or disjoint, $J_{l_a} \cap J_{l_b} = \emptyset$ for $a \neq b$. The set of {\em all} extractions of a rct $c^i$ is denoted $\mathcal{E}(c^i)$. Note that $\mathcal{W}(c^i) \subseteq \mathcal{E}(c^i)$. As an example, consider the rct
$$
 \begin{tikzpicture}
 \draw[thick] (0cm,0cm) circle(0.4cm);
                \filldraw[black] (90:0.4cm) (0.08,.48)rectangle(-.08,.34);
	                                \draw (75:0.60cm) node[above=-9pt]{$\hspace{-0.2cm}\phantom{l}^{i}$};
	  \filldraw[fill=white, draw=black,thick] (-45:0.4cm) circle(1.7pt);
	  	\draw (-45:0.6cm) node[above=-9pt]{$\hspace{-0.1cm}\phantom{l}^{1}$};
	  \filldraw[fill=white, draw=black,thick] (-135:0.4cm) circle(1.7pt);
	  	\draw (-135:0.6cm) node[above=-9pt]{$\hspace{-0.1cm}\phantom{l}^{2}$};
 \end{tikzpicture}
$$
and its admissible subsets:
$$
	J_{11}=\{\begin{tikzpicture}  \filldraw[fill=white, draw=black,thick] (0,0) circle(2pt); \end{tikzpicture}_1\},		
	J_{12}=\{\begin{tikzpicture}  \filldraw[fill=white, draw=black,thick] (0,0) circle(2pt); \end{tikzpicture}_1,
			\begin{tikzpicture}  \filldraw[fill=white, draw=black,thick] (0,0) circle(2pt); \end{tikzpicture}_2\},
	J_{21}=\{\begin{tikzpicture}  \filldraw[fill=white, draw=black,thick] (0,0) circle(2pt); \end{tikzpicture}_2\}.
$$
Then $\mathcal{E}(c^i)$ contains, among others, the extractions $\mathcal{I}_1=\{J_{12},J_{21}\}$ with $J_{21} \subset J_{12}$. Another example is the rct
$$
 \begin{tikzpicture}
 \draw[thick] (0cm,0cm) circle(0.4cm);
                	\filldraw[black] (90:0.4cm) (0.08,.48)rectangle(-.08,.34);
                			\draw (75:0.60cm) node[above=-9pt]{$\hspace{-0.2cm}\phantom{l}^{i}$};
	\filldraw[fill=white, draw=black,thick] (0:0.4cm) circle(1.7pt);
	                			\draw (0:0.6cm) node[above=-9pt]{$\hspace{-0.1cm}\phantom{l}^{1}$};
	\filldraw[black] (-90:0.4cm) circle(1.7pt);
	\filldraw[fill=white, draw=black,thick] (-180:0.4cm) circle(1.7pt);
	                			\draw (-180:0.6cm) node[above=-9pt]{$\hspace{-0.1cm}\phantom{l}^{2}$};
 \end{tikzpicture}
$$
and its admissible subsets:
$$
	J_{11}=\{\begin{tikzpicture}  \filldraw[fill=white, draw=black,thick] (0,0) circle(2pt); \end{tikzpicture}_1\},
	J_{12}=\{\begin{tikzpicture}  \filldraw[fill=white, draw=black,thick] (0,0) circle(2pt); \end{tikzpicture}_1,
		\begin{tikzpicture}  \filldraw[black] (0,0) circle(2pt); \end{tikzpicture}\},
	J_{13}=\{\begin{tikzpicture}  \filldraw[fill=white, draw=black,thick] (0,0) circle(2pt); \end{tikzpicture}_1,
		\begin{tikzpicture}  \filldraw[fill=white, draw=black,thick] (0,0) circle(2pt); \end{tikzpicture}_2\},
	J_{14}=\{\begin{tikzpicture}  \filldraw[fill=white, draw=black,thick] (0,0) circle(2pt); \end{tikzpicture}_1,
		\begin{tikzpicture}  \filldraw[black] (0,0) circle(2pt); \end{tikzpicture},
		\begin{tikzpicture}  \filldraw[fill=white, draw=black,thick] (0,0) circle(2pt); \end{tikzpicture}_2\},
	J_{21}=\{\begin{tikzpicture}  \filldraw[fill=white, draw=black,thick] (0,0) circle(2pt); \end{tikzpicture}_2\}.
$$
Then $\mathcal{E}(c^i)$ contains, among others, the extractions:
$$
	\mathcal{I}_1=\{J_{14},J_{21}\},
	\qquad
	\mathcal{I}_2=\{J_{13},J_{21}\}
$$
which are nested, i.e., $J_{21} \subset J_{14}$ and $J_{21} \subset J_{13}$. In the rct
\begin{equation}
\label{exp:rct}
 \begin{tikzpicture}
 \draw[thick] (0cm,0cm) circle(0.4cm);
                	\filldraw[black] (90:0.4cm) (0.08,.48)rectangle(-.08,.34);
                			\draw (75:0.60cm) node[above=-9pt]{$\hspace{-0.2cm}\phantom{l}^{i}$};
	\filldraw[fill=white, draw=black,thick] (0:0.4cm) circle(1.7pt);
	                			\draw (0:0.6cm) node[above=-9pt]{$\hspace{-0.1cm}\phantom{l}^{1}$};
	\filldraw[fill=white, draw=black,thick] (-90:0.4cm) circle(1.7pt);
	                			\draw (-90:0.65cm) node[above=-9pt]{$\hspace{-0.1cm}\phantom{l}^{2}$};
	\filldraw[fill=white, draw=black,thick] (-180:0.4cm) circle(1.7pt);
	                			\draw (-180:0.6cm) node[above=-9pt]{$\hspace{-0.1cm}\phantom{l}^{3}$};
 \end{tikzpicture}
\end{equation}
one finds, among others, the extractions:
\allowdisplaybreaks{
\begin{eqnarray*}
	\mathcal{I}_1&=&\{
		J_1=\{\begin{tikzpicture}  \filldraw[fill=white, draw=black,thick] (0,0) circle(2pt); \end{tikzpicture}_1,
		\begin{tikzpicture}  \filldraw[fill=white, draw=black,thick] (0,0) circle(2pt); \end{tikzpicture}_2,
		\begin{tikzpicture}  \filldraw[fill=white, draw=black,thick] (0,0) circle(2pt); \end{tikzpicture}_3\},
		J_2=\{\begin{tikzpicture} \filldraw[fill=white, draw=black,thick] (0,0) circle(2pt); \end{tikzpicture}_2,
  		\begin{tikzpicture} \filldraw[fill=white, draw=black,thick] (0,0) circle(2pt); \end{tikzpicture}_3\},
		J_3=\{\begin{tikzpicture} \filldraw[fill=white, draw=black,thick] (0,0) circle(2pt); \end{tikzpicture}_3\}
\}\\
	\mathcal{I}_2&=&\{		
		J_1=\{\begin{tikzpicture}  \filldraw[fill=white, draw=black,thick] (0,0) circle(2pt); \end{tikzpicture}_1,
		\begin{tikzpicture}  \filldraw[fill=white, draw=black,thick] (0,0) circle(2pt); \end{tikzpicture}_2,
		\begin{tikzpicture}  \filldraw[fill=white, draw=black,thick] (0,0) circle(2pt); \end{tikzpicture}_3\},
		J_2=\{\begin{tikzpicture} \filldraw[fill=white, draw=black,thick] (0,0) circle(2pt); \end{tikzpicture}_2\},
		J_3=\{\begin{tikzpicture} \filldraw[fill=white, draw=black,thick] (0,0) circle(2pt); \end{tikzpicture}_3\}
\}\\		
	\mathcal{I}_3&=&\{		
		J_1=\{\begin{tikzpicture}  \filldraw[fill=white, draw=black,thick] (0,0) circle(2pt); \end{tikzpicture}_1,
		\begin{tikzpicture}  \filldraw[fill=white, draw=black,thick] (0,0) circle(2pt); \end{tikzpicture}_3\},
		J_2=\{\begin{tikzpicture} \filldraw[fill=white, draw=black,thick] (0,0) circle(2pt); \end{tikzpicture}_2\},
		J_3=\{\begin{tikzpicture} \filldraw[fill=white, draw=black,thick] (0,0) circle(2pt); \end{tikzpicture}_3\}
\}\\
	\mathcal{I}_4&=&\{		
		J_1=\{\begin{tikzpicture}  \filldraw[fill=white, draw=black,thick] (0,0) circle(2pt); \end{tikzpicture}_1\},
		J_2=\{\begin{tikzpicture} \filldraw[fill=white, draw=black,thick] (0,0) circle(2pt); \end{tikzpicture}_2\},
		J_3=\{\begin{tikzpicture} \filldraw[fill=white, draw=black,thick] (0,0) circle(2pt); \end{tikzpicture}_3\}
\}.		
\end{eqnarray*}}%
For notational simplicity the second indices of the admissible subsets have been omitted.

Elements in an extraction $\mathcal{I}:=\{J_{l_1},\ldots, J_{l_k}\}$ are either disjoint or strictly nested. This implies a natural poset structure on each extraction, which is represented in terms of a non-planar rooted tree. The tree reflects the resulting hierarchy of admissible subsets of a rct $c^i$. Recall that by definition a non-planar {\it{rooted tree}} $t$ is made out of vertices and nonintersecting oriented edges such that all but one vertex have exactly one outgoing line and an arbitrary number of incoming lines. The particular vertex with no outgoing line is called the root of the tree. It is drawn on top of the tree, whereas the leaves are the only vertices without any incoming lines. The weight of a rooted tree $t$ is defined to be its number of vertices, $\sharp(t):=|V(t)|$. The following list contains rooted trees up to weight four:
$$
\begin{tikzpicture}  \filldraw[green] (0,0) circle(2.5pt); \end{tikzpicture} \qquad
	\arbrea \qquad \arbreb \qquad \arbrec \qquad
	\arbree \qquad \arbref \qquad \arbred \qquad \arbreg
$$
The root has been colored green. The empty rooted tree is denoted by $\mathbf{1}$. The set (resp.~space) of rooted trees is denoted by $T$ (resp.~$\mathcal{T}$). (See \cite{Manchon_08} for details.) Recall the map $B_+$ defined on rooted trees. It maps rooted trees $t_1,\ldots,t_k$  to the rooted tree $t:=B_+(t_1,\ldots,t_k)$ by connecting each of the roots of $t_1,\ldots,t_k$ via an edge to a new root. The tree with only one vertex, i.e., the root, and no edges is written $\begin{tikzpicture}  \filldraw[green] (0,0) circle(2.5pt); \end{tikzpicture} =B_+(\mathbf{1})$. The rest of the above trees can then be written $B_+(B_+(\mathbf{1}))$, $B_+(B_+(\mathbf{1}),B_+(\mathbf{1}))$, $B_+(B_+(B_+(\mathbf{1})))$, etc. Next, rooted trees are decorated in such a way that they represent the hierarchy of admissible subsets in an extraction of any rct. This hierarchy reflects the internal summations in the coproduct \eqref{cop} on rcts. Let $c^i$ be a rct and $\mathcal{I}:=\{J_{l_1},\ldots, J_{l_k}\} \in \mathcal{E}(c^i)$ an extraction. Define a map $h: T \to \{i,l_1,\ldots,l_k\}$ by decorating the root of $t \in T$ by $i$. An edge in $t$ goes out from a vertex decorated by $l_a$ to a vertex decorated by $l_b < l_a$ if and only if $J_{l_a} \subset J_{l_b}$. Regarding example (\ref{exp:rct}), the following rooted trees are associated with the extractions $\mathcal{I}_1$, $\mathcal{I}_2$, $\mathcal{I}_3, \mathcal{I}_4$,
$$
	\arbreci \qquad \arbregi \qquad \arbreei \qquad \arbredi
$$
respectively. Again, with $B^i_+(\mathbf{1}):=\begin{tikzpicture}  \filldraw[green] (0,0) circle(2.5pt); \end{tikzpicture}^i$ denoting the root with no edges and decorated by $i$, any decorated rooted tree can be written using the map $B_+^n$. It adds a new root decorated by $n$. The previous list of decorated rooted trees then becomes $t_1^i=B_+^i(B_+^1(B_+^2(B_+^3(\mathbf{1}))))$, $t_2^i=B_+^i(B_+^1(B_+^2(\mathbf{1}),B_+^3(\mathbf{1})))$, $t_3^i=B_+^i(B_+^1(B_+^3(\mathbf{1})),B_+^2(\mathbf{1}))$, and $t_4^i = B_+^i(B_+^1(\mathbf{1}),B_+^2(\mathbf{1}),B_+^3(\mathbf{1}))$. Note that rooted trees corresponding to admissible extractions  $\mathcal{J}:=\{J_{l_1},\ldots, J_{l_k}\}$ are strictly of the form $B_+^i(B_+^{l_1}(\mathbf{1}),B_+^{l_2}(\mathbf{1}),\ldots,B_+^{l_k}(\mathbf{1}))$.

Using extractions, together with decorated rooted trees to encode the hierarchical structure of extractions of a rct, the antipode $S$ of $c^i \in \mathcal{H}_f^{\circ}$ is given in the following theorem. The decorated rooted tree corresponding to an extraction $\mathcal{I}=\{J_{l_{s_1}},\ldots,J_{l_{s_b}}\} \in \mathcal{E}(c^i)$ is denoted $t^i_\mathcal{I}:=B_+^i(t^{l_{s_1}}, \cdots , t^{l_{s_b}})$.

\begin{thm} \label{thm:forestformula}
For any $c^i\in \mathcal{H}_f^\circ$ the antipode
\begin{equation}
\label{antipode3}
		S(c^i) = \sum_{\mathcal{I} \in \tilde{\mathcal{E}}(c^i) \atop \mathcal{I}:=\{J_{l_1},\ldots, J_{l_k}\}}
							  \chi(t^i_\mathcal{I}),
\end{equation}
where $t^i_\mathcal{I}$ is the decorated rooted tree corresponding to the hierarchy of admissible subsets underlying the extraction $\mathcal{I}=\{J_{l_1},\ldots, J_{l_k}\}$. The set $\tilde{\mathcal{E}}(c^i) \subset \mathcal{E}(c^i)$ must not include $J_i=V(c^i)$. The map $\chi$ is defined on the decorated rooted tree $t^i_\mathcal{I}:=B^i_+(t^{l_{s_1}}, \cdots , t^{l_{s_b}})$ as
\begin{equation}
\label{treeExtractions}
	\chi(t^i_\mathcal{I}):= - (c^i_{|J_i}/J_{l_{s_1}} \cup \cdots \cup J_{l_{s_b}})_{n_{l_{s_1}} \cdots n_{l_{s_b}}}
	\prod_{a=1}^b \chi(t^{n_{l_{s_a}}}_\mathcal{I}),
\end{equation}
and $\chi(\mathbf{1})=1$. Here $\{l_{s_1},\ldots,l_{s_b}\} \subset \{l_{1},\ldots,l_{k}\}$ and $J_{l_{s_1}}, \ldots, J_{l_{s_b}}$ are pairwise disjoint.
\end{thm}

The admissible subsets $J_i=V(c^i)$ and $J_0=\emptyset$ correspond to $c^i=c^i_{|J_i}$ and $\mathbf{1}=c^0_{|J_0}$, respectively. Note that the empty extraction $\mathcal{I}=\{J_0\}$ corresponds to the rooted tree $t^i_\mathcal{I}=B^i_+(\mathbf{1}):=\begin{tikzpicture}  \filldraw[green] (0,0) circle(2.5pt); \end{tikzpicture}^i$. Hence, for a primitive rct $c^i$
$$
	S(c^i) =\chi(t^i_{\{J_0\}}) = - (c^i_{|J_i}/J_{0}) \chi(\mathbf{1}) = -c^i.
$$
As another example, consider the rooted tree $t^i_{\mathcal{I}_3}:=B_+^i(B_+^1(B_+^3(\mathbf{1})),B_+^2(\mathbf{1}))$ corresponding to the extraction $\mathcal{I}_3=\{		
		J_1=\{\begin{tikzpicture}  \filldraw[fill=white, draw=black,thick] (0,0) circle(2pt); \end{tikzpicture}_1,
		\begin{tikzpicture}  \filldraw[fill=white, draw=black,thick] (0,0) circle(2pt); \end{tikzpicture}_3\},
		J_2=\{\begin{tikzpicture} \filldraw[fill=white, draw=black,thick] (0,0) circle(2pt); \end{tikzpicture}_2\},
		J_3=\{\begin{tikzpicture} \filldraw[fill=white, draw=black,thick] (0,0) circle(2pt); \end{tikzpicture}_3\}\}$ in example (\ref{exp:rct}). Then
\allowdisplaybreaks{
\begin{eqnarray*}	
	\chi(t^i_{\mathcal{I}_3}) &=& - (c^i_{|J_i}/J_{1} \cup J_{2})_{n_{1} n_{2}} \prod_{a=1}^2 \chi(t^{n_a}_\mathcal{I}) \\
					  &=&  (c^i_{|J_i}/J_{1} \cup J_{2})_{n_{1} n_{2}}
					  	  (c^{n_1}_{|J_1}/J_{3})_{n_{3}}\ (c^{n_1}_{|J_1})_{|J_{3}}^{n_{3}}\ c^{n_2}_{|J_2}.
\end{eqnarray*}}%
Again, this gets much easier by working diagrammatically. Indeed, observe that $\mathcal{I}_3$ corresponds to
$$
	\raisebox{-20pt}{\begin{tikzpicture}
 \draw[thick] (0cm,0cm) circle(0.4cm);
                	\filldraw[black] (90:0.4cm) (0.08,.48)rectangle(-.08,.34);
                			\draw (75:0.60cm) node[above=-9pt]{$\hspace{-0.2cm}\phantom{l}^{i}$};
			\draw[thick=black] (0:0.4cm) to [bend left=-44] (-180:0.4cm);
			\draw[thick=black] (0:0.4cm) to [bend left=44] (-180:0.4cm);
			\draw[thick=black] (-180:0.4cm) arc (-180:360:0.08cm);
			\draw[thick=black] (-90:0.4cm) arc (-90:290:0.08cm) ;
	\filldraw[fill=white, draw=black,thick] (0:0.4cm) circle(1.7pt);
	                			\draw (0:0.6cm) node[above=-9pt]{$\hspace{-0.1cm}\phantom{l}^{1}$};
	\filldraw[fill=white, draw=black,thick] (-90:0.4cm) circle(1.7pt);
	                			\draw (-90:0.65cm) node[above=-9pt]{$\hspace{-0.1cm}\phantom{l}^{2}$};
	\filldraw[fill=white, draw=black,thick] (-180:0.4cm) circle(1.7pt);
	                			\draw (-180:0.6cm) node[above=-9pt]{$\hspace{-0.1cm}\phantom{l}^{3}$};	
 \end{tikzpicture}}
\;\; \qquad\;\; \raisebox{-20pt}{\arbreei} =B_+^i(B_+^1(B_+^3(\mathbf{1})),B_+^2(\mathbf{1}))
$$
such that
$$
	\chi(t^i_{\mathcal{I}_3}) = \raisebox{-19pt}{\bkbnt} \!\!
\raisebox{-21pt}{
		\begin{tikzpicture}
 			\draw[thick] (0cm,0cm) circle(0.4cm);
                			\filldraw[black] (90:0.4cm) (0.08,.48)rectangle(-.08,.34);
                       	\draw (75:0.60cm) node[above=-9pt]{$\hspace{-0.0cm}\phantom{l}^{n_1}$};
	 			\filldraw[black] (-90:0.4cm) circle(1.7pt);
                       	\draw (72:0.64cm) node[above=-42pt]{$\hspace{-0.1cm}\phantom{l}^{n_3}$};
 		\end{tikzpicture}}
	\raisebox{-9pt}{	\begin{tikzpicture}
 				\draw[thick] (0cm,0cm) circle(0.4cm);
                				\filldraw[black] (90:0.4cm) (0.08,.48)rectangle(-.08,.34);
                			\draw (75:0.60cm) node[above=-9pt]{$\hspace{-0.0cm}\phantom{l}^{n_3}$};
 				\end{tikzpicture}}
	\raisebox{-9pt}{\etk}
$$
Note that the use of rooted trees solely reflects the need to keep track of the hierarchy of admissible subsets in an extraction, which is reflected in the extra summations in the coproduct (\ref{cop}).  The antipode of the rct \eqref{exp:rct} consists of 26 terms corresponding to all possible extractions, including those mentioned above.
\begin{proof}
For any $c^i\in \mathcal{H}_f^\circ$ define
\begin{equation}
\label{Sz}
		S_Z(c^i) = \sum_{\mathcal{I} \in \tilde{\mathcal{E}}(c^i) \atop \mathcal{I}:=\{J_{l_1},\ldots, J_{l_k}\}}
							  \chi(t^i_\mathcal{I}),
\end{equation}
where
$$
		\chi(t^i_\mathcal{I}):=  - (c^i_{|J_i}/J_{l_{s_1}} \cup \cdots \cup J_{l_{s_b}})_{n_{l_{s_1}} \cdots n_{l_{s_b}}}
		\prod_{a=1}^b \chi(t^{n_{l_{s_a}}}_\mathcal{I}).
$$
The goal is to show that the map $S_Z$ satisfies the recursion
$$
	S_Z(c^i) = - c^i - \sum_{\mathcal{J} \in {\mathcal{W}'}(c^i)
					\atop \mathcal{J} :=\{J_{l_1},\ldots, J_{l_k}\}}
				c^i_{n_{l_1}\cdots n_{l_k}}
				S_Z(c_{|J_{l_1}}^{n_{l_1}}) \cdots S_Z(c_{|J_{l_k}}^{n_{l_k}}).
$$
The uniqueness of the antipode then implies that $S=S_Z$. Recall that $\mathcal{W}'(c^i)$ is the set of proper admissible extractions of the rct $c^i$, i.e., it contains neither $J_0$ nor $J_i$. In addition, recall that the admissible extractions are a subset of all extractions, i.e., $\mathcal{W}(c^i) \subseteq {\mathcal{E}}(c^i)$. This allows one to break the sum (\ref{Sz}) into pieces such that
\begin{equation}
\label{splitting}
		S_Z(c^i) = 	- c^i_{|J_i} + \sum_{\mathcal{J} \in \mathcal{W}'(c^i) \atop \mathcal{J}:=\{J_{l_1},\ldots, J_{l_k}\}}
							  \chi(t^i_\mathcal{J})
				+ \sum_{\mathcal{I} \in {\mathcal{E}}(c^i)\backslash \mathcal{W}(c^i) \atop \mathcal{I}:=\{J_{a_1},\ldots, J_{a_s}\}}
							  \chi(t^i_\mathcal{I}).			
\end{equation}
The first term on the right-hand side corresponds to the admissible subset $J_0=\emptyset$. The first sum on the right-hand side is over strictly pairwise disjoint subsets $\mathcal{J}:=\{J_{l_1},\ldots, J_{l_k}\}$, that is, proper admissible extractions. Recall that those extractions correspond to rooted trees of the form $t^i_\mathcal{J} = B_+^i(B_+^{l_1}(\mathbf{1}),\ldots ,B_+^{l_k}(\mathbf{1}))$. Therefore,
$$
	\chi(t^i_\mathcal{J}):= - (c^i_{|J_i}/J_{l_{1}} \cup \cdots \cup J_{l_{k}})_{n_{l_{1}} \cdots n_{l_{k}}}
	\prod_{u=1}^k c_{|J_{l_{u}}}^{n_{l_{u}}}.
$$
The second sum on the right-hand side of (\ref{splitting}) is over extractions $\mathcal{I}:=\{J_{a_1},\ldots, J_{a_s}\}$, which contain at least one pair of admissible subsets, $J_{a_p}, J_{a_q}$, $a_p \neq a_q$, such that either $J_{a_p} \subset J_{a_q}$ or $J_{a_q} \subset J_{a_p}$. Assume that $J_s$ is an admissible subset of the rct $c^i$ with only one white vertex, i.e., its minimal element. This implies that the corresponding rct $c^{n_s}_{|J_{s}}$ is primitive. Therefore, the rooted tree encoding the nesting is simply given by $B^i_+(B^s_+(\mathbf{1}))$. Considering $J_s$ as an admissible extraction $\mathcal{J}:=\{J_s\}$ of $c^i$ yields a single term in $S_Z$
$$
	S_Z(c^i) =- c^i_{|J_i} +  \cdots  + (c^i_{|J_i}/J_{s})_{n_s} c^{n_s}_{|J_{s}} + \cdots,
$$
which can be written
$$
	S_Z(c^i) =- c^i_{|J_i} +  \cdots  - (c^i_{|J_i}/J_{s})_{n_s} S_Z(c^{n_s}_{|J_{s}}) + \cdots,
$$
since $c^{n_s}_{|J_{s}}$ is primitive. Assume now that $J_s$ contains more than one white vertex, which implies that the rct $c^{n_s}_{|J_{s}}$ corresponding to the admissible subset $J_s$ has itself non-trivial sub-rcts. This means that there exist extractions $\mathcal{I}=\{J_{s},J_{l_1},\ldots, J_{l_k}\} \in \mathcal{E}(c^i)\backslash \mathcal{W}(c^i)$ such that $J_{l_a} \subset J_s$ for all $a=1,\ldots,k$, i.e., $\mathcal{I}':=\{J_{l_1},\ldots, J_{l_k}\} \in \mathcal{E}'(c^{n_s}_{|J_{s}})$, where $\mathcal{E}'(c^{n_s}_{|J_{s}})$ are all proper extractions of $c^{n_s}_{|J_{s}}$. This yields
\allowdisplaybreaks{
\begin{eqnarray*}
	S_Z(c^i) &=&- c^i_{|J_i} +  \cdots  + (c^i_{|J_i}/J_{s})_{n_s} c^{n_s}_{|J_{s}} +
					\sum_{\mathcal{I} \in \mathcal{E}(c^i)\backslash \mathcal{W}(c^i) \atop{
	\mathcal{I}:=\{J_s,J_{l_1},\ldots, J_{l_k}\} \atop J_{l_a}\subset J_s, a=1,\ldots,k}}  \chi(t^{i}_{\mathcal{I}})+ \cdots\\
	&=&- c^i_{|J_i} +  \cdots  + (c^i_{|J_i}/J_{s})_{n_s} c^{n_s}_{|J_{s}} +
					\sum_{\mathcal{I}' \in \mathcal{E}'(c^{n_s}_{|J_{s}}) \atop \mathcal{I}':=\{J_{l_1},\ldots, J_{l_k}\}} -(c^i_{|J_i}/J_{s})_{n_s} \chi(t^{n_s}_{\mathcal{I}'})+ \cdots,
\end{eqnarray*}}
which can be written
\begin{eqnarray*}
	S_Z(c^i) &=&- c^i_{|J_i} +  \cdots  - (c^i_{|J_i}/J_{s})_{n_s} \big( - c^{n_s}_{|J_{s}} +
					\sum_{\mathcal{I}' \in \mathcal{E}'(c^{n_s}_{|J_{s}})
 \atop \mathcal{I}':=\{J_{l_1},\ldots, J_{l_k}\}} \chi(t^{n_s}_{\mathcal{I}'}) \big) + \cdots\\	
			&=&- c^i_{|J_i} +  \cdots  - (c^i_{|J_i}/J_{s})_{n_s}
					\sum_{\mathcal{I}' \in \tilde{\mathcal{E}}(c^{n_s}_{|J_{s}}) \atop \mathcal{I}':=\{J_{l_1},\ldots, J_{l_k}\}} \chi(t^{n_s}_{\mathcal{I}'}) + \cdots\\							 
			&=& 	- c^i_{|J_i} +  \cdots - (c^i_{|J_i}/J_{s})_{n_s} S_Z(c^{n_s}_{|J_{s}})	+ \cdots.	
\end{eqnarray*}
Recall that Einstein's convention is in place, which implies a summation over pairs of indices. For an arbitrary proper admissible extraction $\mathcal{I}:=\{J_{l_1},\ldots, J_{l_k}\} \in \mathcal{W}'(c^i)$ this generalizes to
$$
	S_Z(c^i) =- c^i_{|J_i} +  \cdots  - (c^i_{|J_i}/J_{l_1} \cup \ldots  \cup J_{l_k})_{n_{l_1} \cdots n_{l_k}} S_Z(c^{n_{l_1}}_{|J_{{l_1}}}) \cdots S_Z(c^{n_{l_k}}_{|J_{{l_k}}})	+ \cdots,
$$
which then implies that
$$
	S_Z(c^i) = - c^i - \sum_{\mathcal{J} \in {\mathcal{W}'}(c^i)
					\atop \mathcal{J} :=\{J_{l_1},\ldots, J_{l_k}\}}
				c^i_{n_{l_1}\cdots n_{l_k}}
				S_Z(c_{|J_{l_1}}^{n_{l_1}}) \cdots S_Z(c_{|J_{l_k}}^{n_{l_k}}).
$$
\end{proof}

Calculations of the antipode using formula (\ref{antipod1}) show non-trivial cancellations of terms (see Table~\ref{tbl:cancellations-left-antipode}). On the other hand, the forest formula in Theorem~\ref{thm:forestformula} is free of such cancellations. Indeed, this follows from the observation that each extraction $\mathcal{I}= \{J_{l_1},\ldots, J_{l_k}\} \in \tilde{\mathcal{E}}(c^i)$ corresponds to a monomial of $k$ rcts in formula (\ref{antipode3}). Its sign depends solely on the number of elements in $\mathcal{I}$. Cancellations can only occur between monomials of the same number of rcts. Hence, the antipode formula  \eqref{antipode3} is free of such cancellations.

\begin{table}[t]
\caption{Total number of terms (with and without multiplicities) and
cancellations when antipode formula (\ref{antipod1}) is used and $m=1$}
\label{tbl:cancellations-left-antipode}
\begin{tabular}{|c|c|c|c|c|} \hline
coordinate&coordinate & total number & total number & number of \\
map degree&map antipode&  of unique terms & of terms &cancellations \\ \hline
3&$S \Big(\!\!\raisebox{-10pt}{\rt}\Big)$ & 2 & 2 & 0 \\
5&$S \Big(\!\!\raisebox{-10pt}{\rrt}\Big)$ & 6 & 6 & 1 \\
7&$S \Big(\!\!\raisebox{-10pt}{\rrrt}\Big)$ & 17 & 26 & 9 \\
9&$S \Big(\!\!\raisebox{-10pt}{\fourrt}\Big)$ & 50 & 150 & 70 \\
11&$S \Big(\!\!\raisebox{-10pt}{\fivert}\Big)$ & 139 & 1082 & 427 \\
13&$S \Big(\!\!\raisebox{-10pt}{\sixrt}\Big)$ & 390 & 9366 & 2417 \\
15&$S \Big(\!\!\raisebox{-10pt}{\sevenrt}\Big)$ & 1059 & 94,586 & 12,730
\\[0.1in] \hline
\end{tabular}
\end{table}


\subsection{Pre-Lie algebra of rooted circle trees}
\label{ssect:prct}

A {\sl right pre-Lie algebra\/} is a $k$-vector space $P$ with a bilinear composition $\lhd : P \times P \to P$ that satisfies the right pre-Lie identity
\begin{eqnarray}
    (a \lhd b) \lhd c - a \lhd (b \lhd  c) &=& (a \lhd c) \lhd b - a \lhd (c \lhd b),
    \label{prelie}
\end{eqnarray}
for $a,b,c \in P$. As a consequence, the bracket defined by $[a,b]:=a\lhd b-b\lhd a$ satisfies the Jacobi identity.  {\sl Left pre-Lie algebras\/} are defined via an analogous identity. (See \cite{Cartier_11,Manchon_11} for details.)

Restricting the coproduct \eqref{cop} to extractions of single admissible subsets, i.e., linearizing it on both sides, leads to a map $\hat{\Delta}:  \mathcal{C}_f^{\circ} \to \mathcal{C}_f^{\circ} \otimes \mathcal{C}_f^{\circ}$ with
\begin{equation}
\label{PLcop}
	\hat{\Delta}(c^i) = \sum_{J_l \in \mathcal{W}(c^i)} c^i_{n_l} \otimes c_{|J_l}^{n}.
\end{equation}
One can show that \eqref{PLcop} satisfies the right co-pre-Lie relation
\begin{equation}
\label{coPLrel}
	(id - \pi_{23}) \big( (\hat{\Delta} \otimes \id) \circ \hat{\Delta} 
				- (\id \otimes \hat{\Delta}) \circ \hat{\Delta}  \big)(c^i) = 0.
\end{equation}
The map $\pi_{23}$ operates on triple tensor products by exchanging the last two elements, i.e., $\pi_{23}(a \otimes b \otimes c) = a \otimes c \otimes b$.
The combinatorics of this right co-pre-Lie relation is best described by translating it to the dual side of the Hopf algebra of rcts. It will be seen that the corresponding pre-Lie algebra can be described in terms of insertions of rcts into rcts. However, the admissibility of such an insertion depends on the specific decoration of vertices and roots. Therefore, identity \eqref{coPLrel} translates to a right pre-Lie structure on the graded dual $(\mathcal{H}_X^{\circ})^\ast$ of the Hopf algebra $\mathcal{H}_X^{\circ}$ of decorated rcts.

Denote by $z_{{c'}^i}$ the normalized dual of the rct $c^i$ in the graded dual, that is, $\langle z_{{c'}^i}, {c''}^j \rangle=z_{{c'}^i}({c''}^j)=\delta_{{c'}^i, {c''}^j}$, and zero otherwise. The linear map $z_{{c'}^i}$ is an infinitesimal character of $\mathcal{H}_f^{\circ}$, hence a primitive element in the Hopf algebra $(\mathcal{H}_f^{\circ})^\ast$. Therefore, the convolution product
$$
	(z_{{c'}^i} \star z_{{c''}^j})(c^l) 	= m_k \circ (z_{{c'}^i} \otimes z_{{c''}^j}) \circ \Delta(c^l)
							= m_k \circ (z_{{c'}^i} \otimes z_{{c''}^j}) \circ \hat\Delta(c^l) = z_{{c'}^i \lhd {c''}^j}(c^l).
$$
The resulting Lie bracket on the Lie algebra $g=\mathsf{Prim}((\mathcal{H}_f^{\circ})^\ast)$ spanned by the primitive elements $z_{c^i}$ is given by
\begin{eqnarray*}
	(z_{{c'}^i} \star z_{{c''}^j} - z_{{c''}^j} \star z_{{c'}^i})(c^l) &=& \sum_{J \in \mathcal{W}({c}^l)}
								\langle z_{{c'}^i}, c^l_n\rangle \langle z_{{c''}^j}, c^n_{|J}\rangle
											-\sum_{J \in \mathcal{W}(c^l)}
								\langle z_{{c''}^j}, c^l_n\rangle \langle z_{{c'}^i}, c^n_{|J}\rangle\\
								&=& \langle z_{{c'}^i \lhd {c''}^j - {c''}^j \lhd {c'}^i},c^l \rangle.
\end{eqnarray*}
The right pre-Lie product on rcts, ${c'}^i \lhd {c''}^j$, is obtained by inserting ${c''}^j$ into ${c'}^i$. It can be described by writing the rct $c^i$ in terms of the operator $\Theta^{(x_k)}$ for $0\le k \le m$. It is defined on rcts decorated by the alphabet $X$ by inserting a vertex decorated by $x_k$ to the right of the root. Recall that the vertices of $c^i$ are strictly ordered, which implies that the decoration can be written as a word $\eta=x_{k_1} \cdots x_{k_l}=:\tilde{\eta}x_{k_l} = x_{k_1} \bar{\eta} \in X^*$, where $X^*$ denotes the set of words over the alphabet $X$. Then $c^i= \Theta^{(\eta)} (c^i_{\emptyset})=\Theta^{(\tilde{\eta})}\Theta^{(x_{k_l})}(c^i_{\emptyset})=\Theta^{(x_{k_1})}\Theta^{(\bar{\eta})}(c^i_\emptyset)$, where $c^i_{\emptyset}$ is the rct with no internal vertices. Equivalently,
$$
\raisebox{-20pt}{
 \begin{tikzpicture}
 \draw[thick] (0cm,0cm) circle(0.4cm);
                \filldraw[black] (90:0.4cm) (0.08,.48)rectangle(-.08,.34);
                                \draw (75:0.60cm) node[above=-9pt]{$\hspace{-0.2cm}\phantom{l}^{i}$};
	 \filldraw[black] (45:0.4cm) circle(1.7pt);
	 		\draw (30:0.73cm) node[above=-9pt]{${\hspace{-0.2cm}\phantom{l}^{k_1}}$};
	 \filldraw[black] (0:0.4cm) circle(1.7pt);
	 		\draw (-8:0.73cm) node[above=-9pt]{$\hspace{-0.2cm}\phantom{l}^{k_2}$};
	 \filldraw[black] (-45:0.4cm) circle(1.7pt);
	 		\draw (-46:0.74cm) node[above=-9pt]{$\hspace{-0.2cm}\phantom{l}^{k_3}$};
	 \filldraw[black] (-105:0.33cm) circle(0.3pt);
	 \filldraw[black] (-130:0.33cm) circle(0.3pt);
	 \filldraw[black] (-154:0.33cm) circle(0.3pt);
	 \filldraw[black] (-210:0.4cm) circle(1.7pt);
	 		\draw (-215:0.76cm) node[above=-9pt]{$\hspace{0.2cm}\phantom{l}^{k_l}$};
 \end{tikzpicture}}
=
 \Theta^{(\eta)}\Big(\!\! \raisebox{-10pt}{
 \begin{tikzpicture}
 \draw[thick] (0cm,0cm) circle(0.4cm);
   \filldraw[black] (90:0.4cm) (0.08,.48)rectangle(-.08,.34);
                                \draw (75:0.60cm) node[above=-9pt]{$\hspace{-0.2cm}\phantom{l}^{i}$};
  \end{tikzpicture}}\Big)
=
 \Theta^{(\tilde{\eta})}\Big(\!\! \raisebox{-10pt}{
 \begin{tikzpicture}
 \draw[thick] (0cm,0cm) circle(0.4cm);
                \filldraw[black] (90:0.4cm) (0.08,.48)rectangle(-.08,.34);
                                \draw (75:0.60cm) node[above=-9pt]{$\hspace{-0.2cm}\phantom{l}^{i}$};
	 \filldraw[black] (45:0.4cm) circle(1.7pt);
	 		\draw (30:0.7cm) node[above=-9pt]{$\hspace{-0.2cm}\phantom{l}^{k_l}$};
 \end{tikzpicture}}\! \Big)
=
 \Theta^{(x_{k_1})}\Big(\hspace{-0.5cm}
 \raisebox{-20pt}{
 \begin{tikzpicture}
 \draw[thick] (0cm,0cm) circle(0.4cm);
                \filldraw[black] (90:0.4cm) (0.08,.48)rectangle(-.08,.34);
                                \draw (75:0.60cm) node[above=-9pt]{$\hspace{-0.2cm}\phantom{l}^{i}$};
	 \filldraw[black] (0:0.4cm) circle(1.7pt);
	 		\draw (-8:0.73cm) node[above=-9pt]{$\hspace{-0.2cm}\phantom{l}^{k_2}$};
	 \filldraw[black] (-45:0.4cm) circle(1.7pt);
	 		\draw (-46:0.74cm) node[above=-9pt]{$\hspace{-0.2cm}\phantom{l}^{k_3}$};
	 \filldraw[black] (-105:0.33cm) circle(0.3pt);
	 \filldraw[black] (-130:0.33cm) circle(0.3pt);
	 \filldraw[black] (-154:0.33cm) circle(0.3pt);
	 \filldraw[black] (-210:0.4cm) circle(1.7pt);
	 		\draw (-215:0.76cm) node[above=-9pt]{$\hspace{0.2cm}\phantom{l}^{k_l}$};
 \end{tikzpicture}}\!\!\Big)
$$
The right pre-Lie product is given as follows for rcts $c^i  = \Theta^{(x_{k_1})}(\bar{c}^i)= \Theta^{(\eta)}(c^i_\emptyset)$ and $ {c'}^j=\Theta^{(\eta')}(c^j_\emptyset)$ with decorations $\eta, \eta' \in X^*$:
\begin{eqnarray*}
	c^i  \lhd {c'}^j &=& \Theta^{(x_{k_1})}\big(\bar{c}^i  \lhd {c'}^j \big) + \delta_{k_1,j} \Theta^{(x_{0})} \Theta^{(\bar{\eta} \shuffle \eta')}(c^j_\emptyset)
\end{eqnarray*}
and $c^j_\emptyset \lhd {c'}^j =0$. (The definition of $\Theta$ has been linearly extended to $\allpoly$.) This is the multivariable version of the pre-Lie product in \cite{Foissy_13}. Here the {\em shuffle product} of two words is defined by
$$
	(x_i \bar{\eta}) \shuffle (x_j \bar{\xi}) := x_i (\bar{\eta} \shuffle (x_j\bar{\xi})) + x_j((x_i \bar{\eta}) \shuffle \bar{\xi}),
$$
where $x_i,x_j\in X$, $\bar{\eta},\bar{\xi} \in X^\ast$ and with $\eta \shuffle \emptyset = \eta$ \cite{Reutenauer_93}. Hence, the pre-Lie product of two rcts $c^i,{c'}^j$ is the sum over all insertion of ${c'}^j$ at vertices of $c^i$ decorated by $x_j$
$$
	c^i  \lhd {c'}^j = \sum_{v \in V(c^i) \atop \mathsf{Dec}(v)=x_j} c^i   \curvearrowleft   {c'}^j
$$
The insertion itself is defined by changing the decoration of the insertion vertex from $x_{j}$ to $x_{0}$ and by shuffling the word decorating the vertices following the insertion vertex into the word decorating the inner vertices of ${c'}^j$. For example,
$$
\raisebox{-21pt}{
 \begin{tikzpicture}
 \draw[thick] (0cm,0cm) circle(0.4cm);
                \filldraw[black] (90:0.4cm) (0.08,.48)rectangle(-.08,.34);
                                \draw (75:0.60cm) node[above=-9pt]{$\hspace{-0.2cm}\phantom{l}^{i}$};
	 \filldraw[black] (45:0.4cm) circle(1.7pt);
	 		\draw (30:0.73cm) node[above=-9pt]{${\hspace{-0.2cm}\phantom{l}^{k_1}}$};
	 \filldraw[black] (0:0.4cm) circle(1.7pt);
	 		\draw (-8:0.73cm) node[above=-9pt]{$\hspace{-0.2cm}\phantom{l}^{k_2}$};
	 \filldraw[black] (-45:0.4cm) circle(1.7pt);
	 		\draw (-46:0.74cm) node[above=-9pt]{$\hspace{-0.2cm}\phantom{l}^{k_3}$};
 \end{tikzpicture}}
 \lhd
 \raisebox{-25pt}{
 \begin{tikzpicture}
 \draw[thick] (0cm,0cm) circle(0.4cm);
                \filldraw[black] (90:0.4cm) (0.08,.48)rectangle(-.08,.34);
                                \draw (90:0.60cm) node[above=-9pt]{$\hspace{0.2cm}\phantom{l}^{k_1}$};
	 \filldraw[black] (-90:0.4cm) circle(1.7pt);
	 		\draw (-80:0.7cm) node[above=-9pt]{${\hspace{-0.1cm}\phantom{l}^{l}}$};
 \end{tikzpicture}}
=
\raisebox{-25pt}{
 \begin{tikzpicture}
 \draw[thick] (0cm,0cm) circle(0.4cm);
                \filldraw[black] (90:0.4cm) (0.08,.48)rectangle(-.08,.34);
                                \draw (75:0.60cm) node[above=-9pt]{$\hspace{-0.2cm}\phantom{l}^{i}$};
	 \filldraw[black] (45:0.4cm) circle(1.7pt);
	 		\draw (30:0.73cm) node[above=-9pt]{${\hspace{-0.3cm}\phantom{l}^{0}}$};
	 \filldraw[black] (0:0.4cm) circle(1.7pt);
	 		\draw (-8:0.73cm) node[above=-9pt]{$\hspace{-0.3cm}\phantom{l}^{l}$};
	 \filldraw[black] (-45:0.4cm) circle(1.7pt);
	 		\draw (-46:0.74cm) node[above=-9pt]{$\hspace{-0.2cm}\phantom{l}^{k_2}$};
	 \filldraw[black] (-90:0.4cm) circle(1.7pt);
	 		\draw (-90:0.7cm) node[above=-9pt]{$\hspace{-0.2cm}\phantom{l}^{k_3}$};
 \end{tikzpicture}}
+
\raisebox{-25pt}{
 \begin{tikzpicture}
 \draw[thick] (0cm,0cm) circle(0.4cm);
                \filldraw[black] (90:0.4cm) (0.08,.48)rectangle(-.08,.34);
                                \draw (75:0.60cm) node[above=-9pt]{$\hspace{-0.2cm}\phantom{l}^{i}$};
	 \filldraw[black] (45:0.4cm) circle(1.7pt);
	 		\draw (30:0.73cm) node[above=-9pt]{${\hspace{-0.3cm}\phantom{l}^{0}}$};
	 \filldraw[black] (0:0.4cm) circle(1.7pt);
	 		\draw (-8:0.73cm) node[above=-9pt]{$\hspace{-0.2cm}\phantom{l}^{k_2}$};
	 \filldraw[black] (-45:0.4cm) circle(1.7pt);
	 		\draw (-46:0.74cm) node[above=-9pt]{$\hspace{-0.2cm}\phantom{l}^{l}$};
	 \filldraw[black] (-90:0.4cm) circle(1.7pt);
	 		\draw (-90:0.7cm) node[above=-9pt]{$\hspace{-0.2cm}\phantom{l}^{k_3}$};
 \end{tikzpicture}}
+
\raisebox{-25pt}{
 \begin{tikzpicture}
 \draw[thick] (0cm,0cm) circle(0.4cm);
                \filldraw[black] (90:0.4cm) (0.08,.48)rectangle(-.08,.34);
                                \draw (75:0.60cm) node[above=-9pt]{$\hspace{-0.2cm}\phantom{l}^{i}$};
	 \filldraw[black] (45:0.4cm) circle(1.7pt);
	 		\draw (30:0.73cm) node[above=-9pt]{${\hspace{-0.3cm}\phantom{l}^{0}}$};
	 \filldraw[black] (0:0.4cm) circle(1.7pt);
	 		\draw (-8:0.73cm) node[above=-9pt]{$\hspace{-0.2cm}\phantom{l}^{k_2}$};
	 \filldraw[black] (-45:0.4cm) circle(1.7pt);
	 		\draw (-46:0.74cm) node[above=-9pt]{$\hspace{-0.2cm}\phantom{l}^{k_3}$};
	 \filldraw[black] (-90:0.4cm) circle(1.7pt);
	 		\draw (-90:0.7cm) node[above=-9pt]{$\hspace{-0.2cm}\phantom{l}^{l}$};
 \end{tikzpicture}}
+
\raisebox{-25pt}{
 \begin{tikzpicture}
 \draw[thick] (0cm,0cm) circle(0.4cm);
                \filldraw[black] (90:0.4cm) (0.08,.48)rectangle(-.08,.34);
                                \draw (75:0.60cm) node[above=-9pt]{$\hspace{-0.2cm}\phantom{l}^{i}$};
	 \filldraw[black] (45:0.4cm) circle(1.7pt);
	 		\draw (30:0.73cm) node[above=-9pt]{${\hspace{-0.2cm}\phantom{l}^{k_1}}$};
	 \filldraw[black] (0:0.4cm) circle(1.7pt);
	 		\draw (-8:0.73cm) node[above=-9pt]{$\hspace{-0.2cm}\phantom{l}^{k_2}$};
	 \filldraw[black] (-45:0.4cm) circle(1.7pt);
	 		\draw (-46:0.74cm) node[above=-9pt]{$\hspace{-0.2cm}\phantom{l}^{0}$};
	 \filldraw[black] (-90:0.4cm) circle(1.7pt);
	 		\draw (-90:0.7cm) node[above=-9pt]{$\hspace{-0.2cm}\phantom{l}^{l}$};
 \end{tikzpicture}}
$$


\section{Fliess operators and their output feedback interconnection}
\label{sect:Fliess-operators}

The goal of this section is to show that the Hopf algebra $\mathcal{H}_f^\circ$ has a concrete application in the theory of nonlinear control systems. In particular, when two input-output systems represented in terms of Chen-Fliess series are interconnected to form a feedback system, the coordinate maps of the underlying group form an Hopf algebra which is isomorphic to the Hopf algebra $\mathcal{H}_f^\circ$.

\smallskip

Consider an alphabet $X=\{ x_0, x_1, \ldots, x_m\}$ and any finite sequence of letters from $X$, $\eta=x_{i_1}\cdots x_{i_k}$, called words. The set of all words including the empty word, $\emptyset$, is designated by $X^\ast$. It forms a monoid under catenation. Any mapping $\gsc : X^\ast\rightarrow \re^\ell$ is called a {\em formal power series}. The value of $\gsc$ at $\eta \in X^\ast$ is written as $(\gsc,\eta)$ and called the {\em coefficient} of $\eta$ in $\gsc$. Typically, $\gsc$ is represented as the formal sum $\gsc=\sum_{\eta \in X^\ast}(\gsc,\eta)\eta.$ The collection of all formal power series over $X$ is denoted by $\allseriesell$. It forms an associative $\re$-algebra under the catenation product and an associative and commutative $\re$-algebra under the shuffle product \cite{Fliess_81,Reutenauer_93}.

One can formally associate with any series $\gsc \in \allseriesell$ a causal $m$-input, $\ell$-output operator, $F_\gsc$, in the following manner. Let $p \ge 1$ and $t_0 < t_1$ be given. For a Lebesgue measurable function $u: [t_0,t_1] \rightarrow\re^m$, define $\norm{u}_{p}=\max\{\norm{u_i}_{p}: \ 1\le i\le m\}$, where $\norm{u_i}_{p}$ is the usual $L_{p}$-norm for a measurable real-valued function, $u_i$, defined on $[t_0,t_1]$.  Let $L^m_{p}[t_0,t_1]$ denote the set of all measurable functions defined on $[t_0,t_1]$ having a finite $\norm{\cdot}_{p}$ norm and $B_{p}^m(R)[t_0,t_1]:=\{u\in L_{p}^m[t_0,t_1]:\norm{u}_{p}\leq R\}$. Assume $C[t_0,t_1]$ is the subset of continuous functions in $L_{1}^m[t_0,t_1]$. Define inductively for each $\eta \in X^{\ast}$ the map $E_\eta: L_1^m[t_0, t_1]\rightarrow C[t_0, t_1]$ by setting $E_\emptyset[u]=1$ and letting
\[
	E_{x_i\bar{\eta}}[u](t,t_0) = \int_{t_0}^t u_{i}(\tau)E_{\bar{\eta}}[u](\tau,t_0)\,d\tau,
\]
where $x_i \in X$, $\bar{\eta} \in X^{\ast}$, and $u_0=1$. The input-output operator corresponding to $\gsc$ is the {\em Chen-Fliess series} or {\em Fliess operator} \cite{Fliess_81,Fliess_83}
\begeq
	F_\gsc[u](t) = \sum_{\eta \in X^{\ast}} (\gsc,\eta)\,E_\eta[u](t,t_0). \label{eq:Fliess-operator-defined}
\endeq

Given Fliess operators $F_\gsc$ and $F_\gsd$, where $\gsc, \gsd \in \allseriesell$, the parallel and product connections satisfy $F_\gsc+F_\gsd=F_{\gsc+\gsd}$ and $F_\gsc F_\gsd = F_{\gsc \shuffle \gsd}$, respectively \cite{Fliess_81}. When Fliess operators $F_\gsc$ and $F_\gsd$ with $\gsc \in \allseriesell$ and $\gsd \in \allseriesm$ are interconnected in a cascade fashion, the composite system $F_\gsc \circ F_\gsd$ has the Fliess operator representation $F_{\gsc \circ \gsd}$, where the {\em composition product} of $\gsc$ and $\gsd$ is given by
\begdi
	\gsc \circ \gsd = \sum_{\eta \in X^\ast} (\gsc,\eta)\,\compAH_\gsd(\eta)(1)
\enddi
\cite{Ferfera_79,Ferfera_80}. Here $\compAH_\gsd$ is the continuous (in the ultrametric sense) algebra homomorphism from $\allseries$ to $\Endallseries$ uniquely specified by $\compAH_\gsd(x_i \eta)=\compAH_\gsd(x_i)\circ \compAH_\gsd(\eta)$ with
\begdi \label{eq:psi-d-on-words}
	\compAH_\gsd(x_i)(\gse)=x_0(\gsd_i \shuffle \gse),
\enddi
$i=0,1,\ldots,m$ for any $\gse \in \allseries$, and where $\gsd_i$ is the $i$-th component series of $\gsd$, $\gsd_0:=1$, and $\compAH_\gsd(\emptyset)$ is the identity map on $\allseries$. This composition product is associative and $\re$-linear in its left argument.
In the event that two Fliess operators are interconnected to form a feedback system as shown in Figure~\ref{fig:feedback-with-v}, it was shown in \cite{Gray-Wang_08} that there always exists a unique generating series $\gsc @ \gsd$ such that $y=F_{\gsc @ \gsd}[u]$ whenever $\gsc, \gsd \in \allseriesm$. This so called {\em feedback product} of $\gsc$ and $\gsd$ can be viewed as the unique fixed point of a contractive iterated map on a complete ultrametric space, but this approach provides no machinery for computing the product.

Hopf algebraic tools are now employed to compute the feedback product \cite{Gray-Duffauc_Espinosa_SCL11,Gray-Duffauc_Espinosa_FdB14,Gray-et-al_SCL14}. Consider the set of operators
$
	\Fliessdelta:=\{I+F_\gsc : \gsc \in\allseriesm\},
$
where $I$ denotes the identity operator. It is convenient to introduce the symbol $\delta$ as the (fictitious) generating series for the identity map. That is, $F_\delta:=I$ such that $I+F_\gsc :=F_{\delta + \gsc }=F_{\gsc_\delta}$ with $\gsc_\delta := \delta + \gsc$. The set of all such generating series for $\Fliessdelta$ will be denoted by $\allseriesdeltam$. The first theorem describes the multivariable output feedback group which is at the heart of the method. The group product is described in terms of the {\em modified composition product} of $\gsc\in\allseriesell$ and $\gsd\in\allseriesm$,
\begdi
\gsc \modcomp \gsd = \sum_{\eta \in X^\ast} (\gsc,\eta)\, \modcompAH_\gsd(\eta)(1),
\enddi
where $\modcompAH_\gsd$ is the continuous (in the ultrametric sense) algebra homomorphism from $\allseries$ to $\Endallseries$ uniquely specified by $\modcompAH_\gsd(x_i\eta) = \modcompAH_\gsd(x_i) \circ \modcompAH_\gsd(\eta)$ with
\begdi
\label{eq:phi-d-on-words}
	\modcompAH_\gsd(x_i)(\gse)=x_i\gse + x_0(\gsd_i \shuffle \gse),
\enddi
$i=0,1,\ldots,m$ for any $\gse \in\allseries$, and where $\gsd_0:=0$, and $\modcompAH_\gsd(\emptyset)$ is the identity map on $\allseries$. It can be easily shown that for any $x_i \in X$
\begeq
	(x_i \gsc) \modcomp \gsd = x_i( \gsc \modcomp \gsd) + x_0(\gsd_i \shuffle (\gsc \modcomp \gsd)). \label{eq:xic-cmod-d-identity}
\endeq
The following (non-associativity) identity was proved in \cite{Li_04}
\begeq
	(\gsc \modcomp \gsd) \modcomp \gse = \gsc \modcomp (\gsd \modcomp \gse + \gse) \label{eq:cmod-non-associative}
\endeq
for all $\gsc \in \allseriesell$ and $\gsd,\gse \in \allseriesm$. The central idea is that $(\Fliessdelta,\circ,I)$ forms a group of operators under the composition
\begdi
	F_{\gsc_\delta} \circ F_{\gsd_\delta}=(I+F_\gsc) \circ (I+F_\gsd)
	= F_{\gsc_\delta \circ \gsd_\delta},
\enddi
where $\gsc_\delta \circ \gsd_\delta := \delta + \gsd + \gsc \modcomp \gsd=:\delta+\gsc\circledcirc\gsd$.\footnote{The same symbol will be used for composition on $\allseriesm$ and $\allseriesdeltam$. As elements in these two sets have a distinct notation, i.e., $\gsc$ versus $\gsc_\delta$, respectively, it will always be clear which product is at play.}$\,\,$%
Given the uniqueness of generating series of Fliess operators, this assertion is equivalent to the following theorem.

\begth
\label{th:allseriesdeltam-is-group} \cite{Gray-et-al_SCL14}
The triple $\mathcal{G}=(\allseriesdeltam,\circ,\delta)$ is a group.
\endth

An important remark is in order. The reader should be aware that in \cite{Foissy_EJM15} appears another group of a similar type, call it $\hat{\mathcal{G}}$. It is constructed from the direct product of groups and is combinatorially much simpler than $\mathcal{G}$ because the feature
of {\em cross-channel coupling} is highly restricted. Indeed, following the construction in \cite{Foissy_EJM15}, define the canonical projection
$\epsilon_i:\re^m\rightarrow\re^m$ so that
$$
	\epsilon_i[z_1,z_2,\ldots,z_m]\mapsto [0,\ldots,0,z_i,0,\ldots 0].
$$
The group $\hat{\mathcal{G}}_i$ is identified with $\epsilon_i\re^m\langle\langle X\rangle\rangle$, and the corresponding group product is computed using the product
$$
	(\epsilon_i \gsc)\diamond (\epsilon_i \gsd):=(\epsilon_i \gsc)\circledcirc (\epsilon_i \gsd).
$$
Defining $\hat{\mathcal G}$ as the direct product of $\hat{\mathcal{G}}_i$, $i=1,2,\ldots,m$, it is shown in Theorem 4.7 of \cite{Foissy_EJM15} that the induced group product on $\hat{\mathcal G}$ yields the identity
\begin{equation} \label{eq:faux-composition-identity}
	(\epsilon_i \gsc)\diamond (\epsilon_j \gsd)=\epsilon_i \gsc+\epsilon_j \gsd
\end{equation}
when $i \neq j$. That is, the $j$-th input channel represented by the $j$-th component of $\gsd$ does {\em not affect} the $i$-th output channel of the product. But as illustrated in the following example, systems encountered in control theory are rarely so well behaved. In a typical application, such as the one described in \cite{Gray-et-al_SCL14}, the vast majority of terms generated by the group product on ${\mathcal G}$ come precisely from this coupling. Therefore, the product on $\hat{\mathcal G}$ of the same two series will have (orders of magnitude) fewer terms in real problems. 
\begex \label{ex:Foissy-counterexample}
Suppose the alphabet  $X=\{x_0,x_1,x_2\}$ and let
\begdi
\gsc=\left[
\begin{array}{c}
x_2 \\
0
\end{array}
\right],\;\;
\gsd=\left[
\begin{array}{c}
0 \\
x_1
\end{array}
\right]
\enddi
so that the group elements in $\mathcal{G}$ are
\begdi
\gsc_\delta=\left[
\begin{array}{c}
\delta+x_2 \\
\delta
\end{array}
\right],\;\;
\gsd_\delta=\left[
\begin{array}{c}
\delta \\
\delta+x_1
\end{array}
\right].
\enddi
The group product is then 
\begin{equation} \label{eq:counter-example}
\gsc_\delta\circ \gsd_\delta=\delta+\gsc\circledcirc \gsd
=
\left[
\begin{array}{c}
\delta+x_2+x_0x_1 \\
\delta+x_1
\end{array}
\right].
\end{equation}
The corresponding group elements in $\Fliessdelta$ are
\begdi
(I+F_\gsc)[u]=
\left[
\begin{array}{c}
u_1+\int u_2\,d\tau \\
u_2
\end{array}
\right],\;\;
(I+F_\gsd)[u]=
\left[
\begin{array}{c}
u_1 \\
u_2+\int u_1\,d\tau
\end{array}
\right],
\enddi
respectively.
The product on either group corresponds physically to the cascade connection shown in Figure~\ref{fig:groupcascade}.

\begin{figure}[h]
\begin{center}
\includegraphics*[scale=0.7]{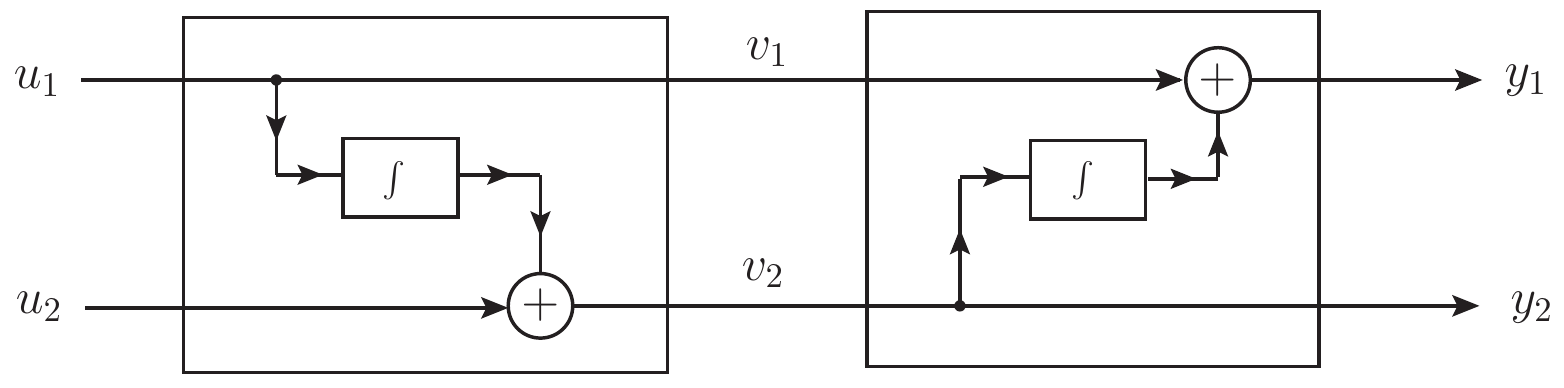}
\end{center}
\vspace*{-0.05in}
\hspace*{-0.1in}$I+F_\gsd$ \hspace{1.42in}$I+F_\gsc$
\caption{Cascade of the two systems in Example~\ref{ex:Foissy-counterexample}}
\label{fig:groupcascade}
\end{figure}

Directly from this diagram it is evident that
\begin{align*}
	y_1&=(I_1+F_{[\gsc\circledcirc \gsd]_1})[u]
		=(I_1+F_{x_2+x_0x_1})[u]=u_1 + u_2\,d\tau_2 +\int \left(\int u_1\,d\tau_1\right)\\
	y_2&=(I_2+F_{[\gsc\circledcirc \gsd]_2})[u]
		=(I_2+F_{x_1})[u]=u_2+\int u_1\,d\tau,
\end{align*}
where $I_i[u]:=u_i$. Note, in particular, that the term $\int \left(\int u_1\,d\tau_1\right)$ corresponding to the word $x_0x_1$ in (\ref{eq:counter-example}) is a direct consequence of the fact that $u_1$ affects $v_2$, and $v_2$ affects $y_1$. On the other hand, the product on $\hat{\mathcal{G}}$ does not capture this coupling term since by \rref{eq:faux-composition-identity}
\begdi
(\epsilon_1 \gsc)\diamond (\epsilon_2 \gsd)=
 \epsilon_1\left[
\begin{array}{c}
x_2 \\
0
\end{array}
\right]\diamond\;
\epsilon_2\left[
\begin{array}{c}
0 \\
x_1
\end{array}
\right]=
\left[
\begin{array}{c}
x_2 \\
0
\end{array}
\right]+
\left[
\begin{array}{c}
0 \\
x_1
\end{array}
\right].
\enddi
Therefore, $\mathcal G$ and $\hat{\mathcal G}$ are distinct group structures defined on
$\re^m\langle\langle X\rangle\rangle$. The latter is simply too restrictive from a control theoretic point of view since it fails to describe the complex interactions that are found in real problems.
\endex


\subsection{Fa\`{a} di Bruno type Hopf algebra}
\label{ssect:FdBHA}

A Fa\`{a} di Bruno type Hopf algebra is now defined for the coordinate maps of the output feedback group. For any $\eta \in X^\ast$, the corresponding coordinate map is the element of the dual space $\allseriesmdual$
\begdi
\label{eq:character-maps}
	a^i_\eta : \allseriesm\rightarrow \re:\gsc \mapsto (\gsc_i,\eta):=a^i_\eta(\gsc),
\enddi
where $i=1,2,\ldots,m$. Let $V$ denote the $\re$-vector space spanned by these maps. Define $\abs{\eta}_{x_i}$ as the number of times the letter $x_i\in X$ appears in the word $\eta$. If the {\em degree} of $a^i_{\eta}$ is defined as $\deg(a^i_{\eta}) = 2\abs{\eta}_{x_0}+\sum_{j=1}^m\abs{\eta}_{x_j}+1$, then $V$ is a connected graded vector space, that is, $V=\bigoplus_{n\geq 0} V^{(n)}$ with
\begdi
	V^{(n)}=\spanset_\re\{a_\eta^i:\deg(a^i_\eta)=n\},\;\;n\geq 1,
\enddi
and $V^{(0)}=\mathbb{R}\mathbf{1}$.
Consider next the free unital commutative $\re$-algebra, $H$, with product
\begdi \label{eq:mu-product}
	m_H(a^i_\eta\otimes a^j_{\xi})= a^i_{\eta}a^j_{\xi}.
\enddi
This product is associative and unital with unit $ \mathbf{1}$. The graduation on $V$ induces a connected graduation on $H$ with $\deg(a^i_\eta a^j_\xi)=\deg(a^i_\eta)+\deg(a^j_\xi)$ and $\deg( \mathbf{1})=0$. Specifically, $H=\bigoplus_{n\geq 0} H^{(n)}$, where $H^{(n)}$ is the set of all elements of degree $n$ and $H^{(0)}=\re \mathbf{1}$.

Three coproducts are now introduced. The first coproduct is $\Delta_\shuffle^j(V^+) \subset V^+\otimes V^+$, where
$V^+ := \bigoplus_{n > 0} V^{(n)}$,
\begin{subequations}
	\label{eq:shuffle-coproduct-induction}
\begin{align}
	\Delta_{\shuffle}^ja^i_{\emptyset}=&a^i_{\emptyset}\otimes a^j_{\emptyset} \\
	\Delta_{\shuffle}^j\circ\theta_k=&(\theta_k\otimes \id + \id \otimes \theta_k)\circ\Delta_{\shuffle}^j,
\end{align}
\end{subequations}
$\id$ is the identity map on $V^+$, and $\theta_k$ denotes the endomorphism on $V^+$ specified by $\theta_k a^i_\eta=a^i_{x_k\eta}$ for $k=0,1,\ldots,m$ and $i,j=1,2,\ldots,m$.
The first few terms of $\Delta_{\shuffle}^j$ are:
\allowdisplaybreaks{
\begin{align*}
	\Delta_{\shuffle}^ja^i_{\emptyset}	 &=a^i_{\emptyset}\otimes a^j_{\emptyset} \\
	\Delta_{\shuffle}^ja^i_{x_{i_1}}		 &=a^i_{x_{i_1}}\otimes a^j_{\emptyset}+a^i_{\emptyset}\otimes a^j_{x_{i_1}} \\
	\Delta_{\shuffle}^ja^i_{x_{i_2}x_{i_1}} &=a^i_{x_{i_2}x_{i_1}}\otimes a^j_{\emptyset}+a^i_{x_{i_2}}\otimes a^j_{x_{i_1}}
									+a^i_{x_{i_1}}\otimes a^j_{x_{i_2}}+a^i_{\emptyset}\otimes a^j_{x_{i_2}x_{i_1}} \\
	\Delta_{\shuffle}^ja^i_{x_{i_3}x_{i_2}x_{i_1}} &= a^i_{x_{i_3}x_{i_2}x_{i_1}}\otimes a^j_{\emptyset}
										+a^i_{x_{i_3}x_{i_2}}\otimes a^j_{x_{i_1}}+
									     a^i_{x_{i_3}x_{i_1}}\otimes a^j_{x_{i_2}}
									    	+a^i_{x_{i_3}}\otimes a^j_{x_{i_2}x_{i_1}}+\\
									    &\hspace*{0.18in} a^i_{x_{i_2}x_{i_1}}\otimes a^j_{x_{i_3}}
									    	+a^i_{x_{i_2}}\otimes a^j_{x_{i_3}x_{i_1}}+
									    a^i_{x_{i_1}}\otimes a^j_{x_{i_3}x_{i_2}}
									    	+a^i_{\emptyset}\otimes a^j_{x_{i_3}x_{i_2}x_{i_1}}.
\end{align*}}

The second coproduct is $\tilde{\Delta}(H)\subset V\otimes H$, which is induced by the identity
\begeq
	\tilde{\Delta}a^i_{\eta}(\gsc,\gsd)=a^i_{\eta}(\gsc \modcomp \gsd)
	=(\gsc_i \modcomp \gsd,\eta),\;\; \eta\in X^\ast.
	\label{eq:tilde-delta-identity}
\endeq
This coproduct can be computed recursively as described in the following lemma, which is the multivariable version of Proposition 3 in \cite{Foissy_13}.

\begle \cite{Gray-et-al_SCL14} \label{le:tilde-delta-inductions}
The following identities hold: $\tilde{\Delta}(\mathbf{1})=0$, and
\begdes
	\item[\hspace*{0.15in}(1)] $\tilde{\Delta}a^i_\emptyset=a^i_{\emptyset}\otimes \mathbf{1}$
	\item[\hspace*{0.15in}(2)]
		$\tilde{\Delta}\circ \theta_i = (\theta_i\otimes \id) \circ \tilde{\Delta}$
	\item[\hspace*{0.15in}(3)]
		$\tilde{\Delta} \circ \theta_0 = (\theta_0 \otimes \id) \circ \tilde{\Delta}+
			(\theta_i \otimes \mu) \circ (\tilde{\Delta} \otimes \id) \circ \Delta_{\shuffle}^i$,
\enddes
$i=1,2,\ldots,m$, where $\id$ denotes the identity map on $H$.
\endle

The third coproduct is defined as $\Delta a^i_\eta := \tilde{\Delta}a^i_\eta + \mathbf{1} \otimes a^i_\eta$. The coassociativity of $\Delta$ follows from the associativity of the product $\gsc\circledcirc \gsd$ and the identity $\Delta a_\eta^i(\gsc,\gsd)=((\gsc \circledcirc \gsd)_i,\eta)$. This coproduct is used in the following central result.

\begth \label{thm:FdBHA}
\cite{Foissy_13,Gray-et-al_SCL14}
	$(H,m_H,\Delta)$ is a connected graded commutative unital Hopf algebra.
\endth

The next theorem provides a {\em fully} recursive algorithm to compute the antipode for the feedback group.

\begth \cite{Gray-et-al_MTNS14,Gray-et-al_SCL14}
The antipode, $S$, of any coordinate map $a_\eta^i \in V^+$ in the output feedback group can be computed by the following algorithm:
\begdes
	\item[\hspace*{0.15in}1] Recursively compute $\Delta_{\shuffle}$ via \rref{eq:shuffle-coproduct-induction}.
	\item[\hspace*{0.15in}2] Recursively compute $\tilde{\Delta}$ via Lemma~\ref{le:tilde-delta-inductions}.
	\item[\hspace*{0.15in}3] Recursively compute $S$  by
\begin{eqnarray}
	S(a_\eta^i)	\!\!&=&\!\!  - m_H(S \otimes (\id - e \circ \varepsilon))\Delta(a_\eta^i)
						= - a_\eta^i - m_H(S \otimes \id) \Delta^\prime(a_\eta^i ) \label{1stAntipode}\\
				\!\!&=&\!\!  - m_H( (\id - e \circ \varepsilon)  \otimes S)\Delta(a_\eta^i)
						= - a_\eta^i - m_H(\id \otimes S) \Delta^\prime(a_\eta^i ), \label{2ndAntipode}
\end{eqnarray}
where $\varepsilon: H \to \mathbb{R}$ and $e: \mathbb{R} \to H$ are the counit and unit maps of $H$, respectively. The reduced coproduct is defined as $\Delta^\prime ( a_\eta^i )=\Delta( a_\eta^i ) -  a_\eta^i \otimes \mathbf{1} - \mathbf{1} \otimes  a_\eta^i = \tilde{\Delta}( a_\eta^i ) -  a_\eta^i \otimes \mathbf{1}$.
\enddes
\endth

Applying the algorithm above gives the antipode of the first few coordinate maps:
\begin{align}
	H^{(1)}&:S(a^i_\emptyset) = - a^i_\emptyset \nonumber \\
	H^{(2)}&:S(a^i_{x_j}) 		= - a^i_{x_j} \nonumber \\
	H^{(3)}&:S(a^i_{x_0}) 		= - a^i_{x_0} + a^i_{x_{n_1}} a^{n_1}_\emptyset \nonumber \\
	H^{(3)}&:S(a^i_{x_jx_k})	= - a^i_{x_jx_k} \nonumber \\
	H^{(4)}&:S(a^i_{x_0x_j}) 	= - a^i_{x_0x_j} + a^i_{x_{n_1}}a^{n_1}_{x_j} + a^i_{x_{n_1} x_j} a^{n_1}_\emptyset \nonumber \\
	H^{(4)}&:S(a^i_{x_jx_0}) 	= - a^i_{x_jx_0} + a^i_{x_jx_{n_1}}a^{n_1}_{\emptyset} \nonumber \\
	H^{(4)}&:S(a^i_{x_j x_k x_l}) 	= - a^i_{x_j x_k x_l} \nonumber \\
	H^{(5)}&:S(a^i_{x_0^2}) 	= - a^i_{x_0^2} - S(a^i_{x_{n_1}})a^{n_1}_{x_0} - S(a^i_{x_{n_1} x_0})a^{n_1}_\emptyset -
							S(a^i_{x_0x_{n_2}})a^{n_2}_\emptyset-  \label{eq:right-S-example} \\
        &\hspace*{0.77in} S(a^i_{x_{n_1} x_{n_2}})a^{n_1}_\emptyset a^{n_2}_\emptyset \nonumber \\
		&\hspace*{0.55in}	= - a^i_{x_0^2} - (-a^i_{x_{n_1}})a^{n_1}_{x_0} - (-a^i_{x_{n_1} x_0}
							+ \cancel{a^i_{x_{n_1} x_{n_2}} a^{n_2}_\emptyset}) a^{n_1}_\emptyset - \nonumber \\
		&\hspace*{0.77in} 		(-a^i_{x_0 x_{n_2}} + a^i_{x_{n_1}} a^{n_1}_{x_{n_2}}
							+ a^i_{x_{n_1} x_{n_2}} a^{n_1}_{\emptyset})a^{n_2}_\emptyset- 	
							\cancel{(-a^i_{x_{n_1} x_{n_2}})a^{n_1}_\emptyset a^{n_2}_\emptyset} \nonumber \\
		&\hspace*{0.55in}	= - a^i_{x_0^2} + a^i_{x_{n_1}}a^{n_1}_{x_0} + a^i_{x_{n_1} x_0}a^{n_1}_\emptyset
							+ a^i_{x_0x_{n_2}}a^{n_2}_\emptyset-
							a^i_{x_{n_1}}a^{n_1}_{x_{n_2}} a^{n_2}_{\emptyset}- \nonumber \\
		&\hspace*{0.77in}	a^i_{x_{n_1} x_{n_2}}a^{n_1}_{\emptyset}a^{n_2}_{\emptyset}, \nonumber
\end{align}
where $i,j,k,l=1,2,\ldots m$. Recall that Einstein's summation convention is in place, so these expressions have implied summations.
The explicit calculation for $S(a^i_{x_0^2})$ is shown above to demonstrate that this approach, not unexpectedly, involves some inter-term cancellations. This is consistent with what is known about the classical Fa\`{a} di Bruno Hopf algebra and Zimmermann's forest formula. The latter can provide a cancellation free approach to computing the antipode \cite{Einziger_10,Haiman-Schmitc_89}. Such a forest formula without cancellations will be given shortly using the results from subsection \ref{ssect:zimmermann}.
But first the link between the Hopf algebra of decorated rct and the Fa\`{a} di Bruno type Hopf algebra for the output feedback group is presented in the following theorem.

\begth \label{thm:HAiso}
	The Hopf algebras $\mathcal{H}_X^\circ$ and $H$ are isomorphic.
\endth

\begin{proof}
Define the bijection $\varphi: \mathcal{C}_X^\circ \to H$, $\varphi(c^i):=a^i_{\eta}$, where the word $\eta$ consist of the letters decorating the rct $c^i$, and the order is implied by the order on $c^i$. The map $\varphi$ is extended multiplicatively to an algebra isomorphism.
The main claim is that the coproduct (\ref{coprod}) satisfies the set of recursions in Lemma~\ref{le:tilde-delta-inductions}. Note that this  also implies coassociativity of (\ref{coprod}). To this end, two operators $\Theta_{\begin{tikzpicture}  \filldraw[fill=white, draw=black,thick] (0,0) circle(1.3pt); \end{tikzpicture}}$ and $\Theta_{\begin{tikzpicture}  \filldraw[black] (0,0) circle(1.3pt); \end{tikzpicture}}$ are defined on $\mathcal{C}_f^{\circ}$. Both insert vertices to the right of the root
$$
	\Theta_{\begin{tikzpicture}  \filldraw[fill=white, draw=black,thick] (0,0) circle(1.3pt); \end{tikzpicture}}\Big(\!\!\raisebox{-10pt}{\et}\Big)=\raisebox{-10pt}{\rt}
	\qquad
	\Theta_{\begin{tikzpicture}  \filldraw[fill=white, draw=black,thick] (0,0) circle(1.3pt); \end{tikzpicture}}
	\Theta_{\begin{tikzpicture}  \filldraw[black] (0,0) circle(1.3pt); \end{tikzpicture}}\Big(\!\!\raisebox{-10pt}{\et}\Big)=
	\Theta_{\begin{tikzpicture}  \filldraw[fill=white, draw=black,thick] (0,0) circle(1.3pt); \end{tikzpicture}}\Big(\!\!\raisebox{-10pt}{\bt}\Big)=\raisebox{-10pt}{\rbt}
	\qquad
	\Theta_{\begin{tikzpicture}  \filldraw[fill=white, draw=black,thick] (0,0) circle(1.3pt); \end{tikzpicture}}\Big(\!\!\raisebox{-10pt}{\brbt}\Big)=\raisebox{-10pt}{\rbrbt}
$$
The map $\Theta^{(n_1)}_{\begin{tikzpicture}  \filldraw[black] (0,0) circle(1.3pt); \end{tikzpicture}}$ is defined by inserting a black vertex indexed by $n_1$ to right of the root
$$
	\Theta^{(n_1)}_{\begin{tikzpicture}  \filldraw[black] (0,0) circle(1.3pt); \end{tikzpicture}}\Big(\!\!\raisebox{-10pt}{\et}\Big)=\raisebox{-20pt}{\btn}
	\qquad
	\Theta^{(n_1)}_{\begin{tikzpicture}  \filldraw[black] (0,0) circle(1.3pt); \end{tikzpicture}}\Big(\!\!\raisebox{-10pt}{\rt}\Big)=\raisebox{-18pt}{\bnrt}
	\qquad
	\Theta^{(n_1)}_{\begin{tikzpicture}  \filldraw[black] (0,0) circle(1.3pt); \end{tikzpicture}}
	\Theta^{(n_2)}_{\begin{tikzpicture}  \filldraw[black] (0,0) circle(1.3pt); \end{tikzpicture}}\Big(\!\!\raisebox{-10pt}{\et}\Big)=\raisebox{-18pt}{\bkbnt}
$$
First, assume $c^i = \Theta_{\begin{tikzpicture}  \filldraw[black] (0,0) circle(1.3pt); \end{tikzpicture}}(\tilde{c}^i) \in \mathcal{H}_f^\circ$. Then it follows that
$$
	\Delta(c^i)=\Delta(\Theta_{\begin{tikzpicture}  \filldraw[black] (0,0) circle(1.3pt); \end{tikzpicture}}(\tilde{c}^i))
			= ( \Theta_{\begin{tikzpicture}  \filldraw[black] (0,0) circle(1.3pt); \end{tikzpicture}}\otimes \id) \Delta(\tilde{c}^i),
$$
since none of the admissible extractions include the minimal element in $V(c)$, that is, the first vertex to the right of the root. Next, it is assumed that $c^i = \Theta_{\begin{tikzpicture}  \filldraw[fill=white, draw=black,thick] (0,0) circle(1.3pt); \end{tikzpicture}}(\tilde{c}^i)$. Then
\allowdisplaybreaks{
\begin{eqnarray}
	\Delta(c^i)\!\!&=&\!\!  \sum_{\mathcal{J} \in {\mathcal{W}}(c^i)
					\atop \mathcal{J} =\{J_{l_1},\ldots, J_{l_k}\}}
				c^i_{n_{l_1}\cdots n_{l_k}}
				\otimes c_{|J_{l_1}}^{n_{l_1}} \cdots c_{|J_{l_k}}^{n_{l_k}} \nonumber \\
			\!\!&=&\!\! 	\sum_{\mathcal{J} \in {\mathcal{W}}(c^i)
					\atop {\mathcal{J} =\{J_{l_1},\ldots, J_{l_k}\}\atop l_1 > 1}}\hspace{-0.4cm}
				c^i_{n_{l_1}\cdots n_{l_k}}
				\otimes c_{|J_{l_1}}^{n_{l_1}} \cdots c_{|J_{l_k}}^{n_{l_k}} +
				\sum_{{\mathcal{J} \in {\mathcal{W}}(c^i)
					\atop \mathcal{J} =\{J_{l_1},\ldots, J_{l_k}\}} \atop l_1=1}\hspace{-0.4cm}
				c^i_{n_{l_1}\cdots n_{l_k}}
				\otimes c_{|J_{l_1}}^{n_{l_1}} \cdots c_{|J_{l_k}}^{n_{l_k}}. \label{split1}
\end{eqnarray}}%
The first sum on the right-hand side of (\ref{split1}) includes all admissible extractions with admissible subsets that do not have the first white vertex of $c^i= \Theta_{\begin{tikzpicture}  \filldraw[fill=white, draw=black,thick] (0,0) circle(1.3pt); \end{tikzpicture}}(\tilde{c}^i)$ as the minimal element. The admissible extractions in the second sum on the right-hand side of (\ref{split1}) all include an admissible subset with the first white vertex of $c^i$ as the minimal element. Therefore, the coproduct \eqref{split1} can be written as
\allowdisplaybreaks{
\begin{eqnarray*}
	\Delta(c^i) &=& \sum_{\mathcal{J} \in {\mathcal{W}}(c^i)
				\atop {\mathcal{J} =\{J_{l_1},\ldots, J_{l_k}\}\atop l_1 > 1}}\hspace{-0.4cm}
			\Theta_{\begin{tikzpicture}  \filldraw[fill=white, draw=black,thick] (0,0) circle(1.3pt); \end{tikzpicture}}(\tilde{c}^i_{n_{l_1}\cdots n_{l_k}})
			\otimes c_{|J_{l_1}}^{n_{l_1}} \cdots c_{|J_{l_k}}^{n_{l_k}} +\\
	& &\qquad\		\sum_{\mathcal{J} \in {\mathcal{W}}(c^i)
				\atop \mathcal{J} =\{J_{1}, J_{l_2},\ldots, J_{l_k}\}}\hspace{-0.4cm}
			\Theta^{(n_{1})}_{\begin{tikzpicture}  \filldraw[black] (0,0) circle(1.3pt); \end{tikzpicture}}\big((c^i / J_{1})_{n_{l_2}\cdots n_{l_k}}\big)
			\otimes c_{|J_{1}}^{n_{1}} c_{|J_{l_2}}^{n_{l_2}}\cdots c_{|J_{l_k}}^{n_{l_k}}.
\end{eqnarray*}}%
The first sum can be simplified to
$$
	\sum_{\mathcal{J} \in {\mathcal{W}}(c^i)
				\atop {\mathcal{J} =\{J_{l_1},\ldots, J_{l_k}\}\atop l_1 > 1}}\hspace{-0.4cm}
			\Theta_{\begin{tikzpicture}  \filldraw[fill=white, draw=black,thick] (0,0) circle(1.3pt); \end{tikzpicture}}(\tilde{c}^i_{n_{l_1}\cdots n_{l_k}})
			\otimes c_{|J_{l_1}}^{n_{l_1}} \cdots c_{|J_{l_k}}^{n_{l_k}}
	=( \Theta_{\begin{tikzpicture}  \filldraw[fill=white, draw=black,thick] (0,0) circle(1.3pt); \end{tikzpicture}} \otimes \id )	\hspace{-0.4cm}
		\sum_{\mathcal{J} \in {\mathcal{W}}(\tilde{c}^i)
				\atop \mathcal{J} =\{J_{l_1},\ldots, J_{l_k}\}}
			\tilde{c}^i_{n_{l_1}\cdots n_{l_k}}
			\otimes \tilde{c}_{|J_{l_1}}^{n_{l_1}} \cdots \tilde{c}_{|J_{l_k}}^{n_{l_k}}.
$$
The second sum yields
\allowdisplaybreaks{
\begin{eqnarray*}
\lefteqn{\sum_{\mathcal{J} \in {\mathcal{W}}(c^i)
				\atop \mathcal{J} =\{J_{1}, J_{l_2},\ldots, J_{l_k}\}}\hspace{-0.4cm}
			\Theta^{(n_{1})}_{\begin{tikzpicture}  \filldraw[black] (0,0) circle(1.3pt); \end{tikzpicture}}\big((c^i / J_{1})_{n_{l_2}\cdots n_{l_k}}\big)
			\otimes c_{|J_{1}}^{n_{1}}  c_{|J_{l_2}}^{n_{l_2}} \cdots c_{|J_{l_k}}^{n_{l_k}}}\\
&=& (\Theta^{(n_{1})}_{\begin{tikzpicture}  \filldraw[black] (0,0) circle(1.3pt); \end{tikzpicture}} \otimes m_{\mathcal{H}^\circ_\mathcal{C}})
	\sum_{\mathcal{J} \in {\mathcal{W}}(c^i)
				\atop \mathcal{J} =\{J_{1}, J_{l_2},\ldots, J_{l_k}\}}\hspace{-0.4cm}
			(c^i / J_{1})_{n_{l_2}\cdots n_{l_k}}
			\otimes c_{|J_{1}}^{n_{1}} \otimes c_{|J_{l_2}}^{n_{l_2}}\cdots c_{|J_{l_k}}^{n_{l_k}}\\
&=& (\Theta^{(n_{1})}_{\begin{tikzpicture}  \filldraw[black] (0,0) circle(1.3pt); \end{tikzpicture}} \otimes m_{\mathcal{H}^\circ_\mathcal{C}})
	\sum_{J_1 \in {\mathcal{W}}(c^i)}\sum_{\mathcal{J} \in {\mathcal{W}}(c^i / J_{1})
				\atop \mathcal{J} =\{J_{s_1},\ldots, J_{s_u}\}}\hspace{-0.4cm}
			(c^i / J_{1})_{n_{s_1}\cdots n_{s_u}}
			 \otimes (c^i / J_{1})_{|J_{s_1}}^{n_{s_1}}\cdots (c^i / J_{1})_{|J_{s_u}}^{n_{s_u}}\otimes c_{|J_{1}}^{n_{1}}\\
&=&  (\Theta^{(n_{1})}_{\begin{tikzpicture}  \filldraw[black] (0,0) circle(1.3pt); \end{tikzpicture}} \otimes m_{\mathcal{H}^\circ_\mathcal{C}})
	(\Delta \otimes \id) \sum_{J_1 \in {\mathcal{W}}(c^i)}
			(c^i / J_{1}) \otimes c_{|J_{1}}^{n_{1}}			
\end{eqnarray*}}
such that
\allowdisplaybreaks{
\begin{eqnarray*}
	\Delta(c^i)&=&( \Theta_{\begin{tikzpicture}  \filldraw[fill=white, draw=black,thick] (0,0) circle(1.3pt); \end{tikzpicture}}\otimes \id) \Delta(\tilde{c}^i) +
				(\Theta^{(n_1)}_{\begin{tikzpicture}  \filldraw[black] (0,0) circle(1.3pt); \end{tikzpicture}} \otimes
				m_{\mathcal{H}^\circ_\mathcal{C}})(\Delta \otimes \id)\sum_{J_1 \in {\mathcal{W}}(c^i)}(c^i / J_1) \otimes c_{|J_1}^{n_1}.
\end{eqnarray*}}%
Recall that $(c^i / J_1)$ is the rct with vertex set $V(c^i / J_1):= V(c) - J_1$. The sub-rct $c_{|J_1}^{n_1}$ is defined by making the minimal element in $J_1$ its root indexed by $n_1$, and the remaining elements of the admissible subset $J_1$ are the internal vertices ordered according to the order in $J_1$. The sum in the second term on the right-hand side is over all admissible subsets that have the first white vertex of $c^i$ as its minimal element. This corresponds to extracting single sub-rcts. In fact, the last sum can be given a name
$$
	\triangle^{n_1}_{\!\shuffle}(c^i) := \sum_{J_1 \in {\mathcal{W}}(c^i)} (c^i / J_1) \otimes c_{|J_1}^{n_1}
	\in \mathcal{C}^\circ_f \otimes \mathcal{C}^\circ_f,
$$
such that
$$
	\Delta(c^i)=( \Theta_{\begin{tikzpicture}  \filldraw[fill=white, draw=black,thick] (0,0) circle(1.3pt); \end{tikzpicture}}\otimes \id) \Delta(\tilde{c}^i) +
				(\Theta^{(n_1)}_{\begin{tikzpicture}  \filldraw[black] (0,0) circle(1.3pt); \end{tikzpicture}} \otimes
				m_{\mathcal{H}^\circ_\mathcal{C}})(\Delta \otimes \id) \triangle^{n_1}_{\!\shuffle}(c^i).
$$
The notation becomes clear when comparing it with the recursive definition of the coproduct in Lemma~\ref{le:tilde-delta-inductions}. Recall that black and white vertices of a rct $c^i \in C^\circ_X$ represent vertices which are decorated by $x_k$, $k \neq 0$, and $x_0$, respectively. The rct $c^i \in C^\circ_X$ with vertices decorated by $x_{k_1} \cdots x_{k_l}$ of weight $l+1$ and degree $\deg(c^i)$ is mapped via $\varphi$ to the coordinate function $a^i_{x_{k_1} \cdots x_{k_l}}$ of degree $\deg(a^i_{x_{k_1} \cdots x_{k_l}})$. The map $ \triangle^{n_1}_{\!\shuffle}$ corresponds via $\varphi$ to the classical deshuffling coproduct, i.e., $(\varphi \otimes \varphi)\triangle^{i}_{\!\shuffle}(c^i) = \Delta_{\shuffle}^i(\varphi(c^i))$, such that
$$
	(\varphi \otimes \varphi)\Delta(c^i) = \Delta(\varphi(c^i)).
$$
On the right-hand side the coproduct is that of $(H,m_H,\Delta)$ in Theorem \ref{thm:FdBHA}.
\end{proof}
In light of Theorem \ref{thm:HAiso}, the forest formula presented in Theorem \ref{thm:forestformula} provides a cancellation free formula for the antipode in $H$. For example, recalling \rref{eq:left-S-rct-example} and \rref{eq:right-S-example}, observe:
\allowdisplaybreaks{
\begin{eqnarray*}
	\varphi\Big(
	S \Big(\!\!\raisebox{-10pt}{\rrt}\Big) \Big)&=&
	\varphi\bigg( - \raisebox{-10pt}{\rrt} 	
	+ \raisebox{-21pt}{\btn}   \raisebox{-10pt}{\rtn}\;
	+ \raisebox{-18pt}{\bnrt}\!\!  \raisebox{-9pt}{\etn}
	+\!\! \raisebox{-19pt}{\rbnt}  \raisebox{-9pt}{\etk}\\
	& &
		- \raisebox{-19pt}{\btn}  \raisebox{-19pt}{\bktn}  \raisebox{-9pt}{\etk}
		-\!\!  \raisebox{-19pt}{\bkbnt} \!\! \raisebox{-9pt}{\etn}  \raisebox{-9pt}{\etk}\bigg) \\
	&=&-a^i_{x_0^2}
		+a^i_{x_{n_1}}a^{n_1}_{x_0}
		+ a^i_{x_{n_1} x_0}a^{n_1}_\emptyset
		+ a^i_{x_0x_{n_2}}a^{n_2}_\emptyset \\
	& &	
		- a^i_{x_{n_1}}a^{n_1}_{x_{n_2}} a^{n_2}_{\emptyset}
		-a^i_{x_{n_1} x_{n_2}}a^{n_1}_{\emptyset}a^{n_2}_{\emptyset}\\	
	&=& S(a^i_{x_0^2}) = S\Big(\varphi \Big(\!\!\raisebox{-10pt}{\rrt}\Big)\Big)
	\end{eqnarray*}}


\subsection{Composition of Chen--Fliess series and the convolution product}
\label{ssect:ChFcomp}

\begin{table}[t]
\caption{Composition products of formal power series induced by Fliess operator compositions}
\label{tbl:composition-products}
\begin{tabular}{|c|c|c|c|} \hline
operator composition & generating series & coproduct & remarks \\ \hline
$F_\gsc\circ F_\gsd$ & $\gsc \circ \gsd$ & $\;\Delta^\prime$ & associative \\
$F_\gsc\circ (I+F_\gsd)$ & $\gsc \modcomp \gsd$ & $\tilde{\Delta}$ & nonassociative \\
$(I+F_\gsc)\circ F_\gsd$ & $\gsc \;\hat{\circ}\; \gsd$ & $\hat{\Delta}$ & nonassociative \\
$(I+F_\gsc)\circ (I+F_\gsd)$ & $\gsc \circledcirc \gsd$ & $\Delta$ & group product \\ \hline
\end{tabular}
\end{table}

\begin{figure}[t]
\begin{tikzpicture}
\draw[->] (0.5,4) -- (3.5,4);
\draw[->] (0,3.5) -- (0,0.5);
\draw[->] (4,3.5) -- (4,0.5);
\draw[->] (0.5,0) -- (3.5,0);
\node at (0,0) {$\hat{\Delta}$,$\hat{\circ}$};
\node at (0,4) {$\Delta^\prime$,$\circ$};
\node at (4,0) {$\Delta$, $\circledcirc$};
\node at (4,4) {$\tilde{\Delta}$,$\modcomp$};
\node at (-0.9,2) {$+({\bf 1}\otimes a_\eta^i)$};
\node at (4.9,2) {$+({\bf 1}\otimes a_\eta^i)$};
\node at (2,-0.4) {$+(a_\eta^i\otimes {\bf 1})$};
\node at (2,4.4) {$+(a_\eta^i\otimes {\bf 1})$};
\end{tikzpicture}
\caption{Relationship between the coproducts induced by series composition}
\label{fig:composition-products-commutative-diagram}
\end{figure}

In this subsection, the various composition products appearing in the previous sections are examined in more depth to reveal their relationship to each other as well as to the convolution product on the dual of $H$. The situation is summarized in Table~\ref{tbl:composition-products}.  The product $\gsc\hat{\circ}\gsd$ is included here for completeness, but its role in the theory is relatively minor. Each composition induces a corresponding coproduct via identities like \rref{eq:tilde-delta-identity}. These coproduct are easily shown to be related by adding specific terms of the primitive $a_\eta^i\otimes {\bf 1}+{ \bf 1}\otimes a_\eta^i$ as shown in Figure~\ref{fig:composition-products-commutative-diagram}.

With each $\gsc_\delta\in\allseriesdeltam$, one can identify an element $\Phi_\gsc\in L(H,\re)$ where by
\begdi
	\Phi_\gsc:a^i_\eta\mapsto a^i_\eta(\gsc)=(\gsc_i,\eta),
\enddi
$\Phi_\gsc(\mbf{1})=1$, and
\begdi
	\Phi_\gsc(a^i_\eta a^j_\xi)=a^i_\eta(\gsc) a^j_\xi(\gsc)=\Phi_\gsc(a^i_\eta)\Phi_\gsc(a^j_\xi).
\enddi
Given that $L(H,\re)\subset \allseries^{\ast\ast}$, there is clearly a canonical embedding of $L(H,\re)$ in $\allseries$ so that
characters of $H$ can be identified with series in $\allseries$.
In addition,
\begin{align*}
	(\Phi_\gsc \star \Phi_\gsd)(a^i_\eta)&=m_k\circ(\Phi_\gsc\otimes \Phi_\gsd)\circ\Delta a^i_\eta \\
		&=\sum \Phi_\gsc((\Delta a^i_\eta)_{(1)})\Phi_\gsd((\Delta a^i_\eta)_{(2)}) \\
		&=\sum (\Delta a^i_\eta)_{(1)}(\gsc) (\Delta a^i_\eta)_{(2)}(\gsd) \\
		&= \Delta a^i_\eta(\gsc,\gsd) \\
		&= a_\eta^i(\gsc\circledcirc \gsd).
\end{align*}
So in some sense an identification can also be made between convolution of characters and composition of series. In fact, this can be done for any of the four series compositions described above. This is not completely surprising from a system theory point of view given the fact that the composition of two linear time-invariant systems with respective impulse responses $h_\gsc$ and $h_\gsd$ has the impulse response $h_\gsc\star h_\gsd$, where in this case `$\star$' denotes function convolution. In the analytic case, function convolution is equivalent to the composition of the corresponding generating series, i.e., $h_\gsc\circ h_\gsd = h_{\gsc\circ \gsd}$, which in turn reduces to a convolution sum over a one letter alphabet \cite{Gray-Li_05}. Put another way, the Hopf convolution product induced by $\gsc\circ \gsd$ can be viewed as a nonlinear generalization of this classic convolution theorem in linear system theory. Continuing with the above development, since the composition inverse of $\gsc_\delta$ in Theorem~\ref{th:allseriesdeltam-is-group} is $\gsc_\delta^{-1}=\delta+\gsc^{-1}$ with $a_\eta^i(\gsc^{-1})=(S (a^i_\eta))(\gsc)$, it follows that
$
	\Phi_\gsc^{\star -1}=\Phi_\gsc \circ S=\Phi_{\gsc^{-1}}.
$
Finally, the infinitesimal characters associated with $H$ can be identified with elements $\xi_\gsc\in L(H,\re)$, where $\xi_\gsc(a_\eta^i)=\sum_{k\geq 1} (-1)^{(k-1)}(\gsc_i,\eta)^k/k$ for any $\gsc\in\allseries$ and $\xi_\gsc(\mbf{1})=0$.


\end{document}